\documentclass[11pt,reqno]{amsart}

\usepackage[utf8]{inputenc}
\usepackage{lmodern}
\usepackage{amsmath,amssymb,amsthm,latexsym,cite,cancel}
\usepackage[small]{caption}
\usepackage{graphicx,wasysym,overpic,tikz,color}
\usepackage{subfigure}
\usepackage{cite}
\usepackage[colorlinks=true,urlcolor=blue,
citecolor=red,linkcolor=blue,linktocpage,pdfpagelabels,
bookmarksnumbered,bookmarksopen]{hyperref}
\usepackage[english]{babel}
\usepackage{units}
\usepackage[utf8]{inputenc}
\usepackage{enumitem}
\usepackage[left=2.1cm,right=2.1cm,top=2.71cm,bottom=2.71cm]{geometry}
\usepackage[hyperpageref]{backref}
\usepackage{float}
\usepackage[T1]{fontenc}

\usepackage{CJKutf8}

\usepackage[colorinlistoftodos]{todonotes}
\newtheorem{theorem}{Theorem}[section]

\newtheorem{lemma}[theorem]{Lemma}
\newtheorem{remark}[theorem]{Remark}

\newcommand{\abs}[1]{\lvert#1\rvert}
\newcommand{\norm}[1]{\lVert#1\rVert}

\newcommand{\blue}[1]{\textcolor{blue}{#1}}

\newcommand{\G}{\mathcal{G}}

\newcommand{\E}{\mathbb{E}}
\newcommand{\V}{\mathbb{V}}
\newcommand{\R}{\mathbb{R}}

\newcommand{\e}{\mathrm{e}}
\newcommand{\vv}{\mathrm{v}}
\newcommand{\w}{\mathrm{w}}
\newcommand{\Z}{\mathbb{Z}}

\newcommand{\Vd}{\mathbb{V}^d_\varepsilon}
\newcommand{\Gd}{\mathcal{G}^d_\varepsilon}
\newcommand{\Ed}{\mathbb{E}^d_\varepsilon}

\makeatletter

\tikzstyle{nodo}=[circle,draw,fill,inner sep=0pt,minimum size=%
1.5mm]
\tikzstyle{bnodo}=[circle,draw,fill,inner sep=0pt,minimum size=%
2mm]

\numberwithin{equation}{section}

\title[Singular limit and critical Lane--Emden equations]{Singular limit of lattice graphs and its application to critical Lane--Emden equations on lattice graphs}
\author[Z. He]{Zhentao He}

\address[Z. He]{\newline\indent
	School of Mathematics
	\newline\indent
	East China University of Science and Technology
	\newline\indent
	Shanghai 200237, PR China }
\email{\href{mailto:hezhentao2001@outlook.com}{hezhentao2001@outlook.com}}

\author[C. Ji]{Chao Ji}

\address[C. Ji]{\newline\indent
	School of Mathematics
	\newline\indent
	East China University of Science and Technology
	\newline\indent
	Shanghai 200237, PR China }
\email{\href{mailto:jichao@ecust.edu.cn}{jichao@ecust.edu.cn}}



\begin{document}
\begin{abstract}
In this paper, we establish new connections between lattice graphs and metric grids, providing a unified framework for the study of singular limit problems and Gagliardo--Nirenberg type inequalities on lattice graphs. The main technical ingredients are restriction and extension estimates, which enable us to compare variational problems posed on lattice graphs, metric grids and \(\mathbb R^d\). As applications, we prove that extensions of action ($2<p<2^*$) and energy ($2<p<2+\frac{4}{d}$) ground states of the nonlinear Schr\"{o}dinger (NLS) equation on $d$-dimensional lattice graphs converge strongly in $H^1(\R^d)$ to the corresponding ground states on $\R^d$ as the edge length tends to zero. As a by-product of the arguments developed for the singular limit problem on lattice graphs, we obtain multiplicity results for fixed-mass critical points of the energy functional on lattice graphs. Beyond the classical subcritical framework, we also study the singular limit of action ground states in the Sobolev supercritical regime ($d \geq 3$ and $p>2^*$), and the singular limit of energy ground states in the mass-supercritical regime ($p>2+\frac{4}{d}$) on lattice graphs. Notably, by establishing a new Gagliardo--Nirenberg type inequality, we settle an open problem posed by Dovetta [Adv. Math. 444 (2024), 109633] on the convergence rates of ground states energies in singular limit problems. In addition, by applying singular limit analysis, we prove that for $d\geq3$ and $p=2^*$, the best constants in the Sobolev inequalities on lattice graphs are attained and, consequently, establish the existence of a positive solution to the associated critical Lane--Emden equation, thereby resolving an open problem posed by Hua and Li [J. Differential Equations, 305 (2021), 224--241]. We further establish the attainability of best constants of Gagliardo--Nirenberg type inequalities on lattice graphs for $p\nearrow 2^*$. Finally, We prove the attainability of best constants of Gagliardo--Nirenberg type inequalities on metric grids.
\end{abstract}
\maketitle

\setcounter{tocdepth}{3}
\begin{center}
	\begin{minipage}{12.5cm}
		\tableofcontents
	\end{minipage}
\end{center}

\section{Introduction}

As in \cite[Section 1.4]{Gr}, for any $u \in C^2(\R^d)$, $\Delta u$ can be approximated by a finite difference expression involving the values of $u$ at neighboring points; for example, if $d=2$, then, for small $\varepsilon >0$ 
\begin{equation}\label{eqlaplace}
    \Delta u (x,y)\approx \frac{1}{\varepsilon^2}\left( u(x+\varepsilon,y)+u(x-\varepsilon,y)+u(x,y+\varepsilon)+u(x, y-\varepsilon)-4u(x,y)\right).
\end{equation}
Formally, the right-hand side of \eqref{eqlaplace} coincides with  $\varepsilon^{-1}\Delta_\varepsilon^d u$, according to the definition of the graph Laplacian in \eqref{eqglaplace}. In this sense, lattice graphs combine a discrete microscale with a $d$-dimensional macroscale. In the singular limit as the edge length approaches zero, these structures can be viewed as locally discrete approximations of $\mathbb{R}^{d}$. This connection offers a dual perspective: lattice graphs can act as numerical or conceptual tools to approximate and investigate phenomena in $\R^d$, while well-understood continuous models in $\mathbb{R}^{d}$ can serve as effective macroscopic descriptions for complex dynamics on lattice graphs. This naturally raises a fundamental mathematical question: as the edge length of lattice graphs  tends to zero, is it possible to rigorously establish the convergence of solutions on lattice graphs to those of the corresponding continuous model in $\mathbb{R}^{d}$?  Indeed, analogous results for the stationary nonlinear Schr\"{o}dinger (NLS) equations  have already been established on metric grids in \cite{Do2}.

In $\mathbb{R}^d$, it is well known (see e.g. \cite{CazenaveLivro}) that for every $p \in (2, 2^*)$, where $2^* = \infty$ if $d=1,2$ and $2^* = \frac{2d}{d-2}$ if $d \ge 3$, the (NLS) equation
\begin{equation}\label{eqrd}
    -\Delta u + \omega u=|u|^{p-2}u 
\end{equation}
admits, for every $\omega > 0$, a unique (up to translation) positive solution $u$ decaying at infinity. From a variational point of view, this solution can be characterized in at least two ways. For every $p \in (2, 2^*)$, $u$ is an action ground state at frequency $\omega$, i.e. a global minimizer of the action functional
\[
    J_{\omega,\mathbb{R}^d}(u) := \frac{1}{2}\|\nabla u\|_{L^2(\mathbb{R}^d)}^2 + \frac{\omega}{2}\|u\|_{L^2(\mathbb{R}^d)}^2 - \frac{1}{p}\|u\|_{L^p(\mathbb{R}^d)}^p
\]
restricted to the associated Nehari manifold
\begin{align*}
    \mathcal{N}_{\omega,\mathbb{R}^d} :&= \left\{ u \in H^1(\mathbb{R}^d)\backslash \{0\}  : J'_{\omega,\mathbb{R}^d}(u)u = 0 \right\} \\
    &= \left\{ u \in H^1(\mathbb{R}^d) \backslash \{0\} : \|\nabla u\|_{L^2(\mathbb{R}^d)}^2 + \omega\|u\|_{L^2(\mathbb{R}^d)}^2 = \|u\|_{L^p(\mathbb{R}^d)}^p \right\}.
\end{align*}
Moreover, for every $p \in \left(2, 2 + \frac{4}{d}\right)$, $u$ is also an energy ground state with mass $\mu$, i.e. a global minimizer of the energy functional
\[
    E_{\mathbb{R}^d}(u) := \frac{1}{2}\|\nabla u\|_{L^2(\mathbb{R}^d)}^2 - \frac{1}{p}\|u\|_{L^p(\mathbb{R}^d)}^p
\]
on the mass constraint
\[
    H_{\mu}^1(\mathbb{R}^d) := \left\{ u \in H^1(\mathbb{R}^d) : \|u\|_{L^2(\mathbb{R}^d)}^2 = \mu \right\}.
\]
In this case, $\omega$ plays the role of a Lagrange multiplier and is in one-to-one correspondence with the mass $\mu$.

In the seminal paper \cite{Do2}, for both action ($2<p<2^*$) and energy ($2<p<2+\frac{4}{d}$)  ground states of the (NLS) equation on metric grids, Dovetta proved that their suitable piecewise-affine extensions converge strongly in $H^1(\R^d)$ to the corresponding ground states on $\R^d$ as the edge length tends to zero. To study this asymptotic behavior,  Dovetta employed purely variational arguments, based only on the minimality of ground states, with a rigorous analysis of the interaction between scales of different dimensions in the grid.  Dovetta first established restriction and extension estimates between functions on metric grids $\Gd$ and Euclidean spaces $\R^d$. With these estimates, the author proved that, after multiplying each term by a suitable positive constant, the sequence of extensions of ground states is a minimizing sequence for the action functional restricted to the Nehari manifold (or the energy functional on the mass constraint, respectively) of the (NLS) equation on $\R^d$,  from which the convergence result follows. Applying these convergence results, Dovetta also derived qualitative properties of ground states and multiplicity results for fixed-mass critical points on metric grids. Furthermore, based on similar ideas, a comparison was established between the optimal constants of $d$-dimensional Gagliardo--Nirenberg inequalities ($2<p<2^*$) on $\mathbb{R}^d$ and on the metric grid $\mathcal{G}_1^d$. In a subsequent work \cite{BBDT2}, Barbera, Boni, Dovetta and Tentarelli studied the existence and the singular limit of energy ground states for focusing doubly (NLS) equations with both standard and concentrated nonlinearities on two-dimensional square grids. For further results concerning nonlinear PDEs on metric graphs, we refer the reader to \cite{ABD, ACFN, AST, AST2, AST3, Do1, NP, PS, Cha, Ca, Do0, HJ7,  Do3} and references therein. It is worth noting that the singular limit problem for the action ($d\geq 3$ and $p \geq 2^*$) or the energy ($p\geq 2+\frac{4}{d}$) ground states of the (NLS) equation on metric grids, as well as the best constants of $d$-dimensional Gagliardo--Nirenberg inequalities on the metric grid $\mathcal{G}_1^d$ for $d\geq 3$ and $p=2^*$, have not been studied.

\begin{figure}[htbp]\label{figure1}
    \centering
    
    \begin{minipage}[b]{0.48\textwidth}
        \centering
        \begin{tikzpicture}[xscale=0.6, yscale=0.6]
            \foreach \y in {-4, -2, 0, 2, 4} {
                \draw (-4, \y) -- (4, \y);
                \draw [dashed] (-5.5, \y) -- (-4, \y);
                \draw [dashed] (4, \y) -- (5.5, \y);
            }
            \foreach \x in {-4, -2, 0, 2, 4} {
                \draw (\x, -4) -- (\x, 4);
                \draw [dashed] (\x, -5.5) -- (\x, -4);
                \draw [dashed] (\x, 4) -- (\x, 5.5);
            }
            \foreach \x in {-4, -2, 0, 2, 4} {
                \foreach \y in {-4, -2, 0, 2, 4} {
                    \node at (\x, \y) [nodo] {};
                }
            }
        \end{tikzpicture}
        \label{fig:2d_grid}
        \caption*{(a)}
    \end{minipage}
    \hspace{0,1em}
    \begin{minipage}[b]{0.48\textwidth}
        \centering
        \begin{tikzpicture}[xscale=0.675, yscale=0.675, z={(0.5, 0.4)}]
            
            \foreach \x in {-3, -1, 1, 3} {
                \foreach \y in {-3, -1, 1, 3} {
                    \draw [dashed] (\x, \y, 3.5) -- (\x, \y, 2);  
                    \draw (\x, \y, 2) -- (\x, \y, 0);             
                    \draw [dashed] (\x, \y, 0) -- (\x, \y, -1.5); 
                }
            }
            
            \foreach \z in {2, 0} {
                \foreach \y in {-3, -1, 1, 3} {
                    \draw [dashed] (-4.5, \y, \z) -- (-3, \y, \z);
                    \draw (-3, \y, \z) -- (3, \y, \z);
                    \draw [dashed] (3, \y, \z) -- (4.5, \y, \z);
                }
                \foreach \x in {-3, -1, 1, 3} {
                    \draw [dashed] (\x, -4.5, \z) -- (\x, -3, \z);
                    \draw (\x, -3, \z) -- (\x, 3, \z);
                    \draw [dashed] (\x, 3, \z) -- (\x, 4.5, \z);
                }
                \foreach \x in {-3, -1, 1, 3} {
                    \foreach \y in {-3, -1, 1, 3} {
                        \node at (\x, \y, \z) [nodo] {};
                    }
                }
            }
        \end{tikzpicture}
        \label{fig:3d_grid}
        \caption*{(b)}
    \end{minipage}
            \caption{The two-dimensional lattice graph $\Z^2_\varepsilon$ and  the three-dimensional lattice graph $\Z^3_\varepsilon$.}
\end{figure}

 Inspired by \cite{BBDT2,Do2}, in this paper we study the asymptotic behavior of action and energy ground states of  (NLS) equation on lattice graphs as the edge length tends to zero.  Compared with the analysis in \cite{Do2}, a major difficulty in our setting is the absence of restriction and extension estimates between functions on lattice graphs $\Vd$ and Euclidean spaces $\R^d$. To overcome this difficulty, in Section \ref{sec2} we establish restriction and extension estimates between functions on lattice graphs $\Vd$ and metric grids $\Gd$, which are of independent interest (see Figure \hyperref[figvg]{2} for the extension from  $\Vd$ to $\Gd$ and Figure \hyperref[figgv]{3} for the restriction from  $\Gd$ to $\Vd$). Then, combining the estimates in \cite{Do2} and Sections \ref{sec2}-\ref{sec4}, we prove in Theorems \ref{th1} and \ref{th3} that the sequences of extensions of action ($2<p<2^*$) and energy ($2<p<2+\frac{4}{d}$) ground states of \eqref{eqnlsv}  are minimizing sequences for the corresponding functionals restricted to the associated manifold, and then converge strongly in $H^1(\R^d)$ to the corresponding ground states on $\R^d$ as $\varepsilon \to 0$. As a by-product of the arguments developed for the singular limit problem, a multiplicity result for fixed-mass critical points of the energy functional on the lattice graph $\V^d_1$ is also established in Theorem \ref{th4}. Moreover, we study the singular limit problem for action ($d\geq 3$ and $p > 2^*$) and energy ($d\geq 2$ and $p> 2+\frac{4}{d}$) ground states of the (NLS) equation on lattice graphs in Theorems \ref{th2} and \ref{thl2sup}, respectively. We also show in Remark \ref{re1} that ground states of (NLS) equation on $\R^d$ can be approximated by ground states of (NLS) equation on finite graphs. By employing a strategy analogous to the analysis of the singular limit problem, we prove in Theorem \ref{th6} that the best constant in the discrete Sobolev inequality on lattice graphs is attained and, consequently, establish the existence of a positive solution to the associated critical Lane--Emden equation on lattice graphs. Furthermore, the attainability of the best constants in Gagliardo--Nirenberg type inequalities for $p\nearrow 2^*$ is established in Theorem \ref{th7}. Finally, we extend the results of Theorems \ref{th6} and \ref{th7} to metric grids in Theorems \ref{thmg1} and \ref{thmg2}, respectively.
 
In \cite[Introduction]{Do2}, Dovetta posed the following question:
{\bf \begin{itemize}
    \item [($Q_1$)]\label{q1} For ground states of the (NLS) equation on metric grids, can the rate $o\left(\varepsilon^{\frac{d-2}{2}(2^*-p)-\gamma}\right)$ for every $\gamma>0$ in \cite[Theorems 2.1 and 2.2]{Do2} be improved to $O(\varepsilon^{\frac{d-2}{2}(2^*-p)})$?
\end{itemize}}
In Theorems \ref{th1} and \ref{th3}, thanks to the newly established Gagliardo--Nirenberg type inequality \eqref{eqlinfty}, we answer the question (\hyperref[q1]{$Q_1$}) in the affirmative for the ground states of the (NLS) equation on lattice graphs. Furthermore, by repeating our arguments, one can also give a positive answer to question (\hyperref[q1]{$Q_1$}) for the ground states of the (NLS) equation on metric grids.

It is also worth noting that the discrete translation alignment used in the proofs of \cite[Theorems~2.1--2.2]{Do2}, while sufficient to prevent vanishing, does not explicitly fix the center of the limiting ground state at the origin. In the proofs of Theorems \ref{th1} and \ref{th3}, we instead choose suitable translations in $\mathbb R^d$ to ensure that every subsequential limit is the unique positive ground state attaining its maximum at the origin. This, in turn, yields convergence of ground states in Theorems \ref{th1}(ii) and \ref{th3}(ii) as $\varepsilon\to0$, without passing to a subsequence.

In \cite{hua2021existence}, Hua and Li posed the following question:
{\bf \begin{itemize}
    \item [($Q_2$)]\label{q2} 
    For $d\geq 3$ and $\frac{2+2d}{d}<q\leq 2^*$, does the following Lane--Emden equation on the lattice graph $\Z^d$
\[
-\Delta u=u^{q-1}\quad\text{in }\mathbb Z^d
\]
admit a non-trivial non-negative  solution?
\end{itemize}}
The nonexistence result of Hua and Li \cite[Theorem 5]{hua2021existence}, initially obtained for
\(2<q\leq 2+2/d\), was later improved by Gu, Huang, and Sun
\cite[Remark 1.3]{GHS} to
\(
2<q\leq 2+\frac{2}{d-2}.
\)
Since positive solutions exist for \(q>2^*=2d/(d-2)\), the intermediate
range
\(
2+\frac{2}{d-2}<q\leq 2^*
\)
remains open, including the critical case \(q=2^*\).

In Theorem \ref{th6}, we answer question (\hyperref[q2]{$Q_2$}) in the affirmative for $q=2^*$. However, the case
\(
2+\frac{2}{d-2}<q<2^*
\)
remains open.
\begin{figure}[htbp]\label{figvg}
    \centering

        \begin{minipage}[b]{0.48\textwidth}
        \centering
    \begin{tikzpicture}[xscale=0.6, yscale=0.6]
                    \useasboundingbox (-6, -3) rectangle (6, 3);

                \draw (-4, 0) -- (4, 0);
                \draw [dashed] (-5.5, 0) -- (-4, 0);
                \draw [dashed] (4, 0) -- (5.5, 0);
            
            \foreach \x in {-4, -2, 0, 2, 4} {   
                    \node at (\x, 0) [nodo] {};
            }
            \node at (-4,-1.7) [nodo] {} [blue];
            \node at (-2,0.4) [nodo] {} [blue];
            \node at (0,1.2) [nodo] {} [blue];
            \node at (2,-0.4) [nodo] {} [blue];
            \node at (4,1.4) [nodo] {} [blue];
                        \draw (4,2)  node{\blue{$u$}};
                    \end{tikzpicture}
    \end{minipage}
       \begin{minipage}[b]{0.48\textwidth}
        \centering
        \begin{tikzpicture}[xscale=0.6, yscale=0.6]
                        \useasboundingbox (-6, -3) rectangle (6, 3);

                \draw (-4, 0) -- (4, 0);
                \draw [dashed] (-5.5, 0) -- (-4, 0);
                \draw [dashed] (4, 0) -- (5.5, 0);
            
            \foreach \x in {-4, -2, 0, 2, 4} {   
                    \node at (\x, 0) [nodo] {};
            }
            \draw (-5.5,-1.5)--(-4,-1.7)[dashed,blue];
            \draw (-4,-1.7)--(-2,0.4)[blue];
            \draw (-2,0.4)--(0,1.2)[blue];
            \draw (0,1.2)--(2,-0.4)[blue];
            \draw (2,-0.4)--(4, 1.4)[blue]; 
            \draw (4,1.4)--(5.5,1.3) [dashed,blue];
                        \draw (4,2)  node{\blue{$\tilde{u}$}};
        \end{tikzpicture}
    \end{minipage}
    \hspace{0,1em}
            \caption{The extension $\tilde{u}$ of a function $u \in H^1(\V^1_\varepsilon)$ to $\G^1_\varepsilon$ (as defined in Section \ref{sec2}).}
\end{figure}
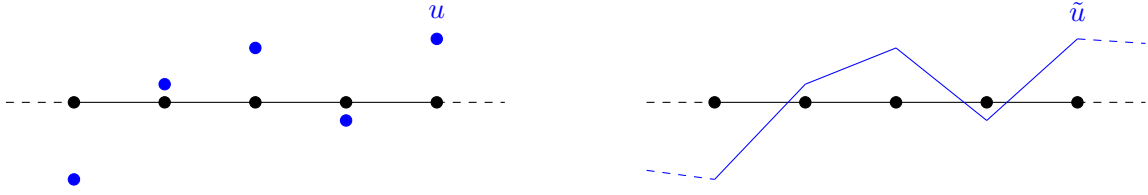
\begin{figure}[htbp]\label{figgv}
    \centering
    
    \begin{minipage}[c]{0.48\textwidth}
        \centering
        \begin{tikzpicture}[ xscale=0.6, yscale=0.6]
            \useasboundingbox (-6, -3) rectangle (6, 3);
                \draw (-4, 0) -- (4, 0);
                \draw [dashed] (-5.5, 0) -- (-4, 0);
                \draw [dashed] (4, 0) -- (5.5, 0);
            
            \foreach \x in {-4, -2, 0, 2, 4} {   
                    \node at (\x, 0) [nodo] {};
            }
            \draw (-5.5,0.3).. controls (-5,0.4) ..(-4,1)[dashed,blue];
            \draw (-4,1).. controls (-3,1.2) ..(-2,2)[blue];
            \draw (-2,2).. controls (-1.5,0.5) and (-1, -1.5) ..(0,1)[blue];
            \draw (0,1).. controls (1,-2) ..(2,-1)[blue];
            \draw (2,-1).. controls (2.5,-1.2) and (3,2) ..(4, -0.5)[blue]; 
            \draw (4, -0.5).. controls (5,-0.2) .. (5.5,-0.1) [dashed,blue];
            \draw (4,2)  node{\blue{$w$}};
        \end{tikzpicture}
    \end{minipage}
        \begin{minipage}[c]{0.48\textwidth}
        \centering
    \begin{tikzpicture}[xscale=0.6, yscale=0.6]
        \useasboundingbox (-6, -3) rectangle (6, 3);
                \draw (-4, 0) -- (4, 0);
                \draw [dashed] (-5.5, 0) -- (-4, 0);
                \draw [dashed] (4, 0) -- (5.5, 0);
            
            \foreach \x in {-4, -2, 0, 2, 4} {   
                    \node at (\x, 0) [nodo] {};
            }
            \node at (-4,1) [nodo] {} [blue];
            \node at (-2,2) [nodo] {} [blue];
            \node at (0,1) [nodo] {} [blue];
            \node at (2,-1) [nodo] {} [blue];
            \node at (4,-0.5) [nodo] {} [blue];
            \draw (4,2)  node{\blue{$\hat{w}$}};
                    \end{tikzpicture}
    \end{minipage}
            \caption{The restriction $\hat{w}$ of a function $w \in H^1(\G^1_\varepsilon)$ to $\V^1_\varepsilon$ (as defined in Section \ref{sec2}).}
\end{figure}
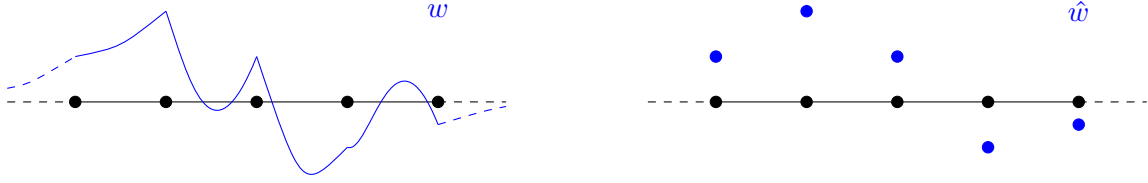

\subsection{The functional setting}

For every $d \geq 1$ and $\varepsilon>0$, define $$\varepsilon\mathbb{Z}^{d}:=\left\{\varepsilon x:=(\varepsilon x_{1},...,\varepsilon x_{d}): x =(x_{1},...,x_{d}) \in \Z^d \right\},$$
where $$\Z^d:=\left\{x=(x_{1},...,x_{d}):x_{i}\in\mathbb{Z},1\le i\le d\right\}.$$
We denote by $\mathbb{Z}^{d}_\varepsilon:=(\V_\varepsilon^d, \E_\varepsilon^d)$ the lattice graph   (with edge length $\varepsilon$, see e.g. Figure \hyperref[figure1]{1} for $d=2,3$) with the set of vertices
$$\V_\varepsilon^d:=\varepsilon\mathbb{Z}^{d}$$
and the set of edges
\[\E_\varepsilon^d=\left\{(x,y):x,y\in \V_\varepsilon^d,\sum_{i=1}^{d}|x_{i}-y_{i}|=\varepsilon \right\}.\]
Let us write $x\sim y$ ($y$ is a neighbour of $x$) if $(x,y) \in \E_\varepsilon^d$.
We denote the space of functions on $\V_\varepsilon^d$ by $C(\V_\varepsilon^d)$. For $u\in C(\V_\varepsilon^d)$, its support set is defined as $\operatorname{supp}(u):=\{x\in\V_\varepsilon^d:u(x)\neq 0\}$. Let $C_{c}(\V_\varepsilon^d)$ be the set of all functions with finite support. 

Let $\mu$ be the counting measure on $\V_\varepsilon^d$, i.e., for any subset $A \subset \V_\varepsilon^d$, $\mu(A):=\#\{x: x \in A\}$. For any function $f$ on $\V_\varepsilon^d$, we write $$\int_{\V_\varepsilon^d} f d \mu:=\sum_{x \in \V_\varepsilon^d} f(x),$$ whenever it makes sense.

For any $1 \le p \le  \infty$, $\ell^p(\V_\varepsilon^d)$ denotes the linear space of $p$-th integrable functions on $\V_\varepsilon^d$ equipped with the norm
$$
\|u\|_{\ell^p(\V_\varepsilon^d)}:=
\begin{cases}\displaystyle
	\left( \sum_{x \in \V_\varepsilon^d} |u(x)|^p \right)^{1/p}, & 1 \le p <\infty, \\
    \displaystyle
	\sup_{x \in \V_\varepsilon^d} |u(x)|, & p = \infty.
\end{cases}
$$

For any functions $u, v \in C(\V_\varepsilon^d)$, we define the associated gradient form as
$$
\Gamma_\varepsilon^d(u, v)(x):= \frac{1}{2\varepsilon} \sum_{y \sim x} (u(y)-u(x))(v(y)-v(x)).
$$
Let $\Gamma_\varepsilon^d(u):=\Gamma_\varepsilon^d(u, u)$, and define
\begin{equation}\label{tidu}
	|\nabla_\varepsilon^d u|(x):=\sqrt{\Gamma_\varepsilon^d(u)(x)}=\left( \frac{1}{2\varepsilon} \sum_{y \sim x}\abs{u(y)-u(x)}^2\right)^{1 / 2} .
\end{equation}
For brevity, we write $\norm{|\nabla_\varepsilon^d u|}_{\ell^2(\V_\varepsilon^d)}$ as $\norm{\nabla_\varepsilon^d u}_{\ell^2(\V_\varepsilon^d)}$.

Let $H^1(\V_\varepsilon^d)$ be the completion of $C_c(\V_\varepsilon^d)$ under the norm
$$\|u\|_{H^{1}(\V_\varepsilon^d)}:=\left(\norm{\nabla_\varepsilon^d u}_{\ell^2(\V_\varepsilon^d)}^2+\norm{u}_{\ell^2(\V_\varepsilon^d)}^2\right)^{1 / 2}=\left(\frac{1}{2\varepsilon} \sum_{x \in \V_\varepsilon^d} \sum_{y \sim x}\abs{u(y)-u(x)}^2+\sum_{x \in \V_\varepsilon^d}  \abs{u(x)}^2\right)^{1 / 2}.$$
For any function $u \in C(\V_\varepsilon^d)$, the graph Laplacian of $u$ is defined by
\begin{equation}\label{eqglaplace}
    \Delta_\varepsilon^d u(x):=\frac{1}{\varepsilon} \sum_{y \sim x} (u(y)-u(x)), \quad \forall x \in \mathbb{V}_\varepsilon^d.
\end{equation}
Extensive research  has been devoted to nonlinear PDEs on discrete graphs, see \cite{Gr, Gri1, hua2021existence, HX1, HLY, DSR, Stefanov,Weinstein,HJ,HJT}  and references therein. In view of the scope of this paper, we would like to mention \cite{HX1,Weinstein}. In \cite{HX1}, applying the Nehari method, Hua and Xu studied the existence of action ground states of the (NLS) equation $-\Delta u+V(x) u=f(x, u)$ on the lattice graph $\V^d_1$ where $d\geq 1$, $f$ satisfies some growth conditions and the potential function $V$ is periodic or bounded. In \cite{Weinstein}, Weinstein studied the existence of energy ground states of the (NLS) equation $-\Delta u+\lambda u=|u|^{p-2}u$ on the lattice graph $\V^d_1$, where $d \geq 1$, $p>2$, the mass $\norm{u}_{\ell^2(\V^d_1)}^2=\mu$ is prescribed,  and  $\lambda \in \mathbb{R}$ arises  as a Lagrange multiplier. The results established in \cite{HX1,Weinstein} can naturally be applied to the (NLS) equation on $\Vd$, which provide the variational existence results needed for the analysis of singular limit problem on lattice graphs.

\subsection{Main results}  

We can now state our main results, starting with the ground state problem for the (NLS) equation. Let us begin with action ground states. For every $p >2$ and every $\omega > 0$, we introduce the action functional $\hat{J}_{\omega, \V_\varepsilon^d} : H^1(\V_\varepsilon^d) \to \mathbb{R}$
\begin{equation*}
    \hat{J}_{\omega, \V_\varepsilon^d}(u) := \frac{1}{2}\|\nabla_\varepsilon^d u\|_{\ell^2(\V_\varepsilon^d)}^2 + \frac{\varepsilon\omega}{2}\|u\|_{\ell^2(\V_\varepsilon^d)}^2 -\frac{\varepsilon}{p}\|u\|_{\ell^p(\V_\varepsilon^d)}^p
\end{equation*}
and the associated Nehari manifold
\begin{equation}\label{eqdefnv}
\begin{aligned}
    \hat{\mathcal{N}}_{\omega, \V_\varepsilon^d} :&= \left\{ u \in H^1(\V_\varepsilon^d) \backslash \{0\} : \hat{J}'_{\omega, \V_\varepsilon^d}(u)u = 0 \right\} \\
    &= \left\{ u \in H^1(\V_\varepsilon^d)\backslash \{0\}  : \|\nabla_\varepsilon^d u\|_{\ell^2(\V_\varepsilon^d)}^2 + \varepsilon\omega\|u\|_{\ell^2(\V_\varepsilon^d)}^2 = \varepsilon\|u\|_{\ell^p(\V_\varepsilon^d)}^p \right\}.
\end{aligned}
\end{equation}
Letting
\begin{equation} \label{eqminv}
    \hat{\mathcal{J}}_{\V_\varepsilon^d}(\omega) := \inf_{v \in \hat{\mathcal{N}}_{\omega, \V_\varepsilon^d}} \hat{J}_{\omega, \V_\varepsilon^d}(v)
\end{equation}
be the corresponding minimization problem,  $u \in \hat{\mathcal{N}}_{\omega, \V_\varepsilon^d}$ is called a ground state of $\hat{J}_{\omega, \V_\varepsilon^d}$ if $\hat{J}_{\omega, \V_\varepsilon^d}(u) = \hat{\mathcal{J}}_{\V_\varepsilon^d}(\omega)$. By \cite[Theorem 1.1]{HX1}, there always exist ground states of $\hat{J}_{\omega, \V_\varepsilon^d}$ in $\hat{\mathcal{N}}_{\omega, \V_\varepsilon^d}$ for every $\omega > 0$, $\varepsilon > 0$ and $p > 2$. Ground states are solutions of the following (NLS) equation  
\begin{equation}\label{eqnlsv}
   -\frac{\Delta_\varepsilon^d}{\varepsilon} u+\omega u = \abs{u}^{p-2}u \quad \text{in } \V_\varepsilon^d. 
\end{equation}
In addition, by \cite[Theorem B.4]{HJ}, ground states have constant sign. Hence, without loss of generality, we may assume that they are positive throughout this paper.
The next theorem provides an affirmative answer to the question of whether action ground states on $d$-dimensional lattice graphs with vanishing edge length approximate ground states in $\mathbb{R}^d$ in the Sobolev subcritical case.

\subsubsection{Singular limits of action and energy ground states} 

In what follows, for every $\omega > 0$, we make use of the shorthand notation
\[
    \mathcal{J}_{\mathbb{R}^d}(\omega) := \inf_{v \in \mathcal{N}_{\omega, \mathbb{R}^d}} J_{\omega, \mathbb{R}^d}(v).
\]

\begin{theorem}\label{th1}
Let $d \geq 1$, $p \in (2, 2^*)$ and $\omega > 0$ be fixed. \begin{enumerate}[label=\rm(\roman*)]
\item If $d=1,2$ and $p>2$, or if $d \ge 3$ and $p \in \left(2, \frac{2^*}{2}+1\right]$, then there exists a constant $C_{d,p,\omega} > 0$, depending only on $d$, $p$ and $\omega$, such that for every $\varepsilon>0$ small enough,
\[
    \left|\varepsilon^{d-1}\hat{\mathcal{J}}_{\V_\varepsilon^d}(\omega) - \mathcal{J}_{\mathbb{R}^d}(\omega)\right| \le C_{d,p,\omega}\varepsilon.
\]
If $d \ge 3$ and $p \in \left(\frac{2^*}{2}+1, 2^*\right)$, then there exists a constant $C_{d,p,\omega}> 0$, depending only on $d$, $p$ and $\omega$, such that for every $\varepsilon>0$ small enough,
\[
    \left|\varepsilon^{d-1}\hat{\mathcal{J}}_{\V_\varepsilon^d}(\omega) - \mathcal{J}_{\mathbb{R}^d}(\omega)\right| \le C_{d,p,\omega}\varepsilon^{\frac{d-2}{2}(2^*-p)}.  
\]
\item For every positive ground state $u_\varepsilon$ of $\hat{J}_{\omega, \V_\varepsilon^d}$ in $\hat{\mathcal{N}}_{\omega, \V_\varepsilon^d}$, there exists a point $x_\varepsilon \in \mathbb{R}^d$ such that, 
\[
   \bigl(\hat{\mathcal{A}}u_\varepsilon\bigr)(\cdot+x_\varepsilon) \xrightarrow{\varepsilon \to 0} \varphi_\omega,\quad \text{in } H^1(\mathbb{R}^d),
\]
where $\varphi_\omega \in \mathcal{N}_{\omega, \mathbb{R}^d}$ is the unique positive ground state of $J_{\omega, \mathbb{R}^d}$ which attains its maximum at the origin and $\hat{\mathcal{A}}$ is the extension operator from $\Vd$ to $\R^d$ introduced in Section \ref{sec3}.
\end{enumerate}
\end{theorem}

In Theorem \ref{th1}, we studied the asymptotic behavior of ground states of $\hat{J}_{\omega, \V_\varepsilon^d}$ as $\varepsilon \to 0$ for all $p \in (2,2^*)$. Since ground states of $\hat{J}_{\omega, \V_\varepsilon^d}$ in $\hat{\mathcal{N}}_{\omega, \V_\varepsilon^d}$ always exist for every $d\geq 1$, $\omega > 0$, $\varepsilon > 0$ and $p > 2$, it is natural to further investigate their asymptotic behavior in the Sobolev supercritical case $p>2^*$ when  $d\geq 3$. Our main result in this direction is stated below, and we also emphasize that no analogous results are available in \cite{Do2}.

\begin{theorem}\label{th2}
    Let $d \ge 3$, $p > 2^*$ and $\omega > 0$ be fixed. Then
\[
    \varepsilon^{d-1}\hat{\mathcal{J}}_{\V_\varepsilon^d}(\omega) \rightarrow 0, \quad \text{as }\varepsilon \to 0,
\]
and, for every positive ground state $u_\varepsilon$ of $\hat{J}_{\omega, \V_\varepsilon^d}$ in $\hat{\mathcal{N}}_{\omega, \V_\varepsilon^d}$
\[
    \hat{\mathcal{A}}u_\varepsilon \xrightarrow{\varepsilon \to 0} 0, \quad \text{in } H^1(\mathbb{R}^d),
\]
where $\hat{\mathcal{A}}$ is the extension operator from $\Vd$ to $\R^d$  introduced in Section \ref{sec3}.
\end{theorem}
It is well known that, for $d \geq 3$ and $p>2^*$, the functional $\int_{R^d}\abs{u}^{p}\,dx$ is no longer well-defined on $H^1(\R^d)$. However, roughly speaking, using scaling techniques, we  prove in Lemma \ref{lemaction} below that the infimum of the “action functional” restricted to the associated “Nehari manifold" is equal to $0$,  where the action functional is actually a map defined on $C_c(\R^d)$ and the Nehari manifold is a subset of $C_c(\R^d)$. After showing the infimum is $0$, by estimates established in \cite{Do2} and Section \ref{sec2}, we further show  that the extensions of action ground states of \eqref{eqnlsv} converge strongly to $0$ in $H^1(\R^d)$  as $\varepsilon \to 0$.

An analogous result holds true for fixed-mass ground states of the energy functional.
For every $\mu>0$, define
$$
H_\mu^1(\Vd):=\left\{u \in H^1(\Vd): \norm{u}^2_{\ell^2(\Vd)}=\mu \right\}.
$$
For every $p > 2$, we introduce the energy functional $\hat{E}_{\Vd} : H^1(\Vd) \to \mathbb{R}$
\begin{equation}
\hat{E}_{\Vd}(u) := \frac{1}{2}\|\nabla^d_\varepsilon u\|_{\ell^2(\Vd)}^2 - \frac{\varepsilon}{p}\|u\|_{\ell^p(\Vd)}^p
\end{equation}
and denote by
$$
\hat{\mathcal{E}}_{\Vd}(\mu) := \inf_{v \in H_\mu^1(\Vd)} \hat{E}_{\Vd}(v)
$$
the corresponding ground state problem at mass $\mu > 0$. As usual, $u \in H_\mu^1(\Vd)$ is called a ground state of $\hat{E}_{\Vd}$ at mass $\mu$ if $\hat{E}_{\Vd}(u) = \hat{\mathcal{E}}_{\Vd}(\mu)$. If $u$ is a ground state of $\hat{E}_{\Vd}$, then
\begin{equation}\label{eqenergyl}
   -\Delta_\varepsilon^d u+\varepsilon\mathcal{L}_{\Vd}(u) u = \varepsilon\abs{u}^{p-2}u \quad \text{in } \V_\varepsilon^d,
\end{equation}
where
\begin{equation}
\mathcal{L}_{\Vd}(u) := \frac{\varepsilon\|u\|_{\ell^p(\Vd)}^p - \|\nabla^d_\varepsilon u\|_{\ell^2(\Vd)}^2}{\varepsilon\|u\|_{\ell^2(\Vd)}^2}.
\end{equation}
Adapting the analysis of \cite{Stefanov,Weinstein} ensures that there always exist nonnegative ground states of $\hat{E}_{\Vd}$ at mass $\mu$ for every $\varepsilon > 0$, $\mu > 0$ and $p \in (2, 2 + \frac{4}{d})$. Moreover, thanks to the connectivity of $\Z^d_\varepsilon$, we know that the nonnegative ground states of $\hat{E}_{\Vd}$ are positive.  For such ground states we have the following convergence result, where we also use the notation

$$
\mathcal{E}_{\mathbb{R}^d}(\mu) := \inf_{v \in H_\mu^1(\R^d)} E_{\mathbb{R}^d}(v).
$$
\begin{theorem}\label{th3}
        Let $d \ge 1$, $p \in (2,2+\frac{4}{d})$ and $\mu> 0$ be fixed. 
\begin{enumerate}[label=\rm(\roman*)]
    \item If $d\in \{1,2,3,4\}$ and $p \in (2, 2+\frac{4}{d})$, or $d \geq 5$ and $p \in \left(2,\frac{2^*}{2}+1\right]$, then there exists a constant $C_{d,p,\mu}>0$, depending only on $d$, $p$ and $\mu$, such that for every $\varepsilon>0$ small enough,
    $$
    \left|\varepsilon^{d-1}\hat{\mathcal{E}}_{\V_\varepsilon^d}\left(\frac{\mu}{\varepsilon^{d}}\right) - \mathcal{E}_{\mathbb{R}^d}(\mu)\right| \le C_{d,p,\mu}\varepsilon.
    $$
If $d \geq 5$ and $p \in \left(\frac{2^*}{2}+1, 2+\frac{4}{d}\right)$, then there exists a constant $C_{d,p,\mu}>0$, depending only on $d$, $p$ and $\mu$, such that for every $\varepsilon>0$ small enough,
    $$
\left|\varepsilon^{d-1}\hat{\mathcal{E}}_{\V_\varepsilon^d}\left(\frac{\mu}{\varepsilon^{d}}\right) - \mathcal{E}_{\mathbb{R}^d}(\mu)\right| \le C_{d,p,\mu}\varepsilon^{\frac{d-2}{2}(2^*-p)}.
    $$
    \item For every positive ground state $u_\varepsilon$ of $\hat{E}_{\Vd}$ in $H^1_{\frac{\mu}{\varepsilon^{d}}}(\Vd)$, there exists $x_\varepsilon\in \R^d$ such that,
    \[
   \bigl(\hat{\mathcal{A}}u_\varepsilon\bigr)(\cdot+x_\varepsilon) \xrightarrow{\varepsilon \to 0}  \varphi_\mu, \quad \text{in } H^1(\mathbb{R}^d),
\]
where $\varphi_\mu\in H^1_\mu(\R^d)$ is the unique positive ground state of $E_{\mathbb{R}^d}$ at mass $\mu$ which attains its maximum at the origin and $\hat{\mathcal{A}}$ is the extension operator from $\Vd$ to $\R^d$  introduced in Section \ref{sec3}. Furthermore, 
$$
\lim_{\varepsilon\to 0}\mathcal{L}_{\Vd}(u_\varepsilon)=\omega_\mu,
$$
where $\omega_\mu$ is the value of the parameter $\omega$ for which $\varphi_\mu$ solves \eqref{eqrd}.
\end{enumerate}
\end{theorem}

As in Theorem \ref{th2}, we are also interested in the asymptotic behavior of energy ground states of $\hat{E}_{\Vd}$ at mass $\mu/\varepsilon^d$ for $d\geq 2$ and $p>2+\frac{4}{d}$. Now, we consider the case that $d \ge 2$, $p > 2+\frac{4}{d}$ and $\mu> 0$. By adapting the arguments in \cite{HJT}, in the proof of Theorem \ref{thl2sup}, we show that for any fixed $\mu>0$ and for $\varepsilon>0$ small enough, there exists a positive ground state $u_\varepsilon$ of  $\hat{E}_{\V^d_\varepsilon}$ at mass $\frac{\mu}{\varepsilon^d}$. 
 However, the following theorem  shows that the family $\{\hat{\mathcal{A}}u_\varepsilon\}$ diverges in $H^1(\R^d)$ as $\varepsilon\to0$. Therefore, in contrast to Theorems \ref{th1}-\ref{th3}, $\{\hat{\mathcal{A}}u_\varepsilon\}$ cannot converge to any function in $H^1(\R^d)$.
\begin{theorem}\label{thl2sup}
       Let $d \ge 2$, $p > 2+\frac{4}{d}$ and $\mu> 0$ be fixed.  Then
\[
    \varepsilon^{d-1}\hat{\mathcal{E}}_{\V_\varepsilon^d}(\frac{\mu}{\varepsilon^d}) \rightarrow -\infty, \quad \text{as }\varepsilon \to 0,
\]
and, for every $\varepsilon > 0$ small enough, and for any positive ground state $u_\varepsilon$ of $\hat{E}_{ \V_\varepsilon^d}$ in $H^1_{\frac{\mu}{\varepsilon^d}}(\V_\varepsilon^d)$, 
\[
  \hat{\mathcal{A}}u_\varepsilon \xrightarrow{\varepsilon \to 0} +\infty, \quad \text{in } H^1(\mathbb{R}^d),
\]
where $\hat{\mathcal{A}}$ is the extension operator from $\Vd$ to $\R^d$  introduced in Section \ref{sec3}.
\end{theorem}

It is well known that, for every $p\in (2+\frac{4}{d},2^*)$, $\inf_{v \in H_\mu^1(\R^d)} E_{\mathbb{R}^d}(v)=-\infty$. Thus, roughly speaking, by Theorem \ref{thl2sup}, we show that, for $d \geq 2$ and $p\in (2+\frac{4}{d},2^*)$, the energy ground state problem on lattice graphs serves as an asymptotic approximation of energy ground state problem on $\R^d$ as $\varepsilon \to 0^+$.

One may also ask whether the conclusion of Theorem \ref{thl2sup} remains valid on metric grids introduced in \cite{Do2}, since it is straightforward to adapt the proof of Theorem \ref{th2} to the setting of metric grids. However, it appears that repeating the arguments of Theorem \ref{thl2sup} on metric grids is non-trivial. This difficulty arises because the equality $\norm{\nabla \hat{A}f_\varepsilon}^2_{L^2(\R^d)}=\varepsilon^{d-1}\|\nabla^d_\varepsilon f_{\varepsilon}\|_{\ell^2(\Vd)}^2$ in \eqref{eqth14} (see also \cite[Lemma 4.4]{Do2}) plays a crucial role in our arguments, while, on metric grids, by \cite[Eq. (39)]{Do2}, one only has
$$
\norm{\nabla\mathcal{A}u}^2_{L^2(\R^d)}\leq \varepsilon^{d-1}\norm{u'}^2_{L^2(\Gd)},\quad \forall u \in H^1(\Gd).
$$

We conclude the presentation of our main results concerning the singular limit problem on lattice graphs with the following remark.
\begin{remark}\label{re1}
    Let $d \geq 1$, $p \in (2, 2^*)$ and $\omega > 0$ be fixed. For $R>0$ and $\varepsilon>0$, let $B_R^d:=\{\w\in \V^d_1:\operatorname{dist}(\w,0)\leq R\}$ and
    $$
    \hat{\mathcal{N}}_{\omega,\varepsilon B_R^d}:=\{u \in \hat{\mathcal{N}}_{\omega,\Vd}:u(\vv)=0\,\,  \text{ for every }\,\, \vv \in \V^d_\varepsilon \setminus\varepsilon B_R^d\}.
    $$
    Then, it is easy to see that there exists $u_{\varepsilon,R}$ such that 
    $$
    \hat{J}_{\omega, \V_\varepsilon^d}(u_{\varepsilon,R})= \inf_{v \in \hat{\mathcal{N}}_{\omega,\varepsilon B_R^d}} \hat{J}_{\omega, \V_\varepsilon^d}(v)
    $$
    and, for every fixed $\varepsilon>0$, there exists $\{\vv_{\varepsilon, R}\}_{R>0}\subset \Vd$, such that, up to a subsequence, $u_{\varepsilon,R}(\cdot-\vv_{\varepsilon, R})\to u_\varepsilon$ in $H^1(\Vd)$ as $R\to+\infty$, where $u_\varepsilon$ is a positive ground state of $\hat{J}_{\omega, \V_\varepsilon^d}$ in $\hat{\mathcal{N}}_{\omega, \V_\varepsilon^d}$. Then, by Lemmas \ref{lemequiv}-\ref{lemh1gv} and \cite[Lemma 4.4]{Do2}, we know that, up to a subsequence,
    $\hat{\mathcal{A}}\left(u_{\varepsilon,R}(\cdot-\vv_{\varepsilon, R})\right)\to \hat{\mathcal{A}}u_\varepsilon$ in $H^1(\R^d)$ as $R\to+\infty$. Therefore, we conclude that there exists $\{x_{\varepsilon_n,R_n}\}\subset \R^d$ such that $\varepsilon_n \to 0$ and $R_n\to +\infty$ as $n \to +\infty$, and $\hat{\mathcal{A}}u_{\varepsilon_n,R_n}(\cdot-x_{\varepsilon_n, R_n})\to \varphi_\omega$ in $H^1(\R^d)$ as $n\to+\infty$, where $\varphi_\omega \in \mathcal{N}_{\omega, \mathbb{R}^d}$ is the unique positive ground state of $J_{\omega, \mathbb{R}^d}$ which attains its maximum at the origin. In conclusion, we show that, up to a subsequence, extensions of action ground states $u_{\varepsilon,R}$ on the ball $B_R^d$ (which also implies that $u_{\varepsilon,R}$ has finite support) converge to the action ground state on $\R^d$ as $\varepsilon\to 0$ and $R \to +\infty$. The same conclusion holds for the energy ground states on lattice graphs under the assumptions of Theorem \ref{th3} and we believe that this result will be very useful for numerical approximations of ground states on $\R^d$. 
\end{remark}

As a by-product of the arguments developed to prove Theorem \ref{th3},
we obtain the following multiplicity result for mass constrained critical points of the energy at large masses on $\V_\varepsilon^d$. For the sake of brevity, we present the result only for $\V_1^d$.
\begin{theorem}\label{th4}
    Let $d \geq 2$ and $p \in \left(2+\frac{4}{d},2^*\right)$. Then there exists $\{\mu_n\} \subset \R$ with $\mu_n \to +\infty$ as $n \to +\infty$, such that $\hat{E}_{\V^d_1}$ has at least two distinct positive critical points $u_{n,1}$ and $u_{n,2}$ in $ H^1_{\mu_n}(\V^d_1)$ for every $n\in \mathbb{N}^+$, where $u_{n,1}$ is a positive ground state of $\hat{E}_{\V^d_1}$ in $H^1_{\mu_n}(\V^d_1)$ and $u_{n,2}$ is a positive ground state of $\hat{J}_{\omega_n,\V^d_1}$ for some $\omega_n>0$. Moreover, we have 
    \begin{equation}\label{eqth4}
            \hat{E}_{\V^d_1}(u_{n,2})>\hat{E}_{\V^d_1}(u_{n,1})=\hat{\mathcal{E}}_{\V^d_1}(\mu_n), \quad \forall n\in \mathbb{N}^+,
    \end{equation}
    and
    $$
    \mathcal{L}_{\V^d_1}(u_{n,1})\to+\infty \quad 
\text{and} \quad \mathcal{L}_{\V^d_1}(u_{n,2})\to0,\quad \text{as }n\to+\infty.
    $$
    \end{theorem}
The inequality \eqref{eqth4} implies that, for every $d \geq 2$ and $p \in \left(2+\frac{4}{d},2^*\right)$, there exists $\{\mu_n\} \subset \R$ with $\mu_n \to +\infty$ as $n \to +\infty$ such that, for every $n \in \mathbb{N}^+$, there exists an action ground state that is not an energy ground state with  mass $\mu_n$. 
This result should be compared with \cite[Theorem 1.3]{DST} and \cite[Theorem 1.6]{JeanjenaLu2}, which prove that energy ground states are necessarily action ground states. Thus, Theorem \ref{th4} shows that the equivalence between the action and energy ground state problems fails on lattice graphs. A similar result concerning the equivalence between action and energy ground states on metric graphs was established in \cite[Theorem 1.4]{Do3}. There, the author considered the problem on compact metric graphs, noncompact metric graphs with finitely many edges, and $\mathbb{Z}$-periodic metric graphs, but did not cover the case of metric grids $\G^d_1$.

We now compare Theorem \ref{th4} with  \cite[Proposition 2.4]{Do2}, which provides multiplicity results for prescribed mass critical points of the energy functional on metric grids $\G^d_1$, without proving \eqref{eqth4}. On the one hand, by repeating the argument in the proof of Theorem \ref{th4}, one can also prove that \eqref{eqth4} holds on $\G^d_1$.  On the other hand, since positive energy ground states exist for all $d\geq 2$, $p>2+\frac{4}{d}$ and $\mu>0$ large enough (see \cite{HJT}), Theorem \ref{th4} holds for all $d\geq 2$ and $p \in (2+\frac{4}{d},2^*)$, while \cite[Proposition 2.4]{Do2} holds for $d\in \{2,3\}$ and $p \in (2+\frac{4}{d},6)$ or $d \geq 4$ and $p \in (2+\frac{4}{d},2^*)$, where the energy ground states do not exist for any $p>6$.

\subsubsection{Best constants in Gagliardo--Nirenberg inequalities}
Employing a strategy analogous to the analysis of the singular limit problem on lattice graphs, we now turn our attention to the Gagliardo--Nirenberg type inequalities  on lattice graphs (see \eqref{eqgnv} and \eqref{eqsobov}).
For $d \geq 1$ and $q \in (2,2^*)$, let
$$\begin{aligned}
&Q_{q,\V^d_1}(u):=\frac{\norm{u}^q_{\ell^q(\V^d_1)}}{\norm{u}^{d+(2-d)\frac{q}{2}}_{\ell^2(\V^d_1)}\norm{\nabla^d_1 u}^{(\frac{q}{2}-1)d}_{\ell^2(\V^d_1)}},\quad  &u \in H^1(\V^d_1)\setminus\{0\},
\\
&Q_{q,\R^d}(v):=\frac{\norm{v}^q_{L^q(\R^d)}}{\norm{v}^{d+(2-d)\frac{q}{2}}_{L^2(\R^d)}\norm{\nabla v}^{(\frac{q}{2}-1)d}_{L^2(\R^d)}},\quad  &v \in H^1(\R^d)\setminus\{0\},
\end{aligned}
$$
and denote by
$$
K_{q,\V^d_1}:=\sup_{u \in H^1(\V^d_1)\setminus\{0\}}Q_{q,\V^d_1}(u),\quad K_{q,\R^d}:=\sup_{v \in H^1(\R^d)\setminus\{0\}}Q_{q,\R^d}(v)
$$
the best constants  in the $d$-dimensional Gagliardo--Nirenberg inequalities on $\V^d_1$ (see \eqref{eqgnv} in Section 2) and on $\R^d$ (see, e.g., \cite[Eq. (4)]{Do2}).

Let $D^{1,2}(\V_\varepsilon^d)$ denote the completion of $C_c(\V_\varepsilon^d)$  with respect to the norm 
$\norm{\nabla ^d_\varepsilon u}_{L^2(\V_\varepsilon^d)}$,  where  $C_{c}(\V_\varepsilon^d)$ is the set of all functions with finite support.
For $d \geq 3$ and $q =2^*$, let
$$\begin{aligned}
&Q_{2^*,\V^d_1}(u):=\frac{\norm{u}^{2^*}_{\ell^{2^*}(\V^d_1)}}{\norm{\nabla^d_1 u}^{2^*}_{\ell^2(\V^d_1)}},\quad  &u \in D^{1,2}(\V_1^d)\setminus\{0\},
\\
&Q_{{2^*},\R^d}(v):=\frac{\norm{v}^{2^*}_{L^{2^*}(\R^d)}}{\norm{\nabla v}^{2^*}_{L^2(\R^d)}},\quad  &v \in D^{1,2}(\R^d)\setminus\{0\},
\end{aligned}
$$
and denote by
$$
K_{{2^*},\V^d_1}:=\sup_{u \in D^{1,2}(\V^d_1)\setminus\{0\}}Q_{{2^*},\V^d_1}(u),\quad K_{{2^*},\R^d}:=\sup_{v \in D^{1,2}(\R^d)\setminus\{0\}}Q_{{2^*},\R^d}(v)
$$
the best constants  in the $d$-dimensional Sobolev inequalities on $\V^d_1$ (see \eqref{eqsobov} and  \cite[Theorem 3.6]{hua2015time}) and on $\R^d$. 
By adapting the argument in the proof of \cite[Proposition 2.6]{Do2}, we first establish a sufficient condition under which the best constant $K_{q,\V^d_1}$ is attained.
\begin{theorem}\label{th5}   
If $d\geq 1$ and $q \in (2,2^*)$, or $d\geq 3$ and $q =2^*$, then 
    \begin{equation}\label{eqth51}
            K_{q,\V^d_1}\geq K_{q,\R^d}.
    \end{equation}
    Furthermore, for every  $d\geq 1$ and $q \in (2,2^*)$, if there exists $u \in H^1(\V^d_1)$ such that 
    $
    Q_{q,\V^d_1}(u) \geq K_{q,\R^d},
    $
    then $K_{q,\V^d_1}$ is attained.
\end{theorem}
Now, we present our main result on the critical Lane--Emden equation on lattice graphs.
\begin{theorem}\label{th6}
If \(d\geq 3\), then 
\begin{equation}\label{eqk2*}
K_{2^*,\V_1^d}>K_{2^*,\mathbb R^d},
\end{equation}
and the best constant \(K_{2^*,\V_1^d}\) is attained by a positive function \(u\in D^{1,2}(\V_1^d)\), which satisfies 
\[
Q_{2^*,\V_1^d}(u)=K_{2^*,\V_1^d}
\]
and solves
\begin{equation}\label{eqcle}
    -\Delta_1^d u=u^{2^*-1}
\quad\text{in }\V_1^d.
\end{equation}
\end{theorem}
In view of \cite[Proof~II of Theorem~1]{hua2021existence}, the main difficulty in proving Theorem~\ref{th6} lies in ruling out the vanishing of a minimizing sequence. In \cite[Theorem~1]{hua2021existence}, the authors prove that, for every $q>2^*$, the best constant in the discrete Sobolev inequality on $\mathbb Z^N$ is attained. A crucial ingredient is the nonvanishing estimate established in \cite[Lemma~13]{hua2021existence}. More precisely, given a minimizing sequence $\{u_n\}$ normalized by $\|u_n\|_{\ell^q(\V_1^d)}=1$, one can choose an exponent $q'$ such that
\[
2^*<q'<q
\]
and use interpolation together with the Sobolev inequality to obtain
\[
1=\|u_n\|_{\ell^q(\V^d_1)}^q
\leq
\|u_n\|_{\ell^{q'}(\V^d_1)}^{q'}\|u_n\|_{\ell^\infty(\V^d_1)}^{q-q'}
\leq
C\|u_n\|_{D^{1,2}(\V^d_1)}^{q'}\|u_n\|_{\ell^\infty(\V^d_1)}^{q-q'}.
\]
Since $\{u_n\}$ is bounded in $D^{1,2}(\mathbb Z^N)$ and $q-q'>0$, this yields a uniform positive lower bound for $\|u_n\|_\infty$, thereby excluding the vanishing case. This argument is genuinely supercritical: when $q=2^*$, there is no exponent $q'$ satisfying $2^*<q'<q$, while taking $q'=q$ eliminates the factor involving $\|u_n\|_{\ell^\infty(\V^d_1)}$. Hence the same interpolation argument provides no information preventing vanishing in the critical case. 

 A key step in the proof of Theorem \ref{th6} is to construct suitable test functions on $\V_1^d$ and prove the strict comparison \eqref{eqk2*}.
This strict inequality allows us to rule out the vanishing of minimizing sequences by a singular limit argument. Indeed, if a minimizing sequence $\{u_n\}$ of $K_{2^*,\V_1^d}^{-\frac{2}{2^*}}$ with $\norm{u_n}_{\ell^{2^*}(\V^d_1)}=1$ were vanishing, then its $\ell^\infty$ norm would tend to zero. The singular limit result would imply that $\norm{\hat{\mathcal{A}}u_n}_{L^{2^*}(\mathbb R^d)}\to 1$, and thus,
\[
K_{2^*,\V_1^d}\leq K_{2^*,\mathbb R^d},
\]
contradicting the strict inequality \eqref{eqk2*}. Therefore, vanishing cannot occur; after suitable translations, the minimizing sequence has a nontrivial limit, which leads to the attainment of the best constant.

The following theorem is a direct consequence of Theorems \ref{th5} and \ref{th6}.
\begin{theorem}\label{th7}
If \(d\geq 3\), then there exists $\delta=\delta(d)>0$ such that, for every $q\in (2^*-\delta,2^*)$,
\[
K_{q,\V_1^d}>K_{q,\mathbb R^d},
\]
and the best constant \(K_{q,\V_1^d}\) is attained by a positive function \(u\in H^{1}(\V_1^d)\), which satisfies 
\[
Q_{q,\V_1^d}(u)=K_{q,\V_1^d}
\]
and, for some $\lambda>0$, solves
\[
-\Delta_1^d u+ \lambda u=u^{q-1}
\quad\text{in }\V_1^d.
\]
\end{theorem}
It remains open whether the conclusions of Theorem \ref{th7} continue to hold for $d=1,2$ and $q>2$, as well as for $d\geq3$ and
$q\in(2,2^*-\delta]$. In particular, it is unknown whether
\[
K_{q,\V_1^d}>K_{q,\mathbb R^d}
\]
and whether the best constant $K_{q,\V_1^d}$ is attained in these
cases.

In Theorems~\ref{thmg1} and~\ref{thmg2}, we extend Theorems~\ref{th6} and~\ref{th7} to the metric grid $\G^d_1$.

\subsection{Organization of the paper}
The rest of the paper is organized as follows. In Section \ref{sec2}, we establish connections between Sobolev spaces and ground states for (NLS) equation on lattice graphs and on metric grids. In Section \ref{sec3}, we study the asymptotic behavior of action ground states of $\hat{J}_{\omega, \V_\varepsilon^d}$ as $\varepsilon \to 0$ and present the proofs of Theorems \ref{th1} and \ref{th2}. In Section \ref{sec4}, we study the asymptotic behavior of energy ground states of $\hat{E}_{\omega, \V_\varepsilon^d}$ at mass $\mu/\varepsilon^d$ as $\varepsilon \to 0$ and the multiplicity of critical points of $\hat{E}_{\V^d_1}$ constrained on $H^1_\mu(\V^d_1)$, and provide the proofs of Theorems \ref{th3}-\ref{th4}. Section \ref{sec5} is devoted to studying the sufficient condition under which the best constant $K_{q,\V^d_1}$ is attained and to the proof of Theorem \ref{th5}. In Section \ref{seccle}, we establish the attainability of $K_{2^*,\V^d_1}$ and $K_{q,\V^d_1}$ with $q\nearrow2^*$,  and provide the proofs of Theorems \ref{th6} and \ref{th7}. In Section \ref{sec6}, we raise the question of whether Theorem \ref{th4} can be improved,  and formulate further open problems concerning the singular limit behavior of solutions to the Sobolev critical (NLS) equation on lattice graphs. In Appendix \ref{app}, we further extend Theorems \ref{th6} and \ref{th7} to the metric grid $\G^d_1$ and provide the proofs of Theorems \ref{thmg1} and \ref{thmg2}.

In the following, $C$ and $C_i$ denote positive constants independent of $\varepsilon>0$, which may vary from line to line. 
  
\section{Connections between lattice graphs and metric grids}
\label{sec2}
In this section, we establish connections between Sobolev spaces and ground states for (NLS) equation on lattice graphs and on metric grids. To this end, we first recall some basic settings of metric grids. For every $\varepsilon > 0$, let $\mathcal{G}_\varepsilon^d = (\mathbb{V}_\varepsilon^d, \mathbb{E}_\varepsilon^d)$ be the $d$-dimensional metric grid with edge length $\varepsilon$ given by the subset of $\mathbb{R}^d$ with vertices on $\varepsilon\mathbb{Z}^d$ and edges between every couple of vertices at distance $\varepsilon$.
Each edge $e \in \Ed$ is identified with a closed and bounded interval $I_e=[0, \varepsilon]$.

Consistently, a function $u: \G_\varepsilon^d \to \mathbb{R}$ is actually a family of functions $(u_\e)$, where $u_\e: I_\e \to \mathbb{R}$ is the restriction of $u$ to the edge $\e \in \E_\varepsilon^d$. The usual $L^p$ space on the metric grid  $\G_\varepsilon^d$ is defined as follows
$$
L^p(\G_\varepsilon^d):=\bigoplus_{\e \in \E_\varepsilon^d} L^p(I_\e),
$$
endowed with the norm
$$
\norm{u}_{L^p(\G_\varepsilon^d)}^p := \sum_{\e \in \mathrm{E}_\varepsilon^d}\norm{u_\e}_{L^p(I_\e)}^p, \quad \text { if } p \in[1, \infty) \quad \text { and } \quad \norm{u}_{L^{\infty}(\G_\varepsilon^d)}:=\sup_{\e \in \mathrm{E}_\varepsilon^d}\norm{u_\e}_{L^{\infty}(I_\e)}.
$$
The inner product in $L^2(\mathcal{G}_\varepsilon^d)$ is
\begin{equation*}
(u,v)_2= \int_{\G_\varepsilon^d} uv\,dx.
\end{equation*}
If $u:\G_\varepsilon^d \to\mathbb{R}$ is continuous, then we write $u \in C(\G_\varepsilon^d)$. The Sobolev space $H^1(\G_\varepsilon^d)$ consists of the set of continuous functions $u : \G_\varepsilon^d \to \mathbb{R}$ such that $u_\e \in  H^1([0, \varepsilon])$ for every edge $\e\in \E_\varepsilon^d$ and $\underset{\e \in \E_\varepsilon^d}\sum \left(\norm{u'}^2_{L^2(I_\e)}+\norm{u}^2_{L^2(I_\e)}\right)< \infty$; the norm in $H^1(\G_\varepsilon^d)$ is defined as
\begin{equation*}
    \norm{u}^2_{H^1(\G_\varepsilon^d)} =\norm{u'}^2_{L^2(\G_\varepsilon^d)} + \norm{u}^2_{L^2(\G_\varepsilon^d)},
\end{equation*}
and the inner product in this space is
\begin{equation*}
(u,v)= \int_{\G_\varepsilon^d} (u'v' + uv)\,dx.
\end{equation*}
For any $u \in H^1(\G_\varepsilon^d)$, define $\hat{u}\in C(\V_\varepsilon^d)$ by $\hat{u}:=u|_{\V_\varepsilon^d}$.\\
For every $u \in H^1(\G_\varepsilon^d)$, define 
\begin{equation*}
    \norm{u}^2_{\hat{H}^1(\G_\varepsilon^d)} = \norm{u'}^2_{L^2(\G_\varepsilon^d)}+\sum_{\vv \in \V_\varepsilon^d}  \abs{\hat{u}(\vv)}^2.
\end{equation*}
For any $f \in H^1(\V_\varepsilon^d)$, define $\tilde{f}\in C(\G_\varepsilon^d)$ by
$$
\tilde{f}_\e(x):= \frac{f(\w)-f(\vv)}{\varepsilon}x+f(\vv), \quad \forall \e=(\vv,\w) \in \E_\varepsilon^d,
$$
where, for any edge $\e=(\vv,\w) \in \E_\varepsilon^d$, $e$ is identified with interval $I_e=[0, \varepsilon]$, $\vv$ is identified with $0$, and $\w$ is identified with $\varepsilon$.

By repeating the arguments in \cite[Lemma 2.3]{BBDT2},  we obtain, for all $d \geq 1$,  the following estimates for concentrated nonlinearities on $\Gd$.
\begin{lemma}\label{lemlp}
Let $p \geq 2$ and $\varepsilon > 0$. Then, there exists a constant $S_{d,p}>0$, depending only on $d$ and $p$ such that
\begin{equation}
\left| d\varepsilon \sum_{\vv \in \V^d_\varepsilon} |\hat{u}(\vv)|^p - \|u\|_{L^{p}(\Gd)}^p \right| \leq S_{d,p} \varepsilon \|u\|_{L^{2p-2}(\Gd)}^{p-1} \|u'\|_{L^2(\Gd)}, \quad \forall u \in H^1(\Gd).
\end{equation}
\end{lemma}
To show that the extension $\tilde{f}$ of $f \in H^1(\V_\varepsilon^d)$ belongs to $H^1(\G_\varepsilon^d)$, we need the following lemma.
\begin{lemma}\label{lemequiv}
 If $u \in C(\G_\varepsilon^d)$ satisfies $u_\e \in  H^1([0,\varepsilon])$ for every edge $\e\in \E_\varepsilon^d$ and $\norm{u}_{\hat{H}^1(\G_\varepsilon^d)}< \infty$, then $u \in H^1(\G_\varepsilon^d)$ and 
 \begin{equation}\label{eql2}
    \int_{\Gd}\abs{u}^2\,dx \leq 2d\varepsilon\sum_{\vv \in \Vd} \abs{\hat{u}(\vv)}^2 + 2d^2\varepsilon^2\int_{\Gd}\abs{u'}^2\,dx.
 \end{equation}
    In addition,  $\norm{u}_{\hat{H}^1(\G_\varepsilon^d)}$ and $\norm{u}_{H^1(\G_\varepsilon^d)}$ are equivalent on $ H^1(\G_\varepsilon^d)$.
\end{lemma}
\begin{proof}
     Let $u \in C(\G_\varepsilon^d)$ satisfy $u_\e \in  H^1([0, \varepsilon])$ for every edge $\e\in \E_\varepsilon^d$ and $\norm{u}_{\hat{H}^1(\G_\varepsilon^d)}< \infty$. We first prove \eqref{eql2}. For any $\vv \in \Vd$ and $R>0$, set $B_R(\vv):=\{x \in \Gd: \abs{x-\vv}\leq R \}$. For any fixed $\vv \in \Vd$, by H\"older's inequality, we have 
        $$
        \begin{aligned}
        \int_{B_{\frac{\varepsilon}{2}}(\vv)}\abs{u}^2\,dx\leq \int_{B_{\frac{\varepsilon}{2}}(\vv)}\,\max_{x \in \overline{B_{\frac{\varepsilon}{2}}(\vv)}}\abs{u(x)}^2\,dx&\leq \int_{B_{\frac{\varepsilon}{2}}(\vv)}\left(\abs{u(\vv)}+ \int_{B_{\frac{\varepsilon}{2}}(\vv)}\abs{u'(y)}\, dy\right)^2\,dx\\&\leq 2\int_{B_{\frac{\varepsilon}{2}}(\vv)}\left(\abs{u(\vv)}^2+ \left(\int_{B_{\frac{\varepsilon}{2}}(\vv)}\abs{u'(y)}\, dy\right)^2\right)\,dx\\
        &\leq 2d\varepsilon \abs{u(\vv)}^2+2d\varepsilon\left(\int_{B_{\frac{\varepsilon}{2}}(\vv)}\abs{u'(y)}\, dy\right)^2 \\
        &\leq 2d\varepsilon \abs{u(\vv)}^2 + 2d^2\varepsilon^2\int_{B_{\frac{\varepsilon}{2}}(\vv)}\abs{u'}^2\, dx,
        \end{aligned}
        $$
    hence,
    $$
    \int_{\Gd}\abs{u}^2\,dx =\sum_{\vv \in \Vd} \int_{B_{\frac{\varepsilon}{2}}(\vv)}\abs{u}^2\,dx\leq 2d\varepsilon\sum_{\vv \in \Vd} \abs{u(\vv)}^2 + 2d^2\varepsilon^2\int_{\Gd}\abs{u'}^2\,dx.
    $$
Since $\norm{u}_{\hat{H}^1(\G_\varepsilon^d)}< \infty$, we conclude that $\int_{\Gd}\abs{u}^2\,dx<+\infty$, and thus, $u \in H^1(\G_\varepsilon^d)$. Moreover, by \eqref{eql2} and Lemma \ref{lemlp}, we know that $\norm{u}_{\hat{H}^1(\G_\varepsilon^d)}$ and $\norm{u}_{H^1(\G_\varepsilon^d)}$ are equivalent on $ H^1(\G_\varepsilon^d)$
\end{proof}
Now, we can establish connections between the Sobolev spaces on lattice graphs and on metric grids.
It is easy to see that, for any $f \in H^1(\V_\varepsilon^d)$,
\begin{equation}\label{eqftilde}
\norm{\nabla_\varepsilon^d f}^2_{\ell^2(\V_\varepsilon^d)} =\frac{1}{2\varepsilon} \sum_{\vv \in \V_\varepsilon^d} \sum_{\w \sim \vv}\abs{f(\w)-f(\vv)}^2=\int_{\Gd}\abs{\tilde{f}'}^2\, dx=\norm{\tilde{f}'}^2_{L^2(\G_\varepsilon^d)}.
\end{equation}
Thus, we obtain the following result.
\begin{lemma}\label{lemh1gv}
If $u \in H^1(\G_\varepsilon^d)$, then $\hat{u} \in H^1(\V_\varepsilon^d)$ and 
$$
\norm{\nabla_\varepsilon^d \hat{u}}^2_{\ell^2(\V_\varepsilon^d)}\leq \norm{u'}^2_{L^2(\G_\varepsilon^d)}.
$$
In addition, if $f \in H^1(\V_\varepsilon^d)$, then $\tilde{f} \in H^1(\G_\varepsilon^d)$ and 
$$
\norm{\tilde{f}'}^2_{L^2(\G_\varepsilon^d)}=\norm{\nabla_\varepsilon^d f}^2_{\ell^2(\V_\varepsilon^d)}.
$$
\end{lemma}
\begin{proof}
Firstly, let $u \in H^1(\G_\varepsilon^d)$. Then, by Lemma \ref{lemlp} and the equivalence of spaces $H^1(\V_\varepsilon^d)$ and $\ell^2(\V_\varepsilon^d)$ (see, e.g., \cite[Lemma 2.1]{HJ}), we know that $\hat{u} \in H^1(\V_\varepsilon^d)$. Moreover, by \eqref{eqftilde} and \cite[Eq. (26)]{Do2}, we obtain
    $$
    \norm{\nabla_\varepsilon^d \hat{u}}^2_{\ell^2(\V_\varepsilon^d)}=\norm{\tilde{\hat{u}}'}^2_{L^2(\G_\varepsilon^d)}\leq \norm{u'}^2_{L^2(\G_\varepsilon^d)}.
    $$
   Now, let $f \in H^1(\V_\varepsilon^d)$. Then, by \eqref{eqftilde}, we know that $\norm{\tilde{f}'}^2_{L^2(\G_\varepsilon^d)}=\norm{\nabla_\varepsilon^d f}^2_{\ell^2(\V_\varepsilon^d)}$ and $\tilde{f}$ satisfies $\tilde{f}_\e \in  H^1([0,\ell_\e])$ for every edge $\e\in \E_\varepsilon^d$ and $\norm{\tilde{f}}_{\hat{H}^1(\G_\varepsilon^d)}< \infty$. Hence, by Lemma \ref{lemequiv}, we conclude that  $\tilde{f} \in H^1(\G_\varepsilon^d)$.
\end{proof}
As a direct consequence of Lemma \ref{lemh1gv}, we prove the following Gagliardo--Nirenberg type inequalities for concentrated nonlinearities on $\Gd$,  which play a key role in the proofs of Theorems \ref{th1} and \ref{th3}.
\begin{lemma}
    Fix  $\varepsilon>0$. For any $d \geq 1$, the following inequality holds
         \begin{equation}\label{eqgng1}
    \sup_{\vv \in \V_\varepsilon^d}{\abs{\hat{u}(\vv)}} \leq  \sqrt{2}(d\varepsilon)^{1/4}\|u'\|_{L^2(\G_\varepsilon^d)}^{1/2}\left(\sum_{\vv \in \V_\varepsilon^d}  \abs{\hat{u}(\vv)}^2\right)^{1/4},\quad \forall u \in H^1(\G_\varepsilon^d).
\end{equation}
   Moreover,  for any $d \geq 1$ and  any $p \in [2,2^*)$, there exists a constant $C_{d,p}>0$, depending only on $d$ and $p$ such that  the following inequality holds
    \begin{equation}\label{eqgng}
    \sum_{\vv \in \V_\varepsilon^d}  \abs{\hat{u}(\vv)}^p \leq  C_{d,p}(\varepsilon^{1/2}\|u'\|_{L^2(\G_\varepsilon^d)})^{\beta_pp}\left(\sum_{\vv \in \V_\varepsilon^d}  \abs{\hat{u}(\vv)}^2\right)^{(1-\beta_p)p/2},\quad \forall u \in H^1(\G_\varepsilon^d),
\end{equation}
where $\beta_p := d \left( \frac{1}{2} - \frac{1}{p} \right)$.
\end{lemma}
\begin{proof}
By Lemma \ref{lemh1gv}, it suffices to prove that the following inequalities hold
\begin{equation}\label{eqgn1v}
    \|u\|_{\ell^\infty(\V_\varepsilon^d)} \leq  \sqrt{2}(d\varepsilon)^{1/4}\|\nabla_\varepsilon^d u\|_{\ell^2(\V_\varepsilon^d)}^{1/2}\|u\|^{1/2}_{\ell^2(\V_\varepsilon^d)},\quad \forall u \in H^1(\V_\varepsilon^d),
\end{equation}
and  there exists a constant $C_{d,p}>0$, depending only on $d$ and $p$ such that
\begin{equation}\label{eqgnv}
    \|u\|_{\ell^p(\V_\varepsilon^d)}^p \leq  C_{d,p}(\varepsilon^{1/2}\|\nabla_\varepsilon^d u\|_{\ell^2(\V_\varepsilon^d)})^{\beta_p p}\|u\|^{(1 - \beta_p)p}_{\ell^2(\V_\varepsilon^d)},\quad \forall u \in H^1(\V_\varepsilon^d),
\end{equation}
where $\beta_p := d \left( \frac{1}{2} - \frac{1}{p} \right)$.

We first prove \eqref{eqgn1v}. Recall that $C_{c}(\V_\varepsilon^d)$ is the set of all functions with finite support.  For $u \in C_c(\Vd)$, let $\vv_0,\w_0$ be such that $\abs{u(\vv_0)}= \|u\|_{\ell^\infty(\V_\varepsilon^d)}$ and  $u(\w_0)=0$. Denote the path connecting $\vv_0$ and $\w_0$ by $\{\vv_i\}_{i=0}^k$, where $\vv_k=\w_0$. Then
$$
 \|u\|_{\ell^\infty(\V_\varepsilon^d)}= \abs{u(\vv_0)} \leq \sum_{i=0}^{k-1} \abs{u(\vv_{i})-u(\vv_{i+1})} \leq \sum_{(\vv,\w) \in \Ed} \abs{u(\vv)-u(\w)},
$$
where we identify $(\vv,\w) \in \Ed$ and $(\w,\vv) \in \Ed$ as the same element in $\Ed$. Thus, by H\"older's inequality, we obtain
$$
\begin{aligned}
    \|u\|_{\ell^\infty(\V_\varepsilon^d)}^2 &\leq \sum_{(\vv,\w) \in \Ed} \left|\abs{u(\vv)}^2-\abs{u(\w)}^2\right|\\
    &=\sum_{(\vv,\w) \in \Ed} \left|{u(\vv)}-{u(\w)}\right|\left|\vert{u(\vv)}\vert+\vert{u(\w)}\vert\right|\\
    &\leq \left(\sum_{(\vv,\w) \in \Ed} \left|{u(\vv)}-{u(\w)}\right|^2\right)^{1/2}\left(\sum_{(\vv,\w) \in \Ed} \left(\abs{u(\vv)}+\abs{u(\w)}\right)^2\right)^{1/2}\\
    &\leq\left(\frac{1}{2} \sum_{x \in \V_\varepsilon^d} \sum_{y \sim x}\abs{u(y)-u(x)}^2\right)^{1/2}\left(2\sum_{(\vv,\w) \in \Ed} \left(\abs{u(\vv)}^2+\abs{u(\w)}^2\right)\right)^{1/2}\\
    &=\sqrt{\varepsilon}\|\nabla_\varepsilon^d u\|_{\ell^2(\V_\varepsilon^d)}\cdot 2\sqrt{d}\|u\|_{\ell^2(\V_\varepsilon^d)}=2\sqrt{d\varepsilon}\|\nabla_\varepsilon^d u\|_{\ell^2(\V_\varepsilon^d)}\|u\|_{\ell^2(\V_\varepsilon^d)},
\end{aligned}
$$
which implies \eqref{eqgn1v} since $H^1(\V_\varepsilon^d)$ is the completion of $C_c(\V_\varepsilon^d)$ under the norm $\|u\|_{H^{1}(\V_\varepsilon^d)}$.

By \eqref{eqgn1v}, we obtain \eqref{eqgnv} with $d = 1$. For the proof of \eqref{eqgnv} with $d \geq 2$, one can refer to \cite[Eq. (4.11)]{Weinstein} or derive it using the discrete Sobolev inequality (see, e.g., \cite[Theorem 3.6]{hua2015time}) and an interpolation inequality.
\end{proof}
\begin{lemma}\label{lemsobo}
    Fix $\varepsilon>0$. For any $d \geq 3$, there exists a constant $C_{d,2^*}$, depending only on $d$ such that the following inequalities hold
\begin{equation}\label{eqsobo}
   \sum_{\vv \in \V_\varepsilon^d}  \abs{\hat{u}(\vv)}^{2^*}\leq C_{d,2^*}(\varepsilon^{1/2}\|u'\|_{L^2(\G_\varepsilon^d)})^{2^*},\quad \forall u \in H^1(\G_\varepsilon^d),
\end{equation}
and
\begin{equation}\label{eqlinfty}
    \norm{u}_{L^\infty(\Gd)}\leq (C_{d,2^*}^{1/2^*}+2^{-1/2})\varepsilon^{1/2}\|u'\|_{L^2(\G_\varepsilon^d)},\quad \forall u \in H^1(\G_\varepsilon^d).
\end{equation}
\end{lemma}
\begin{proof}
Let $d \geq 3$. The Sobolev type inequality \eqref{eqsobo} follows directly from Lemma \ref{lemh1gv} and the following discrete Sobolev inequality (see, e.g., \cite[Theorem 3.6]{hua2015time})
\begin{equation}\label{eqsobov}
     \norm{f}^{2^*}_{\ell^{2^*}(\V^d_1)}\leq C_{d,2^*}\norm{\nabla^d_1 f}^{2^*}_{\ell^{2}(\V^d_1)},\quad \forall f \in C_{c}(\V^d_1).
 \end{equation}

   Now, we prove \eqref{eqlinfty}. By \eqref{eqsobo}, we have
    \begin{equation}\label{eqlemsobo1}
\max_{\vv \in \Vd}\abs{\hat{u}(\vv)} \leq \left(\sum_{\vv \in \V_\varepsilon^d}  \abs{\hat{u}(\vv)}^{2^*}\right)^{1/2^*} \leq C_{d,2^*}^{1/2^*}\varepsilon^{1/2}\|u'\|_{L^2(\G_\varepsilon^d)},\quad \forall u \in H^1(\G_\varepsilon^d).
   \end{equation}
   For any $x\in \Gd$, we know that there exists $\vv \in \Vd$ such that $x\in B_{\frac{\varepsilon}{2}}(\vv)$, where $B_\frac{\varepsilon}{2}(\vv):=\{x \in \Gd: \abs{x-\vv}\leq \frac{\varepsilon}{2} \}$. Thus, by \eqref{eqlemsobo1}, for any $x\in \Gd$ and $\vv \in \Vd$ such that $x\in B_{\frac{\varepsilon}{2}}(\vv)$, we obtain that
   $$
   \abs{u(x)}\leq \abs{u(\vv)}+\int_{B_\frac{\varepsilon}{2}(\vv)\cap\e}\abs{u'(y)}\,dy\leq \max_{\vv \in \Vd}\abs{\hat{u}(\vv)}+\sqrt{\varepsilon/2}\|u'\|_{L^2(\G_\varepsilon^d)}\leq (C_{d,2^*}^{1/2^*}+2^{-1/2})\varepsilon^{1/2}\|u'\|_{L^2(\G_\varepsilon^d)},
   $$
where $\e \in \Ed$ such that $\e\succ \vv$ and $x$ belonging to $\e$. By the above inequality, we complete the proof of Lemma \ref{lemsobo}.
\end{proof}

For every $p >2$ and every $\omega > 0$, we introduce the action functional $\tilde{J}_{\omega, \G_\varepsilon^d} : H^1(\G_\varepsilon^d) \to \mathbb{R}$
\begin{equation*}
    \tilde{J}_{\omega, \G_\varepsilon^d}(u) := \frac{1}{2}\|u'\|_{L^2(\G_\varepsilon^d)}^2 + \frac{\varepsilon\omega}{2}\sum_{\vv \in \V_\varepsilon^d}  \abs{\hat{u}(\vv)}^2 - \frac{\varepsilon}{p}\sum_{\vv \in \V_\varepsilon^d}  \abs{\hat{u}(\vv)}^p
\end{equation*}
and the associated Nehari manifold
\begin{equation}\label{eqtilden}
    \begin{aligned}
    \tilde{\mathcal{N}}_{\omega, \G_\varepsilon^d} :&= \left\{ u \in H^1(\G_\varepsilon^d) \backslash \{0\} : \tilde{J}'_{\omega, \G_\varepsilon^d}(u)u = 0 \right\} \\
     &= \left\{ u \in H^1(\G_\varepsilon^d) \backslash \{0\} :  \|u'\|_{L^2(\G_\varepsilon^d)}^2 + \varepsilon\omega\sum_{\vv \in \V_\varepsilon^d}  \abs{\hat{u}(\vv)}^2 =\varepsilon\sum_{\vv \in \V_\varepsilon^d}  \abs{\hat{u}(\vv)}^p \right\}.\\
\end{aligned}
\end{equation}
Letting
\begin{equation}
 \label{eqming}
    \tilde{\mathcal{J}}_{\G_\varepsilon^d}(\omega) := \inf_{v \in \tilde{\mathcal{N}}_{\omega, \G_\varepsilon^d}} \tilde{J}_{\omega, \G_\varepsilon^d}(v)
\end{equation}
be the corresponding minimization problem, $u \in \tilde{\mathcal{N}}_{\omega, \G_\varepsilon^d}$ is called a ground state of $\tilde{J}_{\omega, \G_\varepsilon^d}$ if $\tilde{J}_{\omega, \G_\varepsilon^d}(u) = \tilde{\mathcal{J}}_{\G_\varepsilon^d}(\omega)$.
\begin{lemma}
\label{leminfa}
    For any $d \geq 1$, $\varepsilon>0$, $p > 2$ and $\omega >0$, $$\tilde{\mathcal{J}}_{\G_\varepsilon^d}(\omega)=\hat{\mathcal{J}}_{\V_\varepsilon^d}(\omega)>0.$$ 
In addition,  if $f\in  H^1(\V_\varepsilon^d)$ is a ground state of $\hat{J}_{\omega, \V_\varepsilon^d}$, then $\tilde{f} \in H^1(\G_\varepsilon^d)$ is a ground state of $\tilde{J}_{\omega, \G_\varepsilon^d}$. Moreover, if $u \in H^1(\Gd)$ is a ground state of $\tilde{J}_{\omega, \G_\varepsilon^d}$, then $\hat{u}\in H^1(\Vd)$ is a ground state of $\hat{J}_{\omega, \V_\varepsilon^d}$.
\end{lemma}
\begin{proof}
Fix $d \geq 1$, $\varepsilon>0$, $p > 2$ and $\omega >0$. We first prove 
$$\tilde{\mathcal{J}}_{\G_\varepsilon^d}(\omega)=\hat{\mathcal{J}}_{\V_\varepsilon^d}(\omega)>0.$$ 
By \cite[Theorem 1.1]{HX1}, there exists a ground state $f$ of $\hat{J}_{\omega, \V_\varepsilon^d}$ in $\hat{\mathcal{N}}_{\omega, \V_\varepsilon^d}$ such that $$\hat{J}_{\omega, \V_\varepsilon^d}(f)=(1/2-1/p)\varepsilon\|f\|_{\ell^p(\V_\varepsilon^d)}^p>0.$$ Moreover, by \eqref{eqftilde}, we know that $\tilde{f} \in \tilde{\mathcal{N}}_{\omega, \G_\varepsilon^d}$ and  
$$\tilde{\mathcal{J}}_{\G_\varepsilon^d}(\omega)\leq
\tilde{J}_{\omega, \G_\varepsilon^d}(\tilde{f})=\hat{J}_{\omega, \V_\varepsilon^d}(f)=\hat{\mathcal{J}}_{\V_\varepsilon^d}(\omega).
$$
Therefore, to prove that $\tilde{\mathcal{J}}_{\G_\varepsilon^d}(\omega)=\hat{\mathcal{J}}_{\V_\varepsilon^d}(\omega)$, it suffices to show that 
    $\tilde{\mathcal{J}}_{\G_\varepsilon^d}(\omega)\geq \hat{\mathcal{J}}_{\V_\varepsilon^d}(\omega)$.
For any $u\in \tilde{\mathcal{N}}_{\omega, \G_\varepsilon^d}$, by Lemma \ref{lemequiv}, we know that
$$
\varepsilon\sum_{\vv \in \V_\varepsilon^d}  \abs{\hat{u}(\vv)}^p= \|u'\|_{L^2(\G_\varepsilon^d)}^2 + \varepsilon\omega\sum_{\vv \in \V_\varepsilon^d}  \abs{\hat{u}(\vv)}^2 >0.
$$
Thus, $\|\hat{u}\|^{p}_{\ell^p(\V_\varepsilon^d)}=\sum_{\vv \in \V_\varepsilon^d}  \abs{\hat{u}(\vv)}^p>0$. Define
$$
\hat{t}_{\hat{u}}:=\left(\frac{\|\nabla_\varepsilon^d \hat{u}\|_{\ell^2(\V_\varepsilon^d)}^2 + \varepsilon\omega\|\hat{u}\|_{\ell^2(\V_\varepsilon^d)}^2}{ \varepsilon\|\hat{u}\|_{\ell^p(\V_\varepsilon^d)}^p }\right)^\frac{1}{p-2}.
$$
Then, we have that $\hat{t}_{\hat{u}} \hat{u} \in \hat{\mathcal{N}}_{\omega, \V_\varepsilon^d}$.  By Lemma \ref{lemh1gv}, we know that $\hat{t}_{\hat{u}} \leq 1$, hence,
$$
\hat{J}_{\omega, \V_\varepsilon^d}(\hat{t}_{\hat{u}}\hat{u}) =(\frac{1}{2}-\frac{1}{p})\varepsilon\hat{t}_{\hat{u}}^p\|\hat{u}\|_{\ell^p(\V_\varepsilon^d)}^p= (\frac{1}{2}-\frac{1}{p})\varepsilon\hat{t}_{\hat{u}}^p\sum_{\vv \in \V_\varepsilon^d}  \abs{\hat{u}(\vv)}^p=\hat{t}_{\hat{u}}^p\tilde{J}_{\omega, \G_\varepsilon^d}(u) \leq \tilde{J}_{\omega, \G_\varepsilon^d}(u).
$$
Since $\hat{J}_{\omega, \V_\varepsilon^d}(\hat{t}_{\hat{u}}\hat{u}) \geq \hat{\mathcal{J}}_{\V_\varepsilon^d}(\omega)$, we conclude that $\tilde{\mathcal{J}}_{\G_\varepsilon^d}(\omega)\geq \hat{\mathcal{J}}_{\V_\varepsilon^d}(\omega)=\hat{J}_{\omega, \V_\varepsilon^d}(f)>0$. Therefore,
$$\tilde{\mathcal{J}}_{\G_\varepsilon^d}(\omega)=\hat{\mathcal{J}}_{\V_\varepsilon^d}(\omega)>0.$$
It is clear that if $f\in  H^1(\V_\varepsilon^d)$ is a ground state of $\hat{J}_{\omega, \V_\varepsilon^d}$, then $\tilde{f} \in H^1(\G_\varepsilon^d)$ is a ground state of $\tilde{J}_{\omega, \G_\varepsilon^d}$.

Now, if $u \in H^1(\Gd)$ is a ground state of $\tilde{J}_{\omega,\Gd}$ in $\tilde{\mathcal{N}}_{\omega, \G_\varepsilon^d}$. Then, by repeating the arguments in \cite[Lemma 6.1]{HJT2}, we have $\tilde{\hat{u}}=u$. Hence, by Lemma \ref{lemh1gv}, $\hat{u}\in \hat{\mathcal{N}}_{\omega, \V_\varepsilon^d}$ and$$\hat{J}_{\omega,\Vd}(\hat{u})=\tilde{J}_{\omega,\Gd}(u)=\hat{\mathcal{J}}_{\Vd}(\omega)=\tilde{\mathcal{J}}_{\Gd}(\omega),$$
which completes the proof of the lemma.
\end{proof}
For every $p >2$ and every $\mu > 0$, we introduce the energy functional $\tilde{E}_{\G_\varepsilon^d} : H^1(\G_\varepsilon^d) \to \mathbb{R}$
\begin{equation*}
    \tilde{E}_{\G_\varepsilon^d}(u) := \frac{1}{2}\|u'\|_{L^2(\G_\varepsilon^d)}^2  - \frac{\varepsilon}{p}\sum_{\vv \in \V_\varepsilon^d}  \abs{\hat{u}(\vv)}^p
\end{equation*}
and the mass constraint
$$
H_\mu^1(\Gd):=\left\{u \in H^1(\Gd): \sum_{\vv \in \V_\varepsilon^d}  \abs{\hat{u}(\vv)}^2=\mu\right\}.
$$
Due to the continuous embedding $\ell^2(\V_\varepsilon^d)\hookrightarrow \ell^p(\V_\varepsilon^d)$, the  minimization problem
\begin{equation}
 \label{eqinfe}
    \tilde{\mathcal{E}}_{\G_\varepsilon^d}(\mu) := \inf_{v \in H^1_\mu(\Gd)} \tilde{E}_{ \G_\varepsilon^d}(v)
\end{equation}
is well defined. Moreover, a function $u \in H^1_\mu({\G_\varepsilon^d})$ is called a ground state of $\tilde{E}_{\G_\varepsilon^d}$ at mass $\mu$ if $$\tilde{E}_{\G_\varepsilon^d}(u) = \tilde{\mathcal{E}}_{\G_\varepsilon^d}(\mu).$$

\begin{lemma}\label{leminfe}
    For any $d \geq 1$, $\varepsilon>0$, $p > 2$ and $\mu >0$, $$\tilde{\mathcal{E}}_{\G_\varepsilon^d}(\mu)=\hat{\mathcal{E}}_{\V_\varepsilon^d}(\mu).$$ In addition, if $f\in  H^1_\mu(\V_\varepsilon^d)$ is a ground state of $\hat{E}_{\V_\varepsilon^d}$ at mass $\mu$, then $\tilde{f} \in H^1_\mu(\G_\varepsilon^d)$ is a ground state of $\tilde{E}_{\G_\varepsilon^d}$ at mass $\mu$. Moreover, if $u \in H^1_\mu(\Gd)$ is a ground state of $\tilde{E}_{\Gd}$ at mass $\mu$, then $\hat{u} \in H^1_\mu(\Vd)$ is a ground state of $\hat{E}_{\Vd}$ at mass $\mu$.
\end{lemma}
\begin{proof}
   Fix $d \geq 1$, $\varepsilon>0$, $p > 2$ and $\mu >0$. It follows from Lemma \ref{lemh1gv} that 
     $$
    \tilde{E}_{\Gd}(\tilde{f})= \hat{E}_{\Vd}(f) \quad  \text{and }\sum_{\vv \in \V_\varepsilon^d} \left|\hat{\tilde{f}}(\vv)\right|^2= \sum_{\vv \in \V_\varepsilon^d} \left|f(\vv)\right|^2=\mu, \quad\forall f \in H_\mu(\Vd)
    $$
    and
    $$
    \hat{E}_{\Vd}(\hat{u})\leq \tilde{E}_{\Gd}(u) \quad \text{and }\norm{\hat{u}}^2_{\ell^2(\Vd)}=\mu,\quad\forall u \in H_\mu(\Gd).
    $$
    Therefore, we have 
    $\tilde{\mathcal{E}}_{\G_\varepsilon^d}(\mu)=\hat{\mathcal{E}}_{\V_\varepsilon^d}(\mu)$, which completes the proof of the lemma.
\end{proof}
We conclude this section with the following lemma, which characterizes the asymptotic behavior of ground states for concentrated nonlinearities on $\Gd$ as $\varepsilon \to 0$.
\begin{lemma}\label{lemlL} Let $d \geq 1$ and $p \in (2,2^*)$, for every $\varepsilon>0$ small enough, let $u_\varepsilon \in H^1 (\Gd)$ be such that 
\begin{equation}\label{eqbddpoint}
    \frac{1}{C'}\leq \varepsilon^{d-1}\|u_\varepsilon'\|_{L^2(\mathcal{G}_\varepsilon^d)}^2,\,\,\, \varepsilon^{d}\sum_{\vv \in \Vd}\abs{\hat{u}_\varepsilon(\vv)}^2, \,\,\, \varepsilon^{d}\sum_{\vv \in \Vd}\abs{\hat{u}_\varepsilon(\vv)}^p \leq C'
\end{equation}
for a suitable constant $C'>0$, independent of $\varepsilon$. Then, we have
 \begin{enumerate}[label=\rm(\roman*)]
     \item  there exists a constant $C_1>0$, independent of $\varepsilon$, such that,
     $$
        \left| \varepsilon^d \sum_{\vv \in \V^d_\varepsilon} |\hat{u}_\varepsilon(\vv)|^2 - \norm{\mathcal{A}u_\varepsilon}^2_{L^2(\R^d)} \right| \leq C_1\varepsilon;
    $$
    \item if $d \in \{1,2\}$ and $p>2$, or $d \geq 3$ and $p \in (2,\frac{2^*}{2}+1]$, then there exists a constant $C_2>0$, independent of $\varepsilon$, such that
        $$
        \left| \varepsilon^d \sum_{\vv \in \V^d_\varepsilon} |\hat{u}_\varepsilon(\vv)|^p - \norm{\mathcal{A}u_\varepsilon}^p_{L^p(\R^d)} \right| 
    \leq C_2\varepsilon;
    $$
    \item if $d \geq 3$ and $p \in (\frac{2^*}{2}+1,2^*)$, then there exists a constant $C_3>0$, independent of $\varepsilon$, such that
    $$
\left| \varepsilon^d \sum_{\vv \in \V^d_\varepsilon} |\hat{u}_\varepsilon(\vv)|^p - \norm{\mathcal{A}u_\varepsilon}^p_{L^p(\R^d)} \right| 
   \leq  C_3\varepsilon^{\frac{d-2}{2}(2^*-p)}.
    $$
 \end{enumerate}
\end{lemma}
\begin{proof} Let $d \geq 1$ and $p \in (2,2^*)$ be fixed and, for every $\varepsilon>0$ small enough, let $u_\varepsilon \in H^1 (\Gd)$ satisfy \eqref{eqbddpoint}.
        It follows from  \eqref{eqbddpoint} and Lemma \ref{lemequiv} that 
    \begin{equation}\label{eqbddL2}
        \varepsilon^{d-1} \norm{u_\varepsilon}_{L^2(\Gd)}^2 \leq C \varepsilon^d \sum_{\vv \in \V^d_\varepsilon} |\hat{u}_\varepsilon(\vv)|^2 + C\varepsilon^{d+1}\norm{u_\varepsilon'}_{L^2(\Gd)}^2 \leq C.
    \end{equation}
    By \eqref{eqbddpoint} and Lemma \ref{lemlp}, we obtain
        \begin{equation}\label{eqlem2e1}
                    \left| d\varepsilon^d \sum_{\vv \in \V^d_\varepsilon} |\hat{u}_\varepsilon(\vv)|^2 - \varepsilon^{d-1} \|u_\varepsilon\|_{L^{2}(\Gd)}^2 \right| \leq S_{d,2}\varepsilon^d \norm{u_\varepsilon}_{L^2(\Gd)}\norm{u_\varepsilon'}_{L^2(\Gd)}\leq C\varepsilon.
        \end{equation}
    If $d \in \{1,2\}$ and $p>2$, or $d \geq 3$ and $p \in (2,\frac{2^*}{2}+1]$, then, by  \eqref{eqbddpoint}, \eqref{eqbddL2} and \cite[Lemma 3.1]{Do2}, we have
    $$
    \|u_\varepsilon\|_{L^{2p-2}(\Gd)}^{p-1}\leq C\varepsilon^{\frac{(p-2)}{2}(d-1)}\|u_\varepsilon\|_{L^{2}(\Gd)}^{\frac{d+(2-d)(p-1)}{2}} \|u_\varepsilon'\|_{L^2(\Gd)}^{\frac{p-2}{2}d} \leq  C\varepsilon^{\frac{(p-2)}{2}(d-1)}\varepsilon^{\frac{(p-1)}{2}(1-d)} = C\varepsilon^{(1-d)/2},
    $$
hence, by Lemma \ref{lemlp},
    \begin{equation}\label{eqlem2e2}
           \left| d\varepsilon^d \sum_{\vv \in \V^d_\varepsilon} |\hat{u}_\varepsilon(\vv)|^p - \varepsilon^{d-1} \|u_\varepsilon\|_{L^{p}(\Gd)}^p \right| 
    \leq S_{d,p} \varepsilon^d \|u_\varepsilon\|_{L^{2p-2}(\Gd)}^{p-1} \|u_\varepsilon'\|_{L^2(\Gd)} \leq  C\varepsilon.
    \end{equation}
    Then, {\rm(i)} and {\rm(ii)} follow directly from \eqref{eqlem2e1}, \eqref{eqlem2e2} and \cite[Lemma 6.1]{Do2}.
    
If $d \geq 3$ and $p \in (\frac{2^*}{2}+1,2^*)$, then, by \cite[Lemma 3.1]{Do2} and \eqref{eqlinfty}, we have 
    \begin{equation}\label{eqlem2e2p-2}
    \|u_\varepsilon\|_{L^{2p-2}(\Gd)}^{p-1} \leq C\varepsilon^{(2^*/2-1)(d-1)/2}\norm{u_\varepsilon'}_{L^2(\Gd)}^{2^*/2}\norm{u_\varepsilon}_{L^\infty(\Gd)}^{p-1-2^*/2} \leq C\varepsilon^{p/2}\norm{u_\varepsilon'}_{L^2(\Gd)}^{p-1}
        \end{equation}
and thus,  by Lemma \ref{lemlp} and \eqref{eqbddpoint}, one has  
    \begin{equation}\label{eqlem2e3}
    \begin{aligned}
           \left| d\varepsilon^d \sum_{\vv \in \V^d_\varepsilon} |\hat{u}_\varepsilon(\vv)|^p - \varepsilon^{d-1} \|u_\varepsilon\|_{L^{p}(\Gd)}^p \right| 
    &\leq S_{d,p} \varepsilon^d \|u_\varepsilon\|_{L^{2p-2}(\Gd)}^{p-1} \|u_\varepsilon'\|_{L^2(\Gd)}\\
     &\leq  C\varepsilon^d\varepsilon^{p/2}\norm{u_\varepsilon'}_{L^2(\Gd)}^{p}\\
    &\leq C \varepsilon^{d+\frac{p}{2}+\frac{p}{2}(1-d)}= C\varepsilon^{\frac{d-2}{2}(2^*-p)}.
    \end{aligned}
        \end{equation}
    On the one hand, by \eqref{eqbddpoint}, repeating the arguments in the proof of \cite[Lemma 4.3(ii)]{Do2} with \eqref{eqlem2e2p-2} in place of \cite[Proposition 3.2]{Do2} yields
    \begin{equation}\label{eqlem2e4}
            \left|{\|u_\varepsilon\|_{L^{p}(\Gd)}^p-\|\tilde{\hat{u}}_\varepsilon\|_{L^{p}(\Gd)}^p}\right| \leq C\varepsilon^{p/2+1}\norm{u'_\varepsilon}_{L^2(\Gd)}^p \leq C\varepsilon^{p+1-pd/2}.
    \end{equation}
    On the other hand, by repeating the arguments in the proof of \cite[Eq. (59)]{Do2} and \cite[Lemma 4.4]{Do2}, we have
    $$
    \begin{aligned}
    \left|{\frac{\varepsilon^{d-1}}{d}\|\tilde{\hat{u}}_\varepsilon\|_{L^{p}(\Gd)}^p}-\norm{\mathcal{A}u_\varepsilon}^p_{L^p(\R^d)}\right|
    &\leq C\varepsilon\norm{\mathcal{A}u_\varepsilon}^{p-1}_{L^{2p-2}(\R^d)}\norm{\nabla \mathcal{A}u_\varepsilon}_{L^2(\R^d)}\\
    &\leq C\varepsilon\left(\varepsilon^d\sum_{\vv \in \Vd}\abs{\hat{u}_\varepsilon(\vv)}^{2p-2}\right)^{1/2}\left(\varepsilon^{(d-1)/2}\norm{\tilde{\hat{u}}_\varepsilon'}_{L^2(\Gd)}\right),
    \end{aligned}
    $$
    then, it follows from \eqref{eqsobo}, \eqref{eqlinfty}, \eqref{eqbddpoint}  and Lemma \ref{lemh1gv} that
    \begin{equation}\label{eqlem2e5}       
    \begin{aligned}
    \left|{\frac{\varepsilon^{d-1}}{d}\|\tilde{\hat{u}}_\varepsilon\|_{L^{p}(\Gd)}^p}-\norm{\mathcal{A}u_\varepsilon}^p_{L^p(\R^d)}\right|
    &\leq C\varepsilon \left(\varepsilon^{d+2^*/2}\norm{u_\varepsilon'}_{L^2(\Gd)}^{2^*}\norm{u_\varepsilon}_{L^\infty(\Gd)}^{2p-2-2^*}\right)^{1/2}\left(\varepsilon^{(d-1)/2}\norm{u_\varepsilon'}_{L^2(\Gd)}\right)\\
    &\leq  C\varepsilon \left(\varepsilon^{d+p-1}\norm{u_\varepsilon'}_{L^2(\Gd)}^{2p-2}\right)^{1/2}\\
    &\leq C\varepsilon^{(d+p+1)/2}\varepsilon^{(p-1)(1-d)/2}=C\varepsilon^{\frac{d-2}{2}(2^*-p)}.
    \end{aligned}
        \end{equation}
    By \eqref{eqlem2e3}-\eqref{eqlem2e5}, we complete the proof of {\rm (iii)}.
    
\end{proof}
\section{Proofs of Theorems \ref{th1}-\ref{th2}}\label{sec3}
In view of Lemma \ref{lemh1gv}, for any $d \geq 2$, we can define the extension operator $\hat{\mathcal{A}}:H^1(\V_\varepsilon^d) \to H^1(\R^d)$ by
$
\hat{\mathcal{A}}f:=\mathcal{A}\tilde{f},
$
where $\mathcal{A}$ is the extension operator from $H^1 (\Gd)$ to $H^1(\R^d)$ defined as in \cite[Section 2]{Do2}. For any $x \in \R^1$, we know that there exists a unique $\vv\in \V_\varepsilon^1=\varepsilon\Z^1$ such that $x \in [\vv, \vv +\varepsilon)$. For $d=1$, we define the extension operator $\mathcal{A}:H^1(\G_\varepsilon^1) \to H^1(\R^1)$ by
$$
(\mathcal{A}u) (x):= u_\e(x-\vv), \text{ where }x \in [\vv, \vv +\varepsilon), \e:= (\vv,\vv+\varepsilon) \in \E_\varepsilon^1, \text{ and }\vv \text{ identified with } 0 \text{ on }\e,
$$
and define $\hat{\mathcal{A}}:H^1(\V_\varepsilon^1) \to H^1(\R^1)$ by
$
\hat{\mathcal{A}}f:=\mathcal{A}\tilde{f}.
$
It is easy to check that results in \cite[Section 3-5]{Do2} for the case $d \geq 2$ can be extend to the case that $d=1$. As a consequence, we omit the details in this paper.

In view of Lemma \ref{leminfa}, to study the asymptotic behaviour of ground states of  $\hat{J}_{\omega, \V_\varepsilon^d}(u)$, it suffices to study the asymptotic behaviour of ground states of $\tilde{J}_{\omega, \G_\varepsilon^d}(u)$. 

We first give some preliminary estimates on  $\tilde{\mathcal{J}}_{\mathcal{G}_\varepsilon^d}(\omega)$ as $\varepsilon \to 0$.
\begin{lemma}\label{lemaction}
For every $d \geq 1$, $p \in (2, 2^*)$ and $\omega > 0$, there exists a constant $C'>0$, depending only on $p$, $\omega$ and $d$, such that for every $\varepsilon>0$ small enough
\[
    \varepsilon^{d-1}\tilde{\mathcal{J}}_{\mathcal{G}_\varepsilon^d}(\omega) \le \mathcal{J}_{\mathbb{R}^d}(\omega) + C'\varepsilon. 
\]
Moreover, for every $d \geq 3$, $p>2^*$ and $\omega > 0$,
\begin{equation}\label{eqejsuper}
    \varepsilon^{d-1}\tilde{\mathcal{J}}_{\mathcal{G}_\varepsilon^d}(\omega)    \rightarrow 0, \quad \text{as }\varepsilon \to 0.
\end{equation}
\end{lemma}
\begin{proof}
Firstly, let $d \geq 1$, $p \in (2, 2^*)$ and $\omega > 0$. Let $u \in \mathcal{N}_{\omega, \mathbb{R}^d}$ be a ground state of $J_{\omega, \mathbb{R}^d}$, which is positive and in $C^1(\mathbb{R}^d) \cap H^2(\mathbb{R}^d) \cap L^\infty(\mathbb{R}^d)$,  and let $u_\varepsilon \in H^1(\mathcal{G}_\varepsilon^d)$ be its restriction to $\mathcal{G}_\varepsilon^d$. It follows from \cite[Lemma 4.1]{Do2} that for every $\varepsilon>0$ small enough,
\begin{equation}\label{eqgrar}
    \left| \varepsilon^{d-1}\|u'_\varepsilon\|_{L^2(\Gd)}^2- \norm{\nabla u}^2_{L^2(\R^d)}\right| \leq C \varepsilon,
\end{equation}
$$
\left| \frac{\varepsilon^{d-1}}{d}\|u_\varepsilon\|_{L^q(\Gd)}^q- \norm{u}^q_{L^q(\R^d)}\right| \leq C \varepsilon, \quad \text{for }q\in \{2,p,2p-2\}.
$$
Then, by Lemma \ref{lemlp}, for $q\in \{2,p\}$, we have
\begin{equation}\label{eqsumvr}
    \begin{aligned}
    \left| \varepsilon^d\sum_{\vv \in \Vd}\abs{\hat{u}_\varepsilon(\vv)}^q - \norm{u}^q_{L^q(\R^d)} \right|
    &\leq  S_{d, q} \frac{\varepsilon^d }{d}\|u_\varepsilon\|^{q-1}_{L^{2q-2}(\Gd)} \|u_\varepsilon'\|_{L^2(\Gd)} + C\varepsilon\\
    &\leq \frac{S_{d, q}\varepsilon}{d}\left(\varepsilon^{(d-1)/2} \|u_\varepsilon\|^{q-1}_{L^{2q-2}(\Gd)}\right)\left( \varepsilon^{(d-1)/2} \|u_\varepsilon'\|_{L^2(\Gd)}\right) +C\varepsilon\\
    &\leq C \varepsilon.
\end{aligned}
\end{equation}
Define
$$
\tilde{t}_{u_\varepsilon}:=\left(\frac{\displaystyle \|u'_\varepsilon\|_{L^2(\Gd)}^2 + \varepsilon\omega\sum_{\vv \in \Vd}\abs{\hat{u}_\varepsilon(\vv)}^2}{\displaystyle \varepsilon\sum_{\vv \in \Vd}\abs{\hat{u}_\varepsilon(\vv)}^p}\right)^\frac{1}{p-2}.
$$
By the definition of $\tilde{\mathcal{N}}_{\omega,\Gd}$ given in \eqref{eqtilden}, we see that $\tilde{t}_{u_\varepsilon}u_\varepsilon \in \tilde{\mathcal{N}}_{\omega,\Gd}$. Then, for every $\varepsilon>0$ small enough, by \eqref{eqgrar} and \eqref{eqsumvr}, we obtain
$$
\begin{aligned}
    \varepsilon^{d-1}\tilde{\mathcal{J}}_{\mathcal{G}_\varepsilon^d}(\omega) &\leq \varepsilon^{d-1}\tilde{J}_{\omega,\Gd}(\tilde{t}_{u_\varepsilon}u_\varepsilon)\\
    &=(\frac{1}{2}-\frac{1}{p})\varepsilon^d\tilde{t}_{u_\varepsilon}^p\sum_{\vv \in \Vd}\left|\hat{u}_\varepsilon(\vv)\right|^p\\
    &\leq (\frac{1}{2}-\frac{1}{p})\left(\frac{\displaystyle \varepsilon^{d-1}\|u'_\varepsilon\|_{L^2(\Gd)}^2 + \varepsilon^d\omega\sum_{\vv \in \Vd}\abs{\hat{u}_\varepsilon(\vv)}^2}{\displaystyle \varepsilon^d\sum_{\vv \in \Vd}\abs{\hat{u}_\varepsilon(\vv)}^p}\right)^\frac{p}{p-2}\left(\norm{u}^p_{L^p(\R^d)}+C\varepsilon\right)\\
    & \leq (\frac{1}{2}-\frac{1}{p})\left(\frac{\norm{\nabla u}^2_{L^2(\R^d)}+\omega \norm{u}^2_{L^2(\R^d)}+C \varepsilon}{\displaystyle  \norm{u}^p_{L^p(\R^d)}-C\varepsilon}\right)^\frac{p}{p-2}\left(\norm{u}^p_{L^p(\R^d)}+C\varepsilon\right)\\
    & \leq (\frac{1}{2}-\frac{1}{p})(1+C\varepsilon)\norm{u}^p_{L^p(\R^d)}+C \varepsilon \\
    &\leq  J_{\omega,\mathbb{R}^d}(u) + C'\varepsilon\\
    &=\mathcal{J}_{\R^d}(\omega) + C'\varepsilon.
\end{aligned}
$$

Next, let $d \geq 3$, $p > 2^*$ and  $\omega > 0$. Define the map $J: C_c^\infty(\R^d) \to \R$
$$
J(u):=\frac{1}{2}\norm{\nabla u}^2_{L^2(\R^d)} +\frac{\omega}{2}\norm{u}^2_{L^2(\R^d)}-\frac{1}{p}\norm{u}^p_{L^p(\R^d)}
$$
and the set
\begin{align*}
    \mathcal{N} := \left\{ u \in C_c^\infty(\R^d) \backslash \{0\} : \|\nabla u\|_{L^2(\mathbb{R}^d)}^2 + \omega\|u\|_{L^2(\mathbb{R}^d)}^2 = \|u\|_{L^p(\mathbb{R}^d)}^p \right\}.
\end{align*}
For any $v \in C_c^\infty(\R^d)\backslash\{0\}$, it is easy to see that $t_v v \in \mathcal{N}$, where
$$
t_v:=\left(\frac{\norm{\nabla v}^2_{L^2(\R^d)}+\omega \norm{v}^2_{L^2(\R^d)}}{\displaystyle  \norm{v}^p_{L^p(\R^d)}}\right)^\frac{1}{p-2},
$$
hence, $\mathcal{N} \neq \emptyset$.
We now claim that 
\begin{equation}\label{eqsuper}
   \limsup_{\varepsilon\to 0}\varepsilon^{d-1}\tilde{\mathcal{J}}_{\mathcal{G}_\varepsilon^d}(\omega) \leq \inf_{v \in \mathcal{N}} J(v). 
\end{equation}
Indeed, for any fixed $v \in \mathcal{N}$, by repeating the previous arguments, we obtain that, for every $\varepsilon>0$ small enough,
$$
\varepsilon^{d-1}\tilde{\mathcal{J}}_{\mathcal{G}_\varepsilon^d}(\omega) \leq J(v) + C'\varepsilon,
$$
that is \eqref{eqsuper}. Thus, to prove \eqref{eqejsuper}, it suffices to prove that 
$$\inf_{v \in \mathcal{N}} J(v)=0.$$
Let $v\in \mathcal{N}$. For any $s>0$, define $v_s:=v(sx)$. Then, we know that $t_sv_s \in \mathcal{N}$, where
$$
t_s:=\left(\frac{\norm{\nabla v_s}^2_{L^2(\R^d)}+\omega \norm{v_s}^2_{L^2(\R^d)}}{\displaystyle  \norm{v_s}^p_{L^p(\R^d)}}\right)^\frac{1}{p-2}=\left(\frac{s^{2-d}\norm{\nabla v}^2_{L^2(\R^d)}+\omega s^{-d}\norm{v}^2_{L^2(\R^d)}}{\displaystyle  s^{-d}\norm{v}^p_{L^p(\R^d)}}\right)^\frac{1}{p-2}.
$$
Then, since $2p/(p-2)<d$ by $p>2^*$, we conclude that 
$$
J(t_sv_s)=(\frac{1}{2}-\frac{1}{p})t_s^p\norm{v_s}_{L^p(\R^d)}^p=(\frac{1}{2}-\frac{1}{p})t_s^ps^{-d}\norm{v}_{L^p(\R^d)}^p \to 0, \quad \text{as }s \to +\infty,
$$
which completes the proof of the lemma.
\end{proof}
Next, we show that the sequence of ground states of  $\tilde{J}_{\omega, \mathcal{G}_\varepsilon^d}$ satisfies the condition of Lemma \ref{lemlL} for $\varepsilon>0$ small enough.

\begin{lemma}\label{lemactionbdd}
For every $d \geq 1$, $p \in (2, 2^*)$ and $\omega > 0$, there exist $\bar{\varepsilon} > 0$ and $C' > 0$, depending only on $p$, $\omega$ and $d$, such that, for every $\varepsilon \in (0, \bar{\varepsilon})$, if $u \in \tilde{\mathcal{N}}_{\omega, \mathcal{G}_\varepsilon^d}$ is a ground state of $\tilde{J}_{\omega, \mathcal{G}_\varepsilon^d}$, then
\begin{equation}\label{eqbddactp}
\begin{aligned}
    \frac{1}{C'}\leq \varepsilon^{d-1}\|u'\|_{L^2(\mathcal{G}_\varepsilon^d)}^2,\, \varepsilon^{d}\sum_{\vv \in \Vd}\abs{\hat{u}(\vv)}^2,\,  \varepsilon^{d}\sum_{\vv \in \Vd}\abs{\hat{u}(\vv)}^p \leq C'.
\end{aligned}
 \end{equation}
\end{lemma}
\begin{proof}       
    Let $d \geq 1$, $p \in (2, 2^*)$. 
 Let $u \in \tilde{\mathcal{N}}_{\omega, \mathcal{G}_\varepsilon^d}$ be a ground state of $\tilde{J}_{\omega, \mathcal{G}_\varepsilon^d}$. Then, we have 
    $$
    (\frac{1}{2}-\frac{1}{p})\varepsilon^d\sum_{\vv \in \Vd}\abs{\hat{u}(\vv)}^p=\varepsilon^{d-1}\tilde{J}_{\omega,\Gd}(u)=\varepsilon^{d-1}\tilde{\mathcal{J}}_{\mathcal{G}_\varepsilon^d}(\omega).
    $$
    Thus, by Lemma \ref{lemaction} and the definition of $\tilde{\mathcal{N}}_{\omega,\Gd}$ given in \eqref{eqtilden}, we have that for $\varepsilon>0$ small enough
    \begin{equation}\label{eqbddpointup}
            \varepsilon^d\sum_{\vv \in \Vd}\abs{\hat{u}(\vv)}^p, \,\,\varepsilon^d\sum_{\vv \in \Vd}\abs{\hat{u}(\vv)}^2,\,\,\varepsilon^{d-1}\|u'\|_{L^2(\mathcal{G}_\varepsilon^d)}^2\leq C.
    \end{equation}
    Therefore, by the Gagliardo--Nirenberg type inequality \eqref{eqgng}, and the definition of $\tilde{\mathcal{N}}_{\omega,\Gd}$  given in \eqref{eqtilden}, one has
    $$
    \begin{aligned}
        \sum_{\vv \in \V_\varepsilon^d}  \abs{\hat{u}(\vv)}^p &\leq  C_{d,p}(\varepsilon^{1/2}\|u'\|_{L^2(\G_\varepsilon^d)})^{\beta_pp}\left(\sum_{\vv \in \V_\varepsilon^d}  \abs{\hat{u}(\vv)}^2\right)^{(1-\beta_p)p/2}\\
       &\leq C \varepsilon^{\beta_pp/2}\left(\varepsilon \sum_{\vv \in \V_\varepsilon^d}  \abs{\hat{u}(\vv)}^p\right)^{\beta_pp/2} \left(\sum_{\vv \in \V_\varepsilon^d}  \abs{\hat{u}(\vv)}^p\right)^{(1-\beta_p)p/2}\\
       &=C \varepsilon^{\beta_pp}\left(\sum_{\vv \in \V_\varepsilon^d}  \abs{\hat{u}(\vv)}^p\right)^{p/2},
    \end{aligned}
    $$
    that is,
    $$
    \sum_{\vv \in \V_\varepsilon^d}  \abs{\hat{u}(\vv)}^p \geq  C\varepsilon^{-d}.
    $$
    By \eqref{eqbddpointup} and the Gagliardo--Nirenberg type inequality \eqref{eqgng}, we have that for $\varepsilon>0$ small enough
    \begin{equation}\label{eqlem321}
            \begin{aligned}
            C\varepsilon^{-d} \leq  \sum_{\vv \in \V_\varepsilon^d}  \abs{\hat{u}(\vv)}^p &\leq  C_{d,p}(\varepsilon^{1/2}\|u'\|_{L^2(\G_\varepsilon^d)})^{\beta_pp}\left(\sum_{\vv \in \V_\varepsilon^d}  \abs{\hat{u}(\vv)}^2\right)^{(1-\beta_p)p/2} \\
    &\leq C_{d,p}(\varepsilon^{1/2}\|u'\|_{L^2(\G_\varepsilon^d)})^{\beta_pp}C\varepsilon^{-d(1-\beta_p)p/2}, 
    \end{aligned}
    \end{equation}
    that is,
    $$
    \|u'\|_{L^2(\G_\varepsilon^d)} \geq C\varepsilon^{-1/2}\varepsilon^\frac{d(1-\beta_p)p-2d}{2\beta_pp}=C\varepsilon^{(1-d)/2}.
    $$
    Similarly, it follows from \eqref{eqbddpointup} and \eqref{eqlem321} that for $\varepsilon>0$ small enough
    $$
    \begin{aligned}
            C\varepsilon^{-d} \leq  \sum_{\vv \in \V_\varepsilon^d}  \abs{\hat{u}(\vv)}^p &\leq  C_{d,p}(\varepsilon^{1/2}\|u'\|_{L^2(\G_\varepsilon^d)})^{\beta_pp}\left(\sum_{\vv \in \V_\varepsilon^d}  \abs{\hat{u}(\vv)}^2\right)^{(1-\beta_p)p/2} \\
    &\leq C_{d,p}\varepsilon^{(2-d)\beta_pp/2}\left(\sum_{\vv \in \V_\varepsilon^d}  \abs{\hat{u}(\vv)}^2\right)^{(1-\beta_p)p/2}, 
    \end{aligned}
    $$
     that is,
    $$
   \sum_{\vv \in \V_\varepsilon^d}  \abs{\hat{u}(\vv)}^2 \geq C\varepsilon^\frac{(d-2)\beta_pp/2-d}{(1-\beta_p)p/2}=C\varepsilon^{-d},
    $$
    which completes the proof of \eqref{eqbddactp}.
\end{proof}

Now, we are in a position to prove Theorems \ref{th1}-\ref{th2}.
\begin{proof}[Proof of Theorem \ref{th1}]
Let $d \geq 1$, $p \in (2, 2^*)$ and $\omega > 0$ be fixed.
By \cite[Theorem 1.1]{HX1}, for every $\varepsilon>0$, there exists a ground state $f_\varepsilon$ of $\hat{J}_{\omega, \V_\varepsilon^d}$ in $\hat{\mathcal{N}}_{\omega, \V_\varepsilon^d}$. By \cite[Theorem B.4]{HJ}, we may assume that $f_\varepsilon$ is strictly positive. Set $u_\varepsilon:=\tilde{f_\varepsilon}$ (see the definition in Section \ref{sec2}) and 
$$
t_{u_\varepsilon}:=\left(\frac{\displaystyle \|\nabla \mathcal{A}u_\varepsilon\|_{L^2(\R^d)}^2 + \omega\|\mathcal{A}u_\varepsilon\|_{L^2(\R^d)}^2}{\displaystyle \| \mathcal{A}u_\varepsilon\|_{L^p(\R^d)}^p}\right)^\frac{1}{p-2}.
$$
Then, we have $v_\varepsilon:=t_{u_\varepsilon} \mathcal{A}u_\varepsilon \in \mathcal{N}_{\omega,\R^d}$.
Note that, by Lemmas \ref{leminfa} and \ref{lemactionbdd}, we can see that Lemma \ref{lemlL} applies to $u_\varepsilon$ for  $\varepsilon>0$ small enough.
If $d \in \{1,2\}$ and $p>2$, or $d \geq 3$ and $p \in (2,\frac{2^*}{2}+1]$, by Lemma \ref{lemlL} and  \cite[Lemma 4.4]{Do2}, we obtain that for $\varepsilon>0$ small enough
$$
t_{u_\varepsilon}^p\leq \left(\frac{\displaystyle \varepsilon^{d-1}\|u'_\varepsilon\|_{L^2(\Gd)}^2 + \varepsilon^d\omega\sum_{\vv \in \Vd}\abs{\hat{u}_\varepsilon(\vv)}^2+C\varepsilon}{\displaystyle \varepsilon^d\sum_{\vv \in \Vd}\abs{\hat{u}_\varepsilon(\vv)}^p-C\varepsilon}\right)^\frac{p}{p-2}\leq 1+C\varepsilon,
$$
and thus,
\begin{equation}\label{eqactle}
    \begin{aligned}
    J_{\omega,\R^d}(v_\varepsilon)=(\frac{1}{2}-\frac{1}{p})t_{u_\varepsilon}^p\| \mathcal{A}u_\varepsilon\|_{L^p(\R^d)}^p
    &\leq(\frac{1}{2}-\frac{1}{p})\varepsilon^d t_{u_\varepsilon}^p\sum_{\vv \in \V^d_\varepsilon} |\hat{u}_\varepsilon(\vv)|^p +C\varepsilon \\&= (\frac{1}{2}-\frac{1}{p})\varepsilon^d t_{u_\varepsilon}^p\sum_{\vv \in \V^d_\varepsilon} |f_\varepsilon(\vv)|^p +C\varepsilon\\
    &=\varepsilon^{d-1}t_{u_\varepsilon}^p\tilde{\mathcal{J}}_{\mathcal{G}_\varepsilon^d}(\omega) +C \varepsilon\\
    &\leq \varepsilon^{d-1}\tilde{\mathcal{J}}_{\mathcal{G}_\varepsilon^d}(\omega) +C \varepsilon.
\end{aligned}
\end{equation}
Similarly, if $d \geq 3$ and $p \in (\frac{2^*}{2}+1,2^*)$, we obtain that for $\varepsilon>0$ small enough
\begin{equation}\label{eqactge}
J_{\omega,\R^d}(v_\varepsilon) \leq \varepsilon^{d-1}t_{u_\varepsilon}^p\tilde{\mathcal{J}}_{\mathcal{G}_\varepsilon^d}(\omega) + C\varepsilon^{\frac{d-2}{2}(2^*-p)}\leq \varepsilon^{d-1}\tilde{\mathcal{J}}_{\mathcal{G}_\varepsilon^d}(\omega) + C\varepsilon^{\frac{d-2}{2}(2^*-p)}. 
\end{equation}
Since $\mathcal{J}_{\R^d}(\omega) \leq J_{\omega,\R^d}(v_\varepsilon)$ by the definition of ground states of $J_{\omega,\R^d}$,
coupling \eqref{eqactle}-\eqref{eqactge} with Lemma \ref{lemaction} proves Theorem \ref{th1}(i) and $t_{u_\varepsilon}^p \to 1$ as $\varepsilon\to 0$ (noting also that $\frac{d-2}{2}(2^*-p) \in (0,1)$ for $d \geq 3$ and $p \in (\frac{2^*}{2}+1, 2^*)$).

Since $f_\varepsilon \in H^1(\Vd)$, we know that $f_\varepsilon(\vv) \to 0$ as $\operatorname{dist}(\vv,0)\to\infty$, and thus, by the definition of $\mathcal{A}$, for every $\varepsilon>0$ small enough, we have that $\mathcal{A}u_\varepsilon(x) \to 0$ as $\abs{x}\to+\infty$. Therefore, there exists $x_\varepsilon \in \R^d$ such that
\begin{equation}\label{eqpfth1b1}
    \int_{B_{1}(x_\varepsilon)}\abs{\mathcal{A}u_\varepsilon}^2\,dx=\sup_{y \in \R^d}\int_{B_1(y)}\abs{\mathcal{A}u_\varepsilon}^2\,dx,
\end{equation}
where, as usual, $B_1(y):=\{x \in \R^d: \abs{x-y}<1\}$.

It follows from \eqref{eqactle}, \eqref{eqactge} and Theorem \ref{th1}(i) that $\{v_\varepsilon\}$ is a minimizing sequence for $\mathcal{J}_{\R^d}(\omega)$, hence, by $t_{u_\varepsilon}^p \to 1$ as $\varepsilon\to 0$ and by repeating the arguments in the proof of \cite[Theorem 2.1]{Do2}, we obtain that, up to a subsequence, $\mathcal{A}u_\varepsilon(\cdot+x_\varepsilon)\to \varphi_\omega$ as $\varepsilon\to 0$ in $H^1(\R^d)$, where  $\varphi_\omega \in \mathcal{N}_{\omega, \mathbb{R}^d}$ is a positive ground state of $J_{\omega, \mathbb{R}^d}$. Then, by \eqref{eqpfth1b1}, we know that
\begin{equation}\label{eqpfth1b12}
   \int_{B_{1}(0)}\abs{\varphi_\omega}^2\,dx=\sup_{y \in \R^d}\int_{B_1(y)}\abs{\varphi_\omega}^2\,dx.
\end{equation}
Since $\varphi_\omega(\cdot+y)$ is radially symmetric and strictly decreasing for some $y \in \R^d$, by \eqref{eqpfth1b12} and \cite[Corollary 3.8]{Haj}, $\varphi_\omega$ attains its maximum at the origin, hence, $\varphi_\omega \in \mathcal{N}_{\omega, \mathbb{R}^d}$ is the unique positive ground state of $J_{\omega, \mathbb{R}^d}$ which attains its maximum at the origin. Therefore, we conclude that the convergence of $\mathcal{A}u_\varepsilon(\cdot+x_\varepsilon)$ for the whole sequence, rather than only along subsequences (by the uniqueness of the limit), and is strong in $H^1(\R^d)$.

\end{proof}
\begin{proof}[Proof of Theorem \ref{th2}]
Let $d \ge 3$, $p > 2^*$ and $\omega > 0$ be fixed. By \cite[Theorem 1.1]{HX1}, for every $\varepsilon>0$, there exists a ground state $f_\varepsilon$ of $\hat{J}_{\omega, \V_\varepsilon^d}$ in $\hat{\mathcal{N}}_{\omega, \V_\varepsilon^d}$. By \cite[Theorem B.4]{HJ}, we may assume that $f_\varepsilon$ is strictly positive. Set $u_\varepsilon:=\tilde{f_\varepsilon}$ (see the definition in Section \ref{sec2}). By Lemmas \ref{leminfa}, \ref{lemaction} and the definition of $\tilde{\mathcal{N}}_{\omega,\Gd}$  given in \eqref{eqtilden}, we have
$$
\left(\frac{1}{2}-\frac{1}{p}\right)\left( \varepsilon^{d-1}\|u'_\varepsilon\|_{L^2(\Gd)}^2 +\varepsilon^d\omega\sum_{\vv \in \Vd}\abs{f_\varepsilon(\vv)}^2 \right) =  \varepsilon^{d-1}\tilde{\mathcal{J}}_{\mathcal{G}_\varepsilon^d}(\omega)\to 0, \quad \text{as }\varepsilon \to 0.
$$
Then, by \cite[Lemma 4.4]{Do2}, we obtain
$$
\|\nabla \mathcal{A}u_\varepsilon\|_{L^2(\R^d)}^2\leq \varepsilon^{d-1}\|u'_\varepsilon\|_{L^2(\Gd)}^2\to 0, \quad \text{as }\varepsilon \to 0.
$$
Moreover, by Lemma \ref{lemequiv} and \cite[Lemma 4.4]{Do2}, we conclude that
$$
\| \mathcal{A}u_\varepsilon\|_{L^2(\R^d)}^2 \leq C\varepsilon^{d-1}\left(\|u_\varepsilon\|_{L^2(\Gd)}^2 +\varepsilon\|u'_\varepsilon\|_{L^2(\Gd)}^2\right) \leq C\varepsilon^d\sum_{\vv \in \Vd}\abs{f_\varepsilon(\vv)}^2 +C\varepsilon\to 0, \quad \text{as }\varepsilon \to 0,
$$
which completes the proof of Theorem \ref{th2}.
\end{proof}

\section{Proofs of Theorems \ref{th3}-\ref{thl2sup} and \ref{th4}}\label{sec4}
As in Section \ref{sec3}, by Lemma \ref{leminfe}, to study the asymptotic behaviour of ground states of $\hat{E}_{\V_\varepsilon^d}\left(u\right)$ at mass $\mu/\varepsilon^d$, it suffices to study the asymptotic behaviour of ground states of $\tilde{E}_{\G_\varepsilon^d}\left(u\right)$ at mass $\mu/\varepsilon^d$.
We first give some preliminary estimates on  $\tilde{\mathcal{E}}_{\G_\varepsilon^d}\left(\frac{\mu}{\varepsilon^{d}}\right)$ as $\varepsilon \to 0$.
\begin{lemma}\label{lemerengy}
    For every $d \geq 1$, $p \in \left(2, 2 + \frac{4}{d}\right)$ and $\mu > 0$, there exists a constant $C' > 0$, depending only on $p$, $\mu$ and $d$, such that
    for every $\varepsilon>0$ small enough
\[
    \varepsilon^{d-1}\tilde{\mathcal{E}}_{\mathcal{G}_\varepsilon^d}\left( \frac{\mu}{\varepsilon^{d}} \right) \le \mathcal{E}_{\mathbb{R}^d}(\mu) + C'\varepsilon. 
\]
\end{lemma}
\begin{proof}
Let $d \geq 1$, $p \in \left(2, 2 + \frac{4}{d}\right)$ and $\mu > 0$. Let $u \in H^1_\mu(\R^d)$ be a ground state of $E_{\mathbb{R}^d}$, $u_\varepsilon \in H^1(\mathcal{G}_\varepsilon^d)$ be its restriction to $\mathcal{G}_\varepsilon^d$, and 
$$
t_\varepsilon:=\sqrt{\frac{\mu}{\varepsilon^d\sum_{\vv \in \Vd}\left|\hat{u}_\varepsilon(\vv)\right|^2}}.
$$
Then,  we have $t_\varepsilon u_\varepsilon \in H^1_{\frac{\mu}{\varepsilon^d}}(\Gd)$, hence, for $\varepsilon>0$ sufficiently small, by \eqref{eqgrar}, \eqref{eqsumvr} and by repeating the arguments in Lemma \ref{lemaction}, we obtain
$$
\varepsilon^{d-1}\tilde{\mathcal{E}}_{\mathcal{G}_\varepsilon^d}\left( \frac{\mu}{\varepsilon^{d}} \right) \leq \varepsilon^{d-1}\tilde{E}_{\Gd}(v_\varepsilon) \leq E_{\R^d}(u)+C'\varepsilon = \mathcal{E}_{\R^d}(\mu)+C'\varepsilon,
$$
which completes the proof of the lemma.
\end{proof}
Next, we show that the sequence of ground states of  $\tilde{E}_{\mathcal{G}_\varepsilon^d}$ at mass $\mu>0$ satisfies the condition of Lemma \ref{lemlL} for $\varepsilon>0$ small enough.
\begin{lemma}\label{lemerengy2}
    For every $d \geq 1$, $p \in \left(2,2+\frac{4}{d}\right)$ and $\mu >0$, there exist $\bar\varepsilon>0$ and $C' >0$, depending only on $p,\mu$ and $d$, such that, for every $\varepsilon \in (0,\bar\varepsilon)$, if $u \in H^1_{\frac{\mu}{\varepsilon^d}}(\Gd)$ is a ground state of $\tilde{E}_{\Gd}$ at mass $\frac{\mu}{\varepsilon^d}$, then 
    $$
    \frac{1}{C'} \leq  \varepsilon^{d-1}\|u'\|_{L^2(\mathcal{G}_\varepsilon^d)}^2,\,\, \varepsilon^{d}\sum_{\vv \in \Vd}\abs{\hat{u}(\vv)}^p \leq C'.
    $$
\end{lemma}
\begin{proof}
    Let $d \geq 1$, $p \in \left(2,2+\frac{4}{d}\right)$ and $\mu >0$.
 Let $u \in H^1_{\frac{\mu}{\varepsilon^d}}(\Gd)$ be a ground state of $\tilde{E}_{\mathcal{G}_\varepsilon^d}$ at mass $\frac{\mu}{\varepsilon^d}$. It is clear that $\mathcal{E}_{\R^d}(\mu)<0$. Then, by Lemma \ref{lemerengy}, for $\varepsilon>0$ sufficiently small, we have
 $$
 \tilde{\mathcal{E}}_{\mathcal{G}_\varepsilon^d}\left( \frac{\mu}{\varepsilon^{d}} \right)<-C\varepsilon^{1-d},
 $$
thus,
    $$
    \varepsilon\sum_{\vv \in \Vd}\left|\hat{u}(\vv)\right|^p = \frac{p}{2}\|u'\|_{L^2(\mathcal{G}_\varepsilon^d)}^2 -p\tilde{\mathcal{E}}_{\mathcal{G}_\varepsilon^d}\left( \frac{\mu}{\varepsilon^{d}} \right) \geq C \varepsilon^{1-d}.
    $$
    By the Gagliardo--Nirenberg type inequality \eqref{eqgng} and $\sum_{\vv \in \V_\varepsilon^d}  \abs{\hat{u}(\vv)}^2=\frac{\mu}{\varepsilon^d}$, we obtain that for every $\varepsilon>0$ small enough
    \begin{equation}\label{eqenergy1}
            \begin{aligned}
            C\varepsilon^{-d}\leq \sum_{\vv \in \Vd}\left|\hat{u}(\vv)\right|^p&\leq C_{d,p}(\varepsilon^{1/2}\|u'\|_{L^2(\G_\varepsilon^d)})^{\beta_pp}\left(\sum_{\vv \in \V_\varepsilon^d}  \abs{\hat{u}(\vv)}^2\right)^{(1-\beta_p)p/2}\\
            &\leq C\varepsilon^{((1+d)\beta_pp-pd)/2}\|u'\|_{L^2(\G_\varepsilon^d)}^{\beta_pp},\\
    \end{aligned}
    \end{equation}
    that is 
    $$\|u'\|^{2}_{L^2(\G_\varepsilon^d)} \geq C\varepsilon^{1-d}.$$  Moreover, it follows from \eqref{eqgng} that 
    $$
    \|u'\|_{L^2(\G_\varepsilon^d)}^2\leq \|u'\|_{L^2(\G_\varepsilon^d)}^2-2\tilde{\mathcal{E}}_{\mathcal{G}_\varepsilon^d}\left( \frac{\mu}{\varepsilon^{d}} \right)=\frac{2\varepsilon}{p}\sum_{\vv \in \Vd}\left|\hat{u}(\vv)\right|^p \leq C\varepsilon^{((1+d)\beta_pp-pd+2)/2}\|u'\|_{L^2(\G_\varepsilon^d)}^{\beta_pp}.
    $$
   Since $p \in \left(2,2+\frac{4}{d}\right)$, we have $\beta_pp<2$, it follows that 
     $$\|u'\|^{2}_{L^2(\G_\varepsilon^d)} \leq C\varepsilon^{1-d}.$$
Then, by \eqref{eqenergy1}, we conclude that 
    $$
    \sum_{\vv \in \Vd}\left|\hat{u}(\vv)\right|^p\leq C\varepsilon^{((1+d)\beta_pp-pd)/2}\|u'\|_{L^2(\G_\varepsilon^d)}^{\beta_pp}\leq C\varepsilon^{-d},
    $$
    which completes the proof of the lemma.
    \end{proof}
We are now in a position to prove Theorems \ref{th3}-\ref{thl2sup}.
\begin{proof}[Proof of Theorem \ref{th3}]
    Let $d \geq 1$, $p \in \left(2, 2+\frac{4}{d}\right)$ and $\mu > 0$ be fixed. By adapting the arguments in \cite{HJT, Stefanov, Weinstein}, we know that, for every $\varepsilon>0$, there exists a nonnegative ground state  $f_\varepsilon\in H^1_{\frac{\mu}{\varepsilon^d}}(\Vd)$ of $\hat{E}_{\Vd}$ at mass $\frac{\mu}{\varepsilon^d}$. It is easy to see that $f_\varepsilon$ is positive. Indeed, if $u(x_0)=0$ for some fixed $x_0 \in \Vd$, then equation \eqref{eqenergyl} yields that $u(x) = 0$ for all $x \sim x_0$. Hence, thanks to the connectivity of $\Z^d_\varepsilon$,  we obtain that $u(x) \equiv 0$, which leads to a contradiction.
    Set $u_\varepsilon:=\tilde{f_\varepsilon}$ (see the definition in Section \ref{sec2}) and 
    $$
   t_{u_\varepsilon}:=\sqrt{\frac{\mu}{\norm{\mathcal{A}u_\varepsilon}^2_{L^2(\R^d)}}}.
    $$
    Then, we have that $v_\varepsilon:=t_{u_\varepsilon} \mathcal{A}u_\varepsilon\in H_\mu^1(\R^d)$.
    Now, by repeating the arguments in the proof of Theorem \ref{th1}, we obtain that,
    if $d \in \{1,2,3,4\}$, or if $d \geq 5$ and $p \in \left(2,\frac{2^*}{2}+1\right]$ (which is strictly contained in $(2,2+\frac{4}{d})$), then for $\varepsilon>0$ small enough
    $$
    \mathcal{E}_{\R^d}(\mu) \leq E_{\R^d}(v_\varepsilon)\leq \varepsilon^{d-1}\tilde{\mathcal{E}}_{\Gd}\left(\frac{\mu}{\varepsilon^d}\right)+C\varepsilon =\varepsilon^{d-1}\hat{\mathcal{E}}_{\Vd}\left(\frac{\mu}{\varepsilon^d}\right)+C\varepsilon,
    $$
    whereas if $d \geq 5$ and $p \in \left(\frac{2^*}{2}+1,2+\frac{4}{d}\right)$, then for $\varepsilon>0$ small enough
    $$
  \mathcal{E}_{\R^d}(\mu) \leq  E_{\R^d}(v_\varepsilon)\leq \varepsilon^{d-1}\tilde{\mathcal{E}}_{\Gd}\left(\frac{\mu}{\varepsilon^d}\right)+C\varepsilon^{\frac{d-2}{2}(2^*-p)} =\varepsilon^{d-1}\hat{\mathcal{E}}_{\Vd}\left(\frac{\mu}{\varepsilon^d}\right)+C\varepsilon^{\frac{d-2}{2}(2^*-p)}.
    $$
    Since $\frac{d-2}{2}(2^*-p) \in (0,1)$ for all $d \geq 3$ and $p \in (\frac{2^*}{2}+1, 2^*)$), combining with Lemma \ref{lemerengy} proves Theorem \ref{th3}(i). 

   By Lemmas \ref{lemlL} and \ref{lemerengy2}, we obtain that $t_{u_\varepsilon} \to 1$ as $\varepsilon\to 0$. Noting that $\{v_\varepsilon\}$ is a minimizing sequence for $\mathcal{E}_{\R^d}(\mu)$, by $t_{u_\varepsilon} \to 1$ as $\varepsilon\to0$ and by repeating the arguments in the proof in \cite[Theorem 2.2]{Do2}, up to a subsequence, we obtain that $\mathcal{A}u_\varepsilon (\cdot+x_\varepsilon) \to \varphi_\mu$ in $H^1(\R^d)$ for some $\{x_\varepsilon\} \subset \R^d$, where $\varphi_\mu\in H^1_\mu(\R^d)$ is the unique positive ground state of $E_{\mathbb{R}^d}$ at mass $\mu$ which attains its maximum at the origin. Moreover, by repeating the arguments in the proof of Theorem \ref{th1}, we know that the convergence of $\mathcal{A}u_\varepsilon (\cdot+x_\varepsilon)$ for some $\{x_\varepsilon\} \subset \R^d$ for the whole sequence, rather than only along subsequences (by the uniqueness of the limit), and is strong in $H^1(\R^d)$.
     
     Then, by Lemma \ref{lemlL} and Theorem \ref{th3}(i), we conclude that 
    $$ 
    \begin{aligned}
       \mathcal{L}_{\Vd}(f_\varepsilon)&=\frac{\varepsilon\|f_\varepsilon\|_{\ell^p(\Vd)}^p - \|\nabla^d_\varepsilon f_\varepsilon\|_{\ell^2(\Vd)}^2}{\varepsilon\|f_\varepsilon\|_{L^2(\Vd)}^2}\\
        &=\frac{\displaystyle\varepsilon\sum_{\vv \in \Vd}{\left|f_\varepsilon\right|^p} - \|u_\varepsilon'\|_{L^2(\Gd)}^2}{\displaystyle \varepsilon\sum_{\vv \in \Vd}\left|f_\varepsilon\right|^2}\\
        &=\left(1-\frac{2}{p}\right)\frac{\displaystyle\varepsilon^d\sum_{\vv \in \Vd}\left|f_\varepsilon\right|^p}{\displaystyle \varepsilon^d\sum_{\vv \in \Vd}\left|f_\varepsilon\right|^2}-2\frac{\displaystyle\varepsilon^{d-1}\tilde{\mathcal{E}}_{\Gd}\left(\frac{\mu}{\varepsilon^d}\right)}{\displaystyle \varepsilon^d\sum_{\vv \in \Vd}\left|f_\varepsilon\right|^2} \to \omega_\mu \quad \text{as }\varepsilon \to 0,
    \end{aligned}
    $$
    where $\omega_\mu$ is the value of the parameter $\omega$ for which $\varphi_\mu$ solves \eqref{eqrd}. This completes the proof of Theorem \ref{th3}.
\end{proof}
\begin{proof}[Proof of Theorem \ref{thl2sup}]
    Let $d \geq 2$, $p>2+\frac{4}{d}$ and $\mu>0$ be fixed. Let $w_\varepsilon \in H^1_{\frac{\mu}{\varepsilon^d}}(\V^d_\varepsilon)$ be defined by
$$
w_\varepsilon(x):=     \begin{cases}
       \sqrt{\mu/\varepsilon^d},  &x=0,\\
       0, &x \neq 0.
      \end{cases}
$$
    Then, we have 
\begin{equation}\label{eql2sup1}
    \hat{\mathcal{E}}_{\V^d_\varepsilon}(\frac{\mu}{\varepsilon^d})\leq \hat{E}_{\V^d_\varepsilon}(w_\varepsilon)=d\mu\varepsilon^{-d-1}-\frac{\varepsilon}{p}\left(\mu\varepsilon^{-d}\right)^\frac{p}{2}.
\end{equation}
hence, by $p >2+\frac{4}{d}$, we obtain 
\begin{equation}\label{eql2sup2}
\varepsilon^{d-1}\hat{\mathcal{E}}_{\V^d_\varepsilon}(\frac{\mu}{\varepsilon^d}) \to-\infty, \quad \text{as }\varepsilon\to 0.
\end{equation}
By repeating the arguments as in \cite{HJT}, we know that, for any $\varepsilon>0$ small enough, there exists a positive ground state  $f_\varepsilon\in H^1_{\frac{\mu}{\varepsilon^d}}(\V^d_\varepsilon)$ of  $\hat{E}_{\V^d_\varepsilon}$ at mass $\frac{\mu}{\varepsilon^d}$. Then, for such $\varepsilon>0$, it follows from \eqref{eql2sup1} and the Gagliardo--Nirenberg type inequality \eqref{eqgnv} that

\begin{equation*}
\begin{aligned}
    \varepsilon^{d-1}\hat{\mathcal{E}}_{\V^d_\varepsilon}(\frac{\mu}{\varepsilon^d})=\varepsilon^{d-1}\hat{E}_{\V^d_\varepsilon}(f_\varepsilon)&\geq \frac{\varepsilon^{d-1}}{2}\|\nabla^d_\varepsilon f_{\varepsilon}\|_{\ell^2(\Vd)}^2-\varepsilon^{d-1}\varepsilon C\varepsilon^{\beta_pp/2}\|\nabla^d_\varepsilon f_{\varepsilon}\|_{\ell^2(\Vd)}^{\beta_pp}\| f_{\varepsilon}\|_{\ell^2(\Vd)}^{(1-\beta_p)p}\\
    &\geq 
    \frac{\varepsilon^{d-1}}{2}\|\nabla^d_\varepsilon f_{\varepsilon}\|_{\ell^2(\Vd)}^2-C\varepsilon^{d+\beta_pp/2-(1-\beta_p)pd/2}\|\nabla^d_\varepsilon f_{\varepsilon}\|_{\ell^2(\Vd)}^{\beta_pp}\\
    &= \frac{\varepsilon^{d-1}}{2}\|\nabla^d_\varepsilon f_{\varepsilon}\|_{\ell^2(\Vd)}^2-C\left(\varepsilon^{\frac{d-1}{2}}\|\nabla^d_\varepsilon f_{\varepsilon}\|_{\ell^2(\Vd)}\right)^{\beta_pp}.
    \end{aligned}
\end{equation*}
Since $p>2+\frac{4}{d}$, we have $\beta_pp>2$. \eqref{eql2sup2} implies
$$
\varepsilon^{d-1}\|\nabla^d_\varepsilon f_{\varepsilon}\|_{\ell^2(\Vd)}^2\to+\infty, \quad \text{as }\varepsilon\to 0.
$$
Moreover, by repeating the arguments in the proof of \cite[Lemma 4.4]{Do2}, we conclude 
 \begin{equation}\label{eqth14}
 \norm{\nabla \hat{A}f_\varepsilon}^2_{L^2(\R^d)}=\varepsilon^{d-1}\norm{ \tilde{f}_\varepsilon'}^2_{L^2(\G^d_{\varepsilon})}=\varepsilon^{d-1}\|\nabla^d_\varepsilon f_{\varepsilon}\|_{\ell^2(\Vd)}^2\to+\infty, \quad \text{as }\varepsilon\to 0,
 \end{equation}
 where $\tilde{f}_\varepsilon$ is the extension of $f_\varepsilon$ from $\Vd$ to $\Gd$ as in Section \ref{sec3}. This completes the proof of Theorem \ref{thl2sup}.
\end{proof}
We end this section with the proof of Theorem \ref{th4}.
\begin{proof}[Proof of Theorem \ref{th4}]
Let $d \geq 2$ and $p \in \left(2+\frac{4}{d},2^*\right)$. Let $w_\mu \in H^1_{\mu}(\V^d_1)$ be defined by
$$
w_\mu(x):=     \begin{cases}
       \sqrt{\mu},  &x=0,\\
       0, &x \neq 0.
      \end{cases}
$$
Then, we have 
\begin{equation}\label{eqth41}
    \hat{\mathcal{E}}_{\V^d_1}(\mu) \leq \hat{E}_{\V^d_1}(w_\mu) \leq d\mu- \frac{1}{p}\mu^\frac{p}{2}.
\end{equation}
By \cite[Theorem 1.1]{HJT}, we conclude that, for $\mu>0$ sufficiently large, there exists a positive ground state $f_\mu \in H^1_\mu(\V^d_1)$ of $\hat{E}_{\V^d_1}$ at mass $\mu >0$. Then, by \eqref{eqth41}, we obtain
\begin{equation}\label{eqth42}
    \begin{aligned}
    \mathcal{L}_{\V^d_1}(f_\mu)=\frac{\|f_\mu\|_{\ell^p(\V^d_1)}^p - \|\nabla^d_1 f_\mu\|_{\ell^2(\V^d_1)}^2}{\|f_\mu\|_{\ell^2(\V^d_1)}^2}
    &=\left(1-\frac{2}{p}\right)\frac{\|f_\mu\|_{\ell^p(\V^d_1)}^p}{\|f_\mu\|_{\ell^2(\V^d_1)}^2}-2\frac{\hat{\mathcal{E}}_{\V^d_1}(\mu)}{\|f_\mu\|_{\ell^2(\V^d_1)}^2}\\
    &\geq -2\frac{\hat{\mathcal{E}}_{\V^d_1}(\mu)}{\|f_\mu\|_{\ell^2(\V^d_1)}^2}\\
    &\geq \frac{2}{p}\mu^{\frac{p}{2}-1}-2d \to +\infty \quad \text{as }\mu \to+\infty.
\end{aligned}
\end{equation}
Conversely, by \cite[Theorem 1.1]{HX1} and \cite[Theorem B.4]{HJ}, we know that, for every $\omega>0$, there exists a positive ground state $f_\omega$ of $\hat{J}_{\omega,\V^d_1}$. 
Let $\varepsilon:=\omega^{1/2}$ and $g_\varepsilon \in H^1(\Vd)$ be defined by
$$
g_\varepsilon:=\varepsilon^{2/(2-p)}f_{\omega}\left(\frac{x}{\varepsilon}\right), \quad \forall x \in \Vd.
$$
Then, since $f_\omega \in \hat{\mathcal{N}}_{\omega,\V^d_1}$, we have 
$$
\frac{\|\nabla^d_\varepsilon g_\varepsilon\|_{\ell^2(\Vd)}^2+\varepsilon\|g_\varepsilon\|_{\ell^2(\Vd)}^2}{\varepsilon\|g_\varepsilon\|_{\ell^p(\Vd)}^p}
=\frac{\varepsilon^{(2+p)/(2-p)}\|\nabla^d_1 f_\omega \|_{\ell^2(\V^d_1)}^2+\varepsilon^{(6-p)/(2-p)}\| f_\omega \|_{\ell^2(\V^d_1)}^2}{\varepsilon^{(2+p)/(2-p)}\| f_\omega \|_{\ell^p(\V^d_1)}^p}=1,
$$
that is $g_\varepsilon \in \hat{\mathcal{N}}_{1,\Vd}$, hence
$$
\varepsilon^{(2+p)/(2-p)}\hat{\mathcal{J}}_{\V^d_1}(\omega)=\left(\frac{1}{2}-\frac{1}{p}\right)\varepsilon^{(2+p)/(2-p)}\| f_\omega \|_{\ell^p(\V^d_1)}^p=\left(\frac{1}{2}-\frac{1}{p}\right)\varepsilon\|g_\varepsilon\|_{\ell^p(\Vd)}^p\geq \hat{\mathcal{J}}_{\Vd}(1)
$$
Similarly, for any ground state $g_\varepsilon$ of $\hat{J}_{1,\Vd}$, one can prove that 
$
f_\omega:=\omega^{1/(p-2)}g_\varepsilon(\omega^{1/2} x) \in \hat{\mathcal{N}}_{\omega,\V^d_1}$ 
and 
$$\omega^{(p+2)/(2p-4)}\hat{\mathcal{J}}_{\Vd}(1)=\left(\frac{1}{2}-\frac{1}{p}\right)\omega^{p/(p-2)}\| g_\varepsilon \|_{\ell^p(\Vd)}^p \geq \hat{\mathcal{J}}_{\V^d_1}(\omega)= \left(\frac{1}{2}-\frac{1}{p}\right)\| f_\omega \|_{\ell^p(\V^d_1)}^p.
$$
Thus, we know that $g_\varepsilon$ is a ground state of $\hat{J}_{1,\Vd}$. Then, by  Lemmas \ref{leminfa} and \ref{lemactionbdd}, we have
$$
\| f_\omega \|_{\ell^2(\V^d_1)}^2 = \omega^{2/(p-2)}\| g_\varepsilon \|_{\ell^2(\Vd)}^2 \geq C'\omega^{2/(p-2)-d/2},\quad \text{as }\omega \to 0,
$$
hence, by $p>2+\frac{4}{d}$, we obtain
$$
\| f_\omega \|_{\ell^2(\V^d_1)}^2 \to +\infty \quad \text{as }\omega \to 0.
$$

Therefore, there exist sequences $\{\mu_n\}, \{\omega_n\} \subset \R^+$ and $\{f_n\}\subset H^1(\Vd)$ such that $\mu_n \to +\infty$ and $\omega_n \to 0^+$ as $n \to +\infty$ and, for every $n \in \mathbb{N}^+$, $f_n \in \hat{\mathcal{N}}_{\omega_n,\V^d_1}\cap H^1_{\mu_n}(\V^d_1)$ is a ground state of $\hat{J}_{\omega_n,\V^d_1}$. Thus, $f_n$ is a critical point of $\hat{E}_{\V^d_1}$ in $H^1_{\mu_n}(\V^d_1)$ with
$$
\mathcal{L}_{\V^d_1}(f_n) =\omega_n \to 0, \quad \text{as }n \to +\infty,
$$
and, by $p>2$, for every $n \in \mathbb{N}^+$ large enough, we obtain
$$
\begin{aligned}
\hat{E}_{\V^d_1}(f_n)&=\frac{1}{2}\norm{\nabla^d_1f_n}^2_{\ell^2(\V^d_1)}-\frac{1}{p}\norm{f_n}^p_{\ell^p(\V^d_1)}\\
&=\left(\frac{1}{2}-\frac{1}{p}\right)\norm{\nabla^d_1f_{n}}^2_{\ell^2(\V^d_1)}-\frac{\omega_n}{p}\norm{f_n}^2_{\ell^2(\V^d_1)}\\
&\geq -\frac{\omega_n\mu_n}{p}>d\mu_n- \frac{1}{p}\mu_n^\frac{p}{2}.
\end{aligned}
$$
Comparing with \eqref{eqth41} and \eqref{eqth42}, we complete the proof of Theorem \ref{th4}.
\end{proof}
\section{Proof of Theorem \ref{th5}}\label{sec5}

By \eqref{eqgnv} and \eqref{eqsobov}, for every $d \geq 1$, $q \in (2,2^*]$ and $\varepsilon>0$, there exists a constant $K_{q,\V^d_\varepsilon}>0$, depending only on $d,q$ and $\varepsilon$, such that 
\begin{equation}\label{eqgnnoe}
    \|f\|_{\ell^q(\V_\varepsilon^d)}^q \leq  K_{q,\V^d_\varepsilon}\|\nabla_\varepsilon^d f\|_{\ell^2(\V_\varepsilon^d)}^{\beta_q q}\|f\|^{(1 - \beta_q)q}_{\ell^2(\V_\varepsilon^d)},\quad \forall f \in H^1(\V_\varepsilon^d).
\end{equation}
Let $K_{q,\V^d_\varepsilon}$ be the best constant in \eqref{eqgnnoe}.
For every $f_\varepsilon \in H^1(\V^d_\varepsilon)$, the function defined by $f_1(x):=f_\varepsilon(\varepsilon x)$ belongs to $H^1(\V^d_1)$ and 
$$
 \|\nabla_1^d f_1\|_{\ell^2(\V_1^d)}^2=\varepsilon\|\nabla_\varepsilon^d f_\varepsilon\|_{\ell^2(\V_\varepsilon^d)}^2 \,\,\,  \text{and} \,\,\,\|f_1\|_{\ell^r(\V_1^d)}^r =\|f_\varepsilon\|_{\ell^r(\V_\varepsilon^d)}^r, \quad \forall r \geq 2.
$$
By repeating the arguments in \cite[Lemma 3.1]{Do2}, we obtain that 
\begin{equation}\label{eqgnconst}
    K_{q,\V^d_\varepsilon}=\varepsilon^{\beta_qq/2}K_{q,\V^d_1}.
\end{equation}
We are now in a position to prove Theorem \ref{th5}.
\begin{proof}[Proof of Theorem \ref{th5}]
{\rm (i)} We first prove \eqref{eqth51}. \\
    Let $d\geq 1$ and $q \in (2,2^*)$.
By \cite[Remark 7.1]{Do2}, there exists a function $u \in C^1(\R^d)\cap H^2(\R^d)\cap L^\infty(\R^d)$ such that $Q_{q,\R^d}(u)=K_{q,\R^d}$. Denote by $u_\varepsilon$ the restriction of $u$ to $\Gd$ for every $\varepsilon>0$. Then, by \eqref{eqgnconst} and Lemma \ref{lemequiv}, we have
$$
\begin{aligned}
K_{q,\V^d_1}&=\varepsilon^{-\beta_qq/2}\sup_{f \in H^1(\Vd)}\frac{\norm{f}_{\ell^q(\Vd)}^q}{\|\nabla_\varepsilon^d f\|_{\ell^2(\V_\varepsilon^d)}^{\beta_q q}\|f\|^{(1 - \beta_q)q}_{\ell^2(\V_\varepsilon^d)}}\\
&= \varepsilon^{-\beta_qq/2}\sup_{v \in H^1(\Gd)}\frac{\norm{\hat{v}}_{\ell^q(\Vd)}^q}{\|\nabla_\varepsilon^d \hat{v}\|_{\ell^2(\V_\varepsilon^d)}^{\beta_q q}\|\hat{v}\|^{(1 - \beta_q)q}_{\ell^2(\V_\varepsilon^d)}}\\
 &\geq \varepsilon^{-\beta_qq/2}\sup_{v\in H^1(\Gd)}\frac{\displaystyle\sum_{\vv\in \V}\abs{\hat{v}(\vv)}^q}{\displaystyle\|v'\|_{L^2(\G_\varepsilon^d)}^{\beta_q q}\left({\sum_{\vv\in \V}\abs{\hat{v}(\vv)}^2}\right)^{(1 - \beta_q)q/2}}\\
 &\geq \varepsilon^{-\beta_qq/2}\frac{\displaystyle\sum_{\vv\in \V}\abs{\hat{u}_\varepsilon(\vv)}^q}{\displaystyle\|u_\varepsilon'\|_{L^2(\G_\varepsilon^d)}^{\beta_q q}\left({\sum_{\vv\in \V}\abs{\hat{u}_\varepsilon(\vv)}^2}\right)^{(1 - \beta_q)q/2}},
\end{aligned}
$$
and thus, it follows from \eqref{eqgrar} and \eqref{eqsumvr} that
$$
K_{q,\V^d_1} \geq \frac{\varepsilon^d\displaystyle\sum_{\vv\in \V}\abs{\hat{u}_\varepsilon(\vv)}^q}{\displaystyle\varepsilon^{(d-1)\beta_qq/2}\|u_\varepsilon'\|_{L^2(\G_\varepsilon^d)}^{\beta_q q}\varepsilon^{d(1 - \beta_q)q/2}\left({\sum_{\vv\in \V}\abs{\hat{u}_\varepsilon(\vv)}^2}\right)^{(1 - \beta_q)q/2}} \to Q_{q,\R^d}(u)= K_{q,\R^d},\,\, \text{as } \varepsilon \to 0,
$$
that is \eqref{eqth51}.

$$Q_{q,\V^d_1}=\frac{\norm{f}_{\ell^q(\Vd)}^q}{\|\nabla_1^d f\|_{\ell^2(\V_1^d)}^{\beta_q q}\|f\|^{(1 - \beta_q)q}_{\ell^2(\V_1^d)}}$$ for $f\in H^{1}(\V_1^d)$

Now, let $d \geq 3$. We consider the case that $q =2^*$. By repeating the previous arguments, we have 
$$
K_{2^*,\V^d_1} \geq Q_{2^*,\R^d}(u),\quad  \forall u \in C_c(\R^d)\setminus\{0\}. 
$$
Since $C_c^{\infty}(\R^d)$ is dense in $D^{1,2}(\R^d)$, we conclude that $K_{2^*,\V^d_1} \geq K_{2^*,\R^d}$, which completes the proof of \eqref{eqth51}.
\\
\noindent {\rm (ii)} We now prove that for every  $d\geq 1$ and $q \in (2,2^*)$, if there exists $u \in H^1(\V^d_1)$ such that 
    $
    Q_{q,\V^d_1}(u) \geq K_{q,\R^d},
    $
    then $K_{q,\V^d_1}$ is attained.\\
From  \eqref{eqth51}, we know that for any $d\geq 1$ and $q \in (2,2^*)$, $K_{q,\V^d_1} \geq K_{q,\R^d}$. If $K_{q,\V^d_1}= K_{q,\R^d}$ and there exists
a function $u\in H^1(\V^d_1)$ such that $Q_{q,\V^d_1}(u) = K_{q,\R^d}$, then the conclusion holds immediately.

Now suppose that $K_{q,\V^d_1}>K_{q,\R^d}$ for any $d\geq 1$ and $q \in (2,2^*)$. Choose $\{f_n\}\subset H_{1}^1(\V^d_1)$ such that $Q_{q,\V^d_1}(f_n)=K_{q,\V^d_1}-\frac{1}{n}$ for all sufficiently large $n \in \mathbb{N}^+$. By \cite[Lemma 2.1]{HJ}, we know that $\{f_n\}$ is bounded in $H^1(\V^d_1)$,  and, without loss of generality, we may assume that 
$$
\abs{f_n(0)}=\norm{f_n}_{\ell^\infty(\V^d_1)}.
$$
Hence, up to a subsequence, there exists  $f \in H^1(\V^d_1)$ such that $f_n\rightharpoonup f$ in $H^1(\V^d_1)$ as $n \to +\infty$, and $f_n(\vv) \to f(\vv)$ for every $\vv\in \V^d_1$ as $n \to +\infty$. By Fatou's lemma, we have $m:=\norm{f}^2_{\ell^2(\V^d_1)} \in [0,1]$. On the one hand, if $m \in (0,1)$, then, by arguing as in the proof of \cite[Theorem 2.5]{Do2}, we will obtain $K_{q,\V^d_1}<K_{q,\V^d_1}$, which is a contradiction. On the other hand, if $m=0$, then we have $f\equiv 0$, and thus
\begin{equation}\label{eqprop11}
    \norm{f_n}_{\ell^\infty(\V^d_1)}=\abs{f_n(0)}\to 0, \quad \text{as }n \to +\infty. 
\end{equation}

It follows from $Q_{q,\V^d_1}(f_n)=K_{q,\V^d_1}-\frac{1}{n}$ that
$$
Q_{q,\V^d_1}(f_n)=\frac{\norm{f_n}_{\ell^q(\V^d_1)}^q}{\norm{\nabla^d_1 f_n}_{\ell^2(\V^d_1)}^{(\frac{q}{2}-1)d}}\leq \frac{\norm{f_n}_{\ell^\infty(\V^d_1)}^{q-2}}{\norm{\nabla^d_1 f_n}_{\ell^2(\V^d_1)}^{(\frac{q}{2}-1)d}},
$$
thus, by \eqref{eqprop11}, we obtain that $$\norm{\nabla^d_1 f_n}_{\ell^2(\V^d_1)} \to 0,\quad \text{as }n \to +\infty.$$
Next, for every $n \in \mathbb{N}^+$, set $\varepsilon_n:=\norm{\nabla^d_1 f_n}_{\ell^2(\V^d_1)}$ and define $g_n:\V^d_{\varepsilon_n}\to\mathbb R$ by $g_n(\vv):=f_n(\vv/\varepsilon_n)$ for every $\vv \in \V^d_{\varepsilon_n}$. 
Let $\hat{\mathcal{A}}g_n$ be the extension of $g_n$ to $\R^d$  and $\tilde{g}_n$ be the extension of $g_n$ to $\Gd$ as in Section \ref{sec3}. It is clear that $\tilde{g}_n$ is also the restriction of $\hat{\mathcal{A}}g_n$ to $\Gd$. By Lemma \ref{lemlp} and Lemma \ref{lemh1gv}, we have
\begin{equation}\label{eqsec51}
\begin{aligned}
    \norm{\tilde{g}_n}^2_{L^2(\G^d_{\varepsilon_n})}&\leq d\varepsilon_n\norm{g_n}^2_{\ell^2(\V^d_{\varepsilon_n})} +C\varepsilon_n\norm{g_n}_{\ell^2(\V^d_{\varepsilon_n})}\norm{\nabla^d_{\varepsilon_n} g_n}_{\ell^2(\V^d_{\varepsilon_n})}\\
    &= d\varepsilon_n\norm{f_n}^2_{\ell^2(\V^d_{1})} +C\varepsilon_n\varepsilon_n^{-1/2}\norm{f_n}_{\ell^2(\V^d_{1})}\norm{\nabla^d_1f_n}_{\ell^2(\V^d_1)}\\
    &\leq d\varepsilon_n +C\varepsilon_n^{3/2}\\
    &\leq C\varepsilon_n, \,\, \text{for}\,\, n\in \mathbb{N}^+\,\, \text{large}.
\end{aligned}
\end{equation}
Moreover, by Lemma \ref{lemh1gv} and arguing as in the proof of \cite[Lemma 4.4]{Do2}, we have 
\begin{equation}\label{eqsec5nabla}
    \norm{\nabla \hat{A}g_n}^2_{L^2(\R^d)}=\varepsilon_n^{d-1}\norm{ \tilde{g}_n'}^2_{L^2(\G^d_{\varepsilon_n})}=\varepsilon_n^{d-1}\norm{\nabla^d_{\varepsilon_n}g_n}^2_{\ell^2(\V^d_{\varepsilon_n})}=\varepsilon_n^d,
\end{equation}
and thus, by arguing as \cite[Lemma 4.1]{Do2}, we obtain
$$
\norm{\hat{\mathcal{A}}g_n}^2_{L^2(\R^d)}\leq \frac{\varepsilon^{d-1}_n}{d} \norm{\tilde{g}_n}^2_{L^2(\G^d_{\varepsilon_n})} + \varepsilon_n\norm{\hat{\mathcal{A}}g_n}^2_{L^2(\R^d)} + (\varepsilon_n+\varepsilon_n^2)\norm{\nabla \hat{A}g_n}^2_{L^2(\R^d)},
$$
that is
\begin{equation}\label{eqsec52}
    \norm{\hat{\mathcal{A}}g_n}^2_{L^2(\R^d)}\leq \frac{\varepsilon_n^d}{1-\varepsilon_n}(1+C\varepsilon_n^{1/2})\leq C\varepsilon_n^d, \,\, \text{for}\,\, n\in \mathbb{N}^+\,\, \text{large}.
\end{equation}

Let now $d\in \{1,2\}$ and $q>2$ or $d \geq 3$ and $q \in \left(2,\frac{2^*}{2}+1\right]$. From \eqref{eqsec51}, \eqref{eqsec5nabla}, Lemma \ref{lemlp} and the arguments as in Lemma \ref{lemlL}, we deduce that
$$\begin{aligned}
\norm{\tilde{g}_n}^q_{L^q(\G^d_{\varepsilon_n})}&\geq \varepsilon_n d \norm{g_n}^q_{\ell^q(\V^d_{\varepsilon_n})}-C\varepsilon_n\|\tilde{g}_n\|_{L^{2q-2}(\G^d_{\varepsilon_n})}^{q-1} \|\tilde{g}_n'\|_{L^2(\G^d_{\varepsilon_n})}\\
&\geq \varepsilon_n d \norm{g_n}^q_{\ell^q(\V^d_{\varepsilon_n})}-C\varepsilon_n^{\frac{(q-2)}{2}(d-1)+1}\|\tilde{g}_n\|_{L^{2}(\G^d_{\varepsilon_n})}^{\frac{d+(2-d)(q-1)}{2}} \|\tilde{g}_n'\|_{L^2(\G^d_{\varepsilon_n})}^{\frac{q-2}{2}d+1}\\
&\geq \varepsilon_n d \norm{f_n}^q_{\ell^q(\V^d_{1})}-C\varepsilon_n^{\frac{(q-2)d}{2}+2},
\end{aligned}
$$
thus, by arguing as in \cite[Eq. (24)]{Do2} and using the $d$-dimensional Gagliardo--Nirenberg inequality  \cite[Eq. (4)]{Do2}, \eqref{eqsec5nabla} and \eqref{eqsec52}, we conclude that
\begin{equation}\label{eqsec53}
\begin{aligned}
    \norm{\hat{\mathcal{A}}g_n}^q_{L^q(\R^d)} &\geq \frac{\varepsilon_n^{d-1}}{d}\norm{\tilde{g}_n}^q_{L^q(\G^d_{\varepsilon_n})}-C\varepsilon_n\norm{\hat{\mathcal{A}}g_n}^{q-1}_{L^{2q-2}(\R^d)}\norm{\nabla\hat{\mathcal{A}}g_n}_{L^2(\R^d)}\\
    &\geq \varepsilon_n^d\norm{f_n}^q_{\ell^q(\V^d_{1})}-C\varepsilon_n\norm{\hat{\mathcal{A}}g_n}_{L^{2}(\R^d)}^{\frac{d+(2-d)(q-1)}{2}}\norm{\nabla\hat{\mathcal{A}}g_n}_{L^2(\R^d)}^{\frac{q-2}{2}d+1}\\
    &\geq  \varepsilon_n^d\norm{f_n}^q_{\ell^q(\V^d_{1})}-C\varepsilon_n^{1+qd/2},  \quad \text{for}\,\, n\in \mathbb{N}^+\,\, \text{large}.
\end{aligned}
\end{equation}

Let then $d \geq 3$ and $q \in \left(\frac{2^*}{2}+1,2^*\right)$. Then, again by Lemma \ref{lemlp} and arguing as in \eqref{eqlem2e2p-2}, we have that

$$
\begin{aligned}
  \norm{\tilde{g}_n}^q_{L^q(\G^d_{\varepsilon_n})} &\geq \varepsilon_n d \norm{g_n}^q_{\ell^q(\V^d_{\varepsilon_n})}-C\varepsilon_n\|\tilde{g}_n\|_{L^{2q-2}(\G^d_{\varepsilon_n})}^{q-1} \|\tilde{g}_n'\|_{L^2(\G^d_{\varepsilon_n})}\\
  &\geq \varepsilon_n d \norm{g_n}^q_{\ell^q(\V^d_{\varepsilon_n})}-C\varepsilon_n^{q/2+1}\norm{\tilde{g}_n'}_{L^2(\G^d_{\varepsilon_n})}^{q} \\
  & \geq \varepsilon_n d \norm{g_n}^q_{\ell^q(\V^d_{\varepsilon_n})}- C\varepsilon_n^{q+1}, \quad \text{for}\,\, n\in \mathbb{N}^+\,\, \text{large},
\end{aligned}
$$
thus, by arguing as in \eqref{eqlem2e5}, and using \eqref{eqsec5nabla} and  \cite[Lemma 4.4]{Do2}, we conclude that
\begin{equation}\label{eqsec54}
\begin{aligned}
    \norm{\hat{\mathcal{A}}g_n}^q_{L^q(\R^d)} & \geq \frac{\varepsilon_n^{d-1}}{d}\norm{\tilde{g}_n}^q_{L^q(\G^d_{\varepsilon_n})}-C\varepsilon_n\norm{\hat{\mathcal{A}}g_n}^{q-1}_{L^{2q-2}(\R^d)}\norm{\nabla\hat{\mathcal{A}}g_n}_{L^2(\R^d)}\\
    &\geq \varepsilon_n^d\norm{g_n}^q_{\ell^q(\Vd)}-C\varepsilon_n^{q+d}-C\varepsilon_n^{1+d/2+{d/2}}\left(\varepsilon_n^{q-1}\norm{\tilde{g}_n'}_{L^2(\G^d_{\varepsilon_n})}^{2q-2} \right)^{1/2}\\
    & =\varepsilon_n^d\norm{f_n}^q_{\ell^q(\V^d_{1})}- C\varepsilon_n^{q+d},\quad\text{for}\,\, n\in \mathbb{N}^+\,\, \text{large}.
\end{aligned}
\end{equation}
Since $\{\hat{\mathcal{A}}g_n\}\subset H^1(\R^d)$, it follows from \eqref{eqsec5nabla}-\eqref{eqsec54} that 
$$
\begin{aligned}
    K_{q,\R^d} &\geq \limsup_{n \to +\infty}Q_{q,\R^d}(\hat{\mathcal{A}}g_n)\\
    &= \limsup_{n \to +\infty}\frac{\norm{\hat{\mathcal{A}}g_n}^q_{L^q(\R^d)}}{\norm{\hat{\mathcal{A}}g_n}^{d+(2-d)\frac{q}{2}}_{L^2(\R^d)}\norm{\nabla \hat{\mathcal{A}}g_n}^{(\frac{q}{2}-1)d}_{L^2(\R^d)}}\\
    & \geq \limsup_{n \to +\infty}\frac{\norm{f_n}^q_{\ell^q(\V^d_{1})}}{\varepsilon_n^{(\frac{q}{2}-1)d}}=\limsup_{n \to +\infty}\frac{\norm{f_n}^q_{\ell^q(\V^d_{1})}}{\norm{\nabla^d_1 f_n}^{(\frac{q}{2}-1)d}_{\ell^2(\V^d_{1})}}=K_{q,\V^d_1},
\end{aligned}
$$
which is a contradiction since we assume that $K_{q,\V^d_1}> K_{q,\R^d}$. Therefore, we obtain that $m=1$, which, by \cite[Lemma 2.1]{HJ}, implies that $f_n \to f$ in $H^1(\V^d_1)$. Thus, we have that $Q_{q,\V^d_1}(f)=K_{q,\V^d_1}$, which completes the proof of Theorem \ref{th5}.
\end{proof}
\section{Proofs of Theorems \ref{th6} and \ref{th7}}\label{seccle}
Set
\[
S_{\V^d_1}:=K_{2^*,\V^d_1}^{-\frac{2}{2^*}}=\inf_{u\in D^{1,2}(\V^d_1)\setminus \{0\} }\frac{\norm{\nabla^d_1 u}^{2}_{\ell^2(\V^d_1)}}{\norm{u}^{2}_{\ell^{2^*}(\V^d_1)}}\quad \text{and}\quad S_{\R^d}:=K_{2^*,\R^d}^{-\frac{2}{2^*}}=\inf_{u\in D^{1,2}(\R^d)\setminus \{0\} }\frac{\norm{\nabla u}^{2}_{L^2(\R^d)}}{\norm{u}^{2}_{L^{2^*}(\R^d)}}.
\]

We first assume that \eqref{eqk2*} holds, i.e. 
$$S_{\V^d_1}<S_{\R^d},$$
and prove that the best constant \(S_{\V_1^d}\) is attained by a positive function
\(u\in D^{1,2}(\V_1^d)\) satisfying
\[
\frac{\norm{\nabla^d_1 u}^{2}_{\ell^2(\V^d_1)}}{\norm{u}^{2}_{\ell^{2^*}(\V^d_1)}}=S_{\V^d_1}
\]
and
\[
-\Delta_1^d u=u^{2^*-1}
\quad\text{in }\V_1^d.
\]
\begin{lemma}\label{lem:strict_discrete_sobolev_constant}
If $S_{\V^d_1}<S_{\mathbb R^d}$, then \(S_{\V_1^d}\) is attained by a positive function
\(u\in D^{1,2}(\V_1^d)\) satisfying
\[
\frac{\norm{\nabla^d_1 u}^{2}_{\ell^2(\V^d_1)}}{\norm{u}^{2}_{\ell^{2^*}(\V^d_1)}}=S_{\V^d_1}
\]
and
\[
-\Delta_1^d u=u^{2^*-1}
\quad\text{in }\V_1^d.
\]
\end{lemma}
\begin{proof}
     Let $\{u_n\} \subset C_c(\V^d_1)$ be a  minimizing sequence for $S_{\V^d_1}$. We may assume that  $u_n$ is nonnegative and $\|u_n\|_{\ell^{2^*}(\V^d_1)}=1$ for all $n \in \mathbb{N}^{+}$.  Moreover, $u_n$ satisfies
     \begin{equation}\label{eqpf2*1}
         \norm{\nabla^d_1 u_n}_{\ell^2(\V^d_1)}^2\to S_{\V^d_1}>0 \quad \text{as }n \to +\infty.
     \end{equation}
       By repeating the argument in the proof of  \cite[Lemma 4.4]{Do2}, we obtain
\begin{equation}\label{eqgradient}
\|\nabla \hat{\mathcal{A}}u_n\|_{L^2(\R^d)}^2= \|\nabla^d_1 u_n\|_{\ell^2(\V^d_1)}^2,
\end{equation}
where $\hat{\mathcal{A}}$ is the extension operator from $\V^d_1$ to $\R^d$ introduced in Section 3.
    
    We now prove that $\norm{u_n}_{\ell^{\infty}(\V^d_1)}\not \to 0$ as $n \to +\infty$. We argue by contradiction and suppose that $\norm{u_n}_{\ell^{\infty}(\V^d_1)}  \to 0$ as $n \to +\infty$.
    By using Lemmas \ref{lemlp} and \ref{lemsobo} and repeating the argument in the proof of Lemma \ref{lemlL}, we have 
    \begin{equation}\label{equnzg2*}
\abs{\norm{\hat{\mathcal{A}}u_n}_{L^{2^*}(\R^d)}^{2^*}-\norm{u_n}_{\ell^{2^*}(\V^d_1)}^{2^*}}\leq C \|\nabla^d_1 u_n\|_{\ell^2(\V^d_1)}^{\frac{2^*}{2}+1}\|u_n\|_{\ell^\infty(\V^d_1)}^{\frac{2^*}{2}-1}.
\end{equation}
    Indeed, on the one hand, by Lemma \ref{lemlp}, 
\[
                 \left| d\norm{u_n}_{\ell^{2^*}(\V^d_1)}^{2^*} - \norm{\tilde{u}_n}_{L^{2^*}(\G^d_1)}^{2^*} \right| 
    \leq C\|\tilde{u}_n\|_{L^{2\cdot2^*-2}(\G^d_1)}^{2^*-1} \|\tilde{u}_n'\|_{L^2(\G^d_1)}.
\]
    Then, Lemma \ref{lemsobo} and \eqref{eqgradient} imply that 
    \[
           \left| d\norm{u_n}_{\ell^{2^*}(\V^d_1)}^{2^*} - \norm{\tilde{u}_n}_{L^{2^*}(\G^d_1)}^{2^*} \right| 
    \leq C\|\tilde{u}_n\|_{L^{2\cdot2^*-2}(\G^d_1)}^{2^*-1} \|\tilde{u}_n'\|_{L^2(\G^d_1)}\leq  C \|\nabla^d_1 u_n\|_{\ell^2(\V^d_1)}^{\frac{2^*}{2}+1}\|u_n\|_{\ell^\infty(\V^d_1)}^{\frac{2^*}{2}-1}.
    \]
    On the other hand, by repeating the arguments in the proof of \cite[Eq. (59)]{Do2} and \cite[Lemma 4.4]{Do2}, we have
    \[
    \begin{aligned}
    \left|{\|\tilde{u}_n\|_{L^{2^*}(\G^d_1)}^{2^*}}-d\norm{\hat{\mathcal{A}}u_n}^{2^*}_{L^{2^*}(\R^d)}\right|
    &\leq C\norm{\hat{\mathcal{A}}u_n}^{2^*-1}_{L^{2\cdot2^*-2}(\R^d)}\norm{\nabla \hat{\mathcal{A}}u_n}_{L^2(\R^d)}\\
    &\leq C \| u_n\|_{\ell^{2 \cdot 2^*-2}(\V^d_1)}^{2^*-1}\norm{\tilde{u}_n'}_{L^2(\G^d_1)}\\
    &\leq C \| u_n\|_{\ell^{2^*}(\V^d_1)}^{2^*/2}\| u_n\|_{\ell^\infty(\V^d_1)}^{\frac{2^*}{2}-1}\norm{\tilde{u}_n'}_{L^2(\G^d_1)}
    \end{aligned}
    \]
    Hence, \eqref{eqgradient} leads to 
    \[
\left|{\|\tilde{u}_n\|_{L^{2^*}(\G^d_1)}^{2^*}}-d\norm{\hat{\mathcal{A}}u_n}^{2^*}_{L^{2^*}(\R^d)}\right|\leq C \| u_n\|_{\ell^{2^*}(\V^d_1)}^{2^*/2}\| u_n\|_{\ell^\infty(\V^d_1)}^{\frac{2^*}{2}-1}\norm{\tilde{u}_n'}_{L^2(\G^d_1)} \leq C \|\nabla^d_1 u_n\|_{\ell^2(\V^d_1)}^{\frac{2^*}{2}+1}\|u_n\|_{\ell^\infty(\V^d_1)}^{\frac{2^*}{2}-1}.
    \]
    Since $\norm{u_n}_{\ell^{\infty}(\V^d_1)}  \to 0$  as $n \to +\infty$, by \blue{\eqref{eqpf2*1},} \eqref{eqgradient} and \eqref{equnzg2*}, 
    \[\norm{\hat{\mathcal{A}}u_n}_{L^{2^*}(\R^d)}^{2^*} \to 1 \text{ as } n \to +\infty.
    \]
Therefore, by  \eqref{eqpf2*1} and \eqref{eqgradient}, we conclude that
\[
S_{\V^d_1} =\frac{\norm{\nabla^d_1 u_n}^{2}_{\ell^2(\V^d_1)}}{\norm{u_n}^{2}_{\ell^{2^*}(\V^d_1)}}+o(1)=\frac
{\|\nabla \hat{\mathcal{A}}u_n\|_{L^2(\R^d)}^{2}}{\|\hat{\mathcal{A}}u_n\|_{L^{2^*}(\mathbb R^d)}^{2}}+o(1)\geq  S_{\R^d}+o(1),
\]
which contradicts the assumption $S_{\V^d_1}<S_{\mathbb R^d}$. We have proved that
    \begin{equation}\label{eq:nonvanishing-revised}
        \limsup_{n\to+\infty}\|u_n\|_{\ell^\infty(\V^d_1)}>0.
    \end{equation}
Up to a subsequence, choose $x_n\in\V^d_1$ and $\eta>0$ such that
    $u_n(x_n)\geq\eta$, and define $w_n(x):=u_n(x+x_n)$. We have that 
    $$
  \|\nabla^d_1 w_n\|_{\ell^2(\V^d_1)}=\|\nabla^d_1 u_n\|_{\ell^2(\V^d_1)},\quad  \|w_n\|_{\ell^{2^*}(\V^d_1)}=\|u_n\|_{\ell^{2^*}(\V^d_1)}=1.
    $$
Thus, after passing to a further subsequence there exists $w\in D^{1,2}(\V^d_1)$ such that
    $w_n\rightharpoonup w$ in $D^{1,2}(\V^d_1)$ and $w_n(x)\to w(x)$ for every $x\in\V^d_1$. In
    particular, $w(0)\geq\eta$, so $w\not\equiv0$.

    Due to $w_n\rightharpoonup w$ in $D^{1,2}(\V^d_1)$, it is clear that
    \begin{equation}\label{eq:energy-splitting-revised}
        \|\nabla^d_1 w_n\|_{\ell^2(\V^d_1)}^2
        =\|\nabla^d_1(w_n-w)\|_{\ell^2(\V^d_1)}^2
         +\|\nabla^d_1 w\|_{\ell^2(\V^d_1)}^2+o(1).
    \end{equation}
    The Br\'ezis--Lieb lemma yields
    \begin{equation}\label{eq:mass-splitting-revised}
        1=\|w_n-w\|_{\ell^{2^*}(\V^d_1)}^{2^*}+\|w\|_{\ell^{2^*}(\V^d_1)}^{2^*}+o(1).
    \end{equation}
    Set $a:=\|w\|_{\ell^{2^*}(\V^d_1)}^{2^*}\in(0,1]$. Applying the definition of
    $S_{\V^d_1}$ to $w$ and $w_n-w$, and using
    \eqref{eq:energy-splitting-revised}--\eqref{eq:mass-splitting-revised},
    we obtain
    \[
    \begin{aligned}
        S_{\V^d_1}
        &=\lim_{n\to\infty}\|\nabla^d_1 w_n\|_{\ell^2(\V^d_1)}^2\\
        &\geq S_{\V^d_1}
        \left(a^{2/2^*}+(1-a)^{2/2^*}\right).
    \end{aligned}
    \]
    Since $2/2^*=(d-2)/d\in(0,1)$, one has
    $a^{2/2^*}+(1-a)^{2/2^*}>1$ whenever $0<a<1$. Hence $a=1$. It follows from
    weak lower semicontinuity and the discrete Sobolev inequality that
    \[
        S_{\V^d_1}
        \leq \|\nabla^d_1 w\|_{\ell^2(\V^d_1)}^2
        \leq \liminf_{n\to\infty}\|\nabla^d_1 w_n\|_{\ell^2(\V^d_1)}^2
        =S_{\V^d_1}.
    \]
    Thus $w$ attains $S_{\V^d_1}$.
Therefore, by the Lagrange multiplier theorem applied to the constrained minimizer $w$, for some $\lambda>0$, $w$ solves
\[
-\Delta^d_1 w= \lambda w^{2^*-1} \quad \text{in }\V^d_1,
\]
and 
\[
\lambda=\frac{\norm{\nabla^d_1w}^2_{\ell^{2}(\V^d_1)}}{\norm{w}^{2^*}_{\ell^{2^*}(\V^d_1)}}=S_{\V^d_1}>0.
\]
Moreover, the connectivity of $\Z^d_1$ implies that $w(\vv)>0$ for all $\vv \in \V^d_1$.
Let $u:=\lambda^{\frac{1}{2^*-2}}w$. Then we conclude that $u>0$ satisfies \[
\frac{\norm{\nabla^d_1 u}^{2}_{\ell^2(\V^d_1)}}{\norm{u}^{2}_{\ell^{2^*}(\V^d_1)}}=S_{\V^d_1}
\]
and
\[
-\Delta_1^d u=u^{2^*-1}
\quad\text{in }\V_1^d.
\]
\end{proof}
We now prove $S_{\V^d_1}<S_{\mathbb R^d}$.

\begin{lemma}\label{lem:strict_comparison_sobolev_constants}
    For $d \geq 3$, 
    \[
S_{\V^d_1}<S_{\mathbb R^d}.
\]
\end{lemma}
\begin{proof}
Recall that the sharp Sobolev constant in $\mathbb R^d$ is given by
\begin{equation}\label{eq:sharp-euclidean-sobolev-constant}
S_{\mathbb R^d}
=
d(d-2)\pi
\left(
\frac{\Gamma(d/2)}{\Gamma(d)}
\right)^{2/d}.
\end{equation}

We first consider the case $d\geq4$. Let $\delta_0:\V^d_1\to\mathbb R$
be defined by
\[
\delta_0(x)
=
\begin{cases}
1, & x=0,\\
0, & x\neq0.
\end{cases}
\]
Since the origin has exactly $2d$ nearest neighbors, we have
\[
\|\nabla\delta_0\|_{\ell^2(\V^d_1)}^2=2d
\]
and
\[
\|\delta_0\|_{\ell^{2^*}(\V^d_1)}=1.
\]
Therefore,
\begin{equation}\label{eq:delta-test-lattice-sobolev}
S_{\V^d_1}
\leq
\frac{\|\nabla\delta_0\|_{\ell^2(\V^d_1)}^2}
{\|\delta_0\|_{\ell^{2^*}(\V^d_1)}^2}
=
2d.
\end{equation}

We claim that
\begin{equation}\label{eq:euclidean-sobolev-larger-than-degree}
S_{\mathbb R^d}>2d
\quad\text{for every }d\geq4.
\end{equation}

We first prove \eqref{eq:euclidean-sobolev-larger-than-degree} when
$d=2m$ for some $m\geq2$. By
\eqref{eq:sharp-euclidean-sobolev-constant},
\[
S_{\mathbb R^{2m}}
=
4m(m-1)\pi
\left(
\frac{\Gamma(m)}{\Gamma(2m)}
\right)^{1/m}.
\]
Since
\[
\Gamma(m)=(m-1)!
\quad\text{and}\quad
\Gamma(2m)=(2m-1)!,
\]
we obtain
\[
\begin{aligned}
\frac{S_{\mathbb R^{2m}}}{4m}
&=
(m-1)\pi
\left(
\frac{(m-1)!}{(2m-1)!}
\right)^{1/m}
\\
&=
\frac{(m-1)\pi}
{\bigl(m(m+1)\cdots(2m-1)\bigr)^{1/m}}.
\end{aligned}
\]
For every $m\geq2$, since $\pi>3$, we have
\[
2m-1<3(m-1)<\pi(m-1).
\]
Consequently,
\[
m(m+1)\cdots(2m-1)
<
\bigl(\pi(m-1)\bigr)^m,
\]
and hence
\[
\frac{S_{\mathbb R^{2m}}}{4m}>1.
\]
Therefore,
\[
S_{\mathbb R^{2m}}>4m=2d.
\]

We now prove \eqref{eq:euclidean-sobolev-larger-than-degree} when
$d=2m+1$ for some $m\geq2$. Using
\[
\Gamma\left(m+\frac12\right)
=
\frac{(2m)!\sqrt{\pi}}{4^m m!},
\]
we obtain
\[
\frac{\Gamma(d/2)}{\Gamma(d)}
=
\frac{\Gamma(m+1/2)}{\Gamma(2m+1)}
=
\frac{\sqrt{\pi}}{4^m m!}.
\]
It follows from \eqref{eq:sharp-euclidean-sobolev-constant} that
\[
\frac{S_{\mathbb R^{2m+1}}}{2(2m+1)}
=
\frac{(2m-1)\pi}{2}
\left(
\frac{\sqrt{\pi}}{4^m m!}
\right)^{2/(2m+1)}.
\]
Set
\[
A_m:=\frac{(2m-1)\pi}{2}.
\]
Then
\[
\left(
\frac{S_{\mathbb R^{2m+1}}}{2(2m+1)}
\right)^{(2m+1)/2}
=
\frac{A_m^{m+1/2}\sqrt{\pi}}{4^m m!}.
\]

Suppose first that $m\geq4$. By the arithmetic--geometric mean
inequality,
\[
m!
=
\prod_{j=1}^m j
\leq
\left(\frac{m+1}{2}\right)^m.
\]
Moreover,
\[
A_m>2(m+1).
\]
Indeed, since $m\geq4$,
\[
\pi(2m-1)-4(m+1)
\geq
7\pi-20>0.
\]
Therefore,
\[
\frac{A_m^{m+1/2}\sqrt{\pi}}{4^m m!}
\geq
\frac{A_m^{m+1/2}\sqrt{\pi}}
{\bigl(2(m+1)\bigr)^m}
>
A_m^{1/2}\sqrt{\pi}
>1.
\]
Thus,
\[
S_{\mathbb R^{2m+1}}>2(2m+1)
\quad\text{for every }m\geq4.
\]

It remains to consider $m=2$ and $m=3$. If $m=2$, then
\[
A_2=\frac{3\pi}{2}>4,
\]
and hence
\[
\left(
\frac{S_{\mathbb R^5}}{10}
\right)^{5/2}
=
\frac{A_2^{5/2}\sqrt{\pi}}{32}
>
\frac{4^{5/2}\sqrt{\pi}}{32}
=
\sqrt{\pi}
>1.
\]
Therefore,
\[
S_{\mathbb R^5}>10.
\]

If $m=3$, then
\[
A_3=\frac{5\pi}{2}>7,
\]
and hence
\[
\left(
\frac{S_{\mathbb R^7}}{14}
\right)^{7/2}
=
\frac{A_3^{7/2}\sqrt{\pi}}{384}
>
\frac{7^{7/2}}{384} >1.
\]
Therefore, 
we obtain
\[
S_{\mathbb R^7}>14.
\]
This proves \eqref{eq:euclidean-sobolev-larger-than-degree} for every
$d\geq4$. Combining
\eqref{eq:delta-test-lattice-sobolev} and
\eqref{eq:euclidean-sobolev-larger-than-degree}, we conclude that
\[
S_{\V^d_1}<S_{\mathbb R^d}
\quad\text{for every }d\geq4.
\]

It remains to consider the case $d=3$. In this case,
\[
2^*=6,
\]
and \eqref{eq:sharp-euclidean-sobolev-constant} yields
\begin{equation}\label{eq:sharp-euclidean-sobolev-dimension-three}
S_{\mathbb R^3}
=
3\left(\frac{\pi}{2}\right)^{4/3}.
\end{equation}

For $a\in(0,1)$, define
\[
u_a(x):=a^{|x|_1},
\qquad
|x|_1:=|x_1|+|x_2|+|x_3|,
\qquad
x=(x_1,x_2,x_3)\in\V^3_1.
\]
The exponential decay of $u_a$ implies that
\[
u_a\in D^{1,2}(\V^3_1)\cap\ell^2(\V^3_1)\cap\ell^6(\V^3_1).
\]
Moreover, if
\[
u_{a,R}(x)
:=
\begin{cases}
u_a(x), & |x|_1\leq R,\\
0, & |x|_1>R,
\end{cases}
\]
then
\begin{equation}\label{equard12}
u_{a,R}\to u_a
\quad\text{in }D^{1,2}(\V^3_1)
\end{equation}
and
\begin{equation}\label{equarl6}
u_{a,R}\to u_a
\quad\text{in }\ell^6(\V^3_1)
\end{equation}
as $R\to+\infty$. Indeed, by Tonelli's theorem,
\begin{equation}\label{equal2}
\begin{aligned}
\|u_a\|_{\ell^2(\V^3_1)}^2
&=
\sum_{(x_1,x_2,x_3)\in\V^3_1}
a^{2(|x_1|+|x_2|+|x_3|)}
\\
&=
\prod_{j=1}^3
\left(
\sum_{n\in\mathbb Z}a^{2|n|}
\right)
\\
&=\left(1+
\frac{2a^2}{1-a^2}
\right)^3=
\left(
\frac{1+a^2}{1-a^2}
\right)^3.
\end{aligned}
\end{equation}
By repeating the argument in the proof of \cite[Lemma 2.1]{HJ}, we have
\begin{equation}\label{eqd12l2}
\|\nabla_1^du\|_{\ell^2(\V_1^d)}\leq C\|u\|_{\ell^2(\V_1^d)} \quad \text{for all }u\in C_c(\V_1^d).
\end{equation}
Therefore, by the definition of $D^{1,2}(\V^3_1)$, since $\ell^2(\V^d_1)$ is the completion of $C_c(\V_1^d)$ under the norm $\|u\|_{\ell^2(\V_1^d)}$, we know $u \in D^{1,2}(\V^3_1)$.
Moreover, by  Lebesgue's dominated convergence theorem, 
\[
u_{a,R}\to u_a
\quad\text{in }\ell^{2}(\V^3_1),
\]
which then leads to \eqref{equard12} and \eqref{equarl6} by \eqref{eqd12l2} and the interpolation inequality $\norm{w}_{\ell^{6}(\V^d_1)}\leq \norm{w}_{\ell^{2}(\V^d_1) }$ for all $w \in \ell^2(\V^d_1)$, respectively.

We first compute its $\ell^6$ norm. Arguing as in the computation of \eqref{equal2}, we obtain
\begin{equation}\label{eq:l6-norm-exponential-test}
\begin{aligned}
\|u_a\|_{\ell^6(\V^3_1)}^6
&=
\sum_{(x_1,x_2,x_3)\in\V^3_1}
a^{6(|x_1|+|x_2|+|x_3|)}
\\
&=
\prod_{j=1}^3
\left(
\sum_{n\in\mathbb Z}a^{6|n|}
\right)
\\
&=
\left(
\frac{1+a^6}{1-a^6}
\right)^3.
\end{aligned}
\end{equation}
Consequently,
\begin{equation}\label{eq:l6-norm-squared-exponential-test}
\|u_a\|_{\ell^6(\V^3_1)}^2
=
\frac{1+a^6}{1-a^6}.
\end{equation}

We next compute its discrete Dirichlet energy. For every
$i\in\{1,2,3\}$, the separation of variables gives
\[
\begin{aligned}
&\sum_{x\in\V^3_1}
|u_a(x+e_i)-u_a(x)|^2
\\
&\quad=
\left(
\sum_{n\in\mathbb Z}
|a^{|n+1|}-a^{|n|}|^2
\right)
\left(
\sum_{m\in\mathbb Z}a^{2|m|}
\right)^2,
\end{aligned}
\]
where $e_i$ is the unit vector in the $i$-th coordinate. On the one hand, the first one-dimensional factor satisfies
\[
\begin{aligned}
\sum_{n\in\mathbb Z}
|a^{|n+1|}-a^{|n|}|^2
&=
2\sum_{n=0}^{\infty}|a^{n+1}-a^n|^2
\\
&=
2(1-a)^2\sum_{n=0}^{\infty}a^{2n}
\\
&=
\frac{2(1-a)^2}{1-a^2}
=\frac{2-2a}{1+a}.
\end{aligned}
\]
On the other hand,
\[
\begin{aligned}
\sum_{m\in\mathbb Z}a^{2|m|}
&=
1+2\sum_{m=1}^{\infty}a^{2m}
\\
&=
1+\frac{2a^2}{1-a^2}
=
\frac{1+a^2}{1-a^2}.
\end{aligned}
\]
Summing over $i=1,2,3$, we obtain
\begin{equation}\label{eq:energy-exponential-test}
\|\nabla u_a\|_{\ell^2(\V^3_1)}^2
=\frac{6-6a}{1+a}
\left(
\frac{1+a^2}{1-a^2}
\right)^2.
\end{equation}

We now choose
\[
a=\frac14.
\]
By \eqref{eq:energy-exponential-test},
\begin{equation}\label{eq:energy-exponential-test-one-fourth}
\|\nabla u_{1/4}\|_{\ell^2(\V^3_1)}^2=
\frac{6-\frac32}{1+\frac14}
\left(
\frac{1+\frac1{16}}{1-\frac1{16}}
\right)^2
=
\frac{578}{125}.
\end{equation}
Similarly, by \eqref{eq:l6-norm-squared-exponential-test},
\begin{equation}\label{eq:l6-norm-exponential-test-one-fourth}
\|u_{1/4}\|_{\ell^6(\V^3_1)}^2
=
\frac{1+4^{-6}}{1-4^{-6}}
=
\frac{4097}{4095}.
\end{equation}
Combining
\eqref{eq:energy-exponential-test-one-fourth} and
\eqref{eq:l6-norm-exponential-test-one-fourth}, we obtain
\begin{equation}\label{eq:sobolev-quotient-exponential-test}
\frac{\|\nabla u_{1/4}\|_{\ell^2(\V^3_1)}^2}
{\|u_{1/4}\|_{\ell^6(\V^3_1)}^2}=
\frac{27846}{6025}.
\end{equation}
It follows that
\begin{equation}\label{eq:lattice-sobolev-upper-bound-dimension-three}
S_{\V^3_1}
\leq
\frac{27846}{6025}.
\end{equation}
Since $\pi>3$, it follows that
\begin{equation}\label{eq:strict-comparison-dimension-three}
\frac{27846}{6025}
<
\frac{24}{5}
=
3\cdot\frac85
<
3\left(\frac32\right)^{4/3}
<
3\left(\frac{\pi}{2}\right)^{4/3}
=
S_{\mathbb R^3}.
\end{equation}
Combining
\eqref{eq:lattice-sobolev-upper-bound-dimension-three} and
\eqref{eq:strict-comparison-dimension-three}, we conclude that
\[
S_{\V^3_1}<S_{\mathbb R^3}.
\]
The proof is complete.
\end{proof}
We are in a position to prove Theorem \ref{th6}.
\begin{proof}[Proof of Theorem \ref{th6}]
    By Lemma \ref{lem:strict_comparison_sobolev_constants}, one has $S_{\V^d_1} < S_{\mathbb R^d}$. The conclusion, including strict positivity of a minimizer, then follows from Lemma \ref{lem:strict_discrete_sobolev_constant}.
\end{proof}
\begin{remark}\label{reclesl}
    We now explain why it is not possible to investigate the singular limit of action ground states for critical Lane--Emden equations \eqref{eqcle} on lattice graphs. For every $d \geq 3$ and every $\varepsilon > 0$, we introduce the action functional $\hat{J}_{\V_\varepsilon^d} : D^{1,2}(\V_\varepsilon^d) \to \mathbb{R}$
\begin{equation*}
    \hat{J}_{\V_\varepsilon^d}(u) := \frac{1}{2}\|\nabla_\varepsilon^d u\|_{\ell^2(\V_\varepsilon^d)}^2  -\frac{\varepsilon}{{2^*}}\|u\|_{\ell^{2^*}(\V_\varepsilon^d)}^{2^*}
\end{equation*}
and the associated Nehari manifold
\begin{equation*}
    \hat{\mathcal{N}}_{ \V_\varepsilon^d} := \left\{ u \in D^{1,2}(\V_\varepsilon^d)\backslash \{0\}  : \|\nabla_\varepsilon^d u\|_{\ell^2(\V_\varepsilon^d)}^2 = \varepsilon\|u\|_{\ell^{2^*}(\V_\varepsilon^d)}^{2^*} \right\}.
\end{equation*}
For $d\geq 3$, define the action functional $J_{\mathbb{R}^d}:D^{1,2}(\R^d) \to \mathbb{R}$
\[
    J_{\mathbb{R}^d}(v) := \frac{1}{2}\|\nabla v\|_{L^2(\mathbb{R}^d)}^2- \frac{1}{2^*}\|v\|_{L^{2^*}(\mathbb{R}^d)}^{2^*}
\]
and the associated Nehari manifold
\begin{align*}
    \mathcal{N}_{\mathbb{R}^d} := \left\{ v \in D^{1,2}(\R^d) \backslash \{0\} : \|\nabla v \|_{L^2(\mathbb{R}^d)}^2  = \|v\|_{L^{2^*}(\mathbb{R}^d)}^{2^*} \right\}.
\end{align*}
A standard computation using the Sobolev inequality and the Nehari identity gives
\[
\inf_{v\in\mathcal N_{\mathbb R^d}}J_{\mathbb R^d}(v)
=\frac{1}{d}S_{\mathbb R^d}^{d/2}.
\]

Let $u$ be the minimizer of $S_{\V_1^d}$ obtained by Theorem \ref{th6}, and let $u_\varepsilon$ be defined by
\[
u_\varepsilon:=\varepsilon^{2/(2-2^*)}u\left(x/\varepsilon\right), \quad \forall x \in \Vd.
\] 
Recall that, in the proof of Lemma \ref{lem:strict_discrete_sobolev_constant}, $w=\lambda^{-\frac{1}{2^*-2}}u$ and $\|\nabla_1^d w \|_{\ell^2(\V_1^d)}^2=S_{\V^d_1}$.
Then, $u_\varepsilon \in  \hat{\mathcal{N}}_{ \V_\varepsilon^d}$ and 
\[
\begin{aligned}
 \hat{J}_{\V_\varepsilon^d}(u_\varepsilon)&=\left(\frac{1}{2}-\frac{1}{2^*}\right)\|\nabla_\varepsilon^d u_\varepsilon\|_{\ell^2(\V_\varepsilon^d)}^2\\
 &=\left(\frac{1}{2}-\frac{1}{2^*}\right)\varepsilon^{(2^*+2)/(2-2^*)}\|\nabla_1^d u \|_{\ell^2(\V_1^d)}^2 \\
 &= \frac{\varepsilon^{1-d}}{d} S_{\V^d_1}^{\frac{2}{2^*-2}}\|\nabla_1^d w \|_{\ell^2(\V_1^d)}^2=\frac{\varepsilon^{1-d}}{d} S_{\V^d_1}^{\frac{2^*}{2^*-2}}= \frac{\varepsilon^{1-d}}{d} S_{\V^d_1}^{\frac{2}{2^*-2}}\|\nabla_1^d w \|_{\ell^2(\V_1^d)}^2=\frac{\varepsilon^{1-d}}{d} S_{\V^d_1}^{d/2}.
\end{aligned}
\]
Thus, by Lemma \ref{lem:strict_comparison_sobolev_constants}, we conclude
\[
\varepsilon^{d-1}\inf_{u \in \hat{\mathcal{N}}_{ \V_\varepsilon^d}}\hat{J}_{\V_\varepsilon^d}(u)=\frac{1}{d} S_{\V^d_1}^{d/2}<\frac{1}{d}S_{\mathbb R^d}^{d/2}= \inf_{v\in\mathcal N_{\mathbb R^d}}J_{\mathbb R^d}(v),
\]
which implies that
\[
\liminf_{\varepsilon\to 0^+}\left|\varepsilon^{d-1}\inf_{u \in \hat{\mathcal{N}}_{ \V_\varepsilon^d}}\hat{J}_{\V_\varepsilon^d}(u)-\inf_{v\in\mathcal N_{\mathbb R^d}}J_{\mathbb R^d}(v)\right|>0.
\]
Therefore, it is impossible to investigate the singular limit of action ground states for critical Lane--Emden equations \eqref{eqcle} on lattice graphs.

Since there is a gap between $S_{\V^d_1}$ and $S_{\R^d}$, we conjecture that \eqref{eqcle} admits higher-energy solutions and that these solutions exhibit singular limit behavior.
\end{remark}
We now provide the proof of Theorem \ref{th7}.
\begin{proof}[Proof of Theorem \ref{th7}]
We first prove that, for every \(d\geq 3\), there exists
\(\delta=\delta(d)>0\) such that
\[
K_{q,\V_1^d}>K_{q,\mathbb R^d}
\quad\text{for every }q\in(2^*-\delta,2^*).
\]
By Lemma \ref{lem:strict_comparison_sobolev_constants} and the definition of $K_{2^*,\V^d_1}$, there exist $u\in C_c(\V^d_1)$ and $\nu>0$ such that
\[
K_{2^*,\R^d}+3\nu<Q_{2^*,\V^d_1}(u).
\]
We first justify the continuity of the Euclidean best constants at $2^*$. For $q\in(2,2^*)$, let $\beta_q=\frac{d(q-2)}{2q}$. By interpolation and the critical Sobolev inequality,
\[
\|v\|_{L^q(\R^d)}^q
\leq \|v\|_{L^2(\R^d)}^{(1-\beta_q)q}
       \|v\|_{L^{2^*}(\R^d)}^{\beta_q q}
\leq K_{2^*,\R^d}^{\frac{\beta_q q}{2^*}}
       \|v\|_{L^2(\R^d)}^{(1-\beta_q)q}
       \|\nabla v\|_{L^2(\R^d)}^{\beta_q q},
\]
so that
\[
K_{q,\R^d}\leq K_{2^*,\R^d}^{\frac{\beta_q q}{2^*}}.
\]
Conversely, for every $\eta>0$, choose $\phi\in C_c^\infty(\R^d)\setminus\{0\}$ such that
$Q_{2^*,\R^d}(\phi)>K_{2^*,\R^d}-\eta$. Since
$Q_{q,\R^d}(\phi)\to Q_{2^*,\R^d}(\phi)$ as $q\nearrow2^*$, we obtain
\[
\lim_{q\nearrow2^*}K_{q,\R^d}=K_{2^*,\R^d}.
\]
Since $u$ has finite support, we also have
\[
\lim_{q\nearrow2^*}Q_{q,\V^d_1}(u)=Q_{2^*,\V^d_1}(u).
\]
Hence there exists $\delta:=\delta(d)>0$ such that, for every $q\in(2^*-\delta,2^*)$,
\[
K_{q,\R^d}<K_{2^*,\R^d}+\nu
<Q_{2^*,\V^d_1}(u)-2\nu
<Q_{q,\V^d_1}(u).
\]
By Theorem \ref{th5}, we know that,  for every \(d\geq 3\) and every  $q \in (2^*-\delta,2^*)$, $K_{q,\V^d_1}$ is attained by some $w_q \in H^1_1(\V^d_1)$. Replacing $w_q$ by $\abs{w_q}$, we know that $\abs{w_q}$ is still a maximizer. Without loss of generality, we assume that $w_q$ is nonnegative. By the Lagrange multiplier theorem applied to the constrained maximizer $w_q$, for some $\lambda_q>0$, $w_q$ solves
\[
-\Delta_1^d w_q
+\lambda_q w_q
=
\frac{2q}{d(q-2)}
\frac{\|\nabla_1^d w_q\|_{\ell^2(\V^d_1)}^2}{\|w_q\|_{\ell^q(\V^d_1)}^q}
w_q^{q-1}\quad \text{in }\V^d_1,
\]
and
\[
\lambda_q=\left(\frac{2q-d(q-2)}{d(q-2)}\right)\|\nabla_1^d w_q\|_{\ell^2(\V^d_1)}^2.
\]
Moreover, the connectivity of $\Z^d_1$ leads to $w_q(\vv)>0$ for all $\vv \in \V^d_1$. Let 
\[u_q:=\left(\frac{2q}{d(q-2)}
\frac{\|\nabla_1^d w_q\|_{\ell^2(\V^d_1)}^2}{\|w_q\|_{\ell^q(\V^d_1)}^q}\right)^{\frac{1}{q-2}}w_q.\]
Then we conclude that $u_q>0$ satisfies 
\[
Q_{q,\V^d_1}(u_q)=K_{q,\V^d_1}
\]
and
\[
-\Delta_1^d u_q+ \lambda_q u_q=u_q^{q-1}
\quad\text{in }\V_1^d.
\]
\end{proof}
\begin{remark}
    The arguments in Sections \ref{sec5} and \ref{seccle} can be adapted directly to 
    $$
    \tilde{K}_{q,\G^d_1}:=\sup_{u\in H^1(\G^d_1)\backslash\{0\}}\frac{\sum_{\vv\in\V^d_1}\abs{\hat{u}(v)}^q}{\norm{u'}^{(\frac{q}{2}-1)d}_{L^2(\G^d_1)}\left(\sum_{\vv\in\V^d_1}\abs{\hat{u}(v)}^2\right)^{\frac{d}{2}+(2-d)\frac{q}{4}}}.
    $$
    Indeed, by Lemma \ref{lemh1gv}, if $u$ is a maximizer of $\tilde{K}_{q,\G^d_1}$, then $\hat{u}$ is a maximizer of $\hat{K}_{q,\V^d_1}$ and $\tilde{\hat{u}}$ is a maximizer of $\tilde{K}_{q,\G^d_1}$, and the same conclusion holds for maximization sequences.
\end{remark}

\section{Further remarks and open problems}\label{sec6}
In this section, we pose several interesting open problems concerning (NLS) equation on lattice graphs with Sobolev critical exponent. We present a possibility to improve Theorem \ref{th4} for the Sobolev critical case and we propose some interesting questions concerning singular limit of solutions to the (NLS) on lattice graphs with Sobolev critical exponent.
\subsection{A possibility to improve Theorem \ref{th4}}

The following lemma is the analogue of Lemma \ref{lemactionbdd} for the case $d\geq 3$ and $p\geq 2^*$.
\begin{lemma}\label{lem43}
    For every $d \geq 3$, $p \geq 2^*$ and $\omega > 0$,  there exist $\bar{\varepsilon} > 0$ and $C' > 0$, depending only on $p$, $\omega$ and $d$, such that, for every $\varepsilon \in (0, \bar{\varepsilon})$, if $f \in \hat{\mathcal{N}}_{\omega, \V_\varepsilon^d}$ is a ground state of $\hat{J}_{\omega, \V_\varepsilon^d}$, then
\begin{equation}\label{eqlem43}
\begin{aligned}
    \varepsilon^{d-1}\|\nabla^d_\varepsilon f\|_{\ell^2(\V_\varepsilon^d)}^2,\,\, \varepsilon^{d}\| f\|_{\ell^2(\V_\varepsilon^d)}^2,\,\,  \varepsilon^{d}\|f\|_{\ell^p(\V_\varepsilon^d)}^p \leq C',
\end{aligned}
 \end{equation}
 and
 \begin{equation}\label{eqlem431}
\begin{aligned}
    \varepsilon^\frac{p+2}{p-2}\|\nabla^d_\varepsilon f\|_{\ell^2(\V_\varepsilon^d)}^2,\, \,\varepsilon^\frac{4}{p-2}\| f\|_{\ell^2(\V_\varepsilon^d)}^2,\, \, \varepsilon^\frac{2p}{p-2}\|f\|_{\ell^p(\V_\varepsilon^d)}^p  \geq \frac{1}{C'}.
\end{aligned}
 \end{equation}
\end{lemma}
\begin{proof}
    Let $d \geq 3$, $p \geq 2^*$ and $\omega > 0$.  Let $f \in \hat{\mathcal{N}}_{\omega, \V_\varepsilon^d}$ be a ground state of $\hat{J}_{\omega, \V_\varepsilon^d}$. By Lemma \ref{leminfa} and by repeating the arguments in the proof of Lemma \ref{lemaction}, we know that, for every $\varepsilon>0$ small enough,
$$
\varepsilon^{d-1}\hat{\mathcal{J}}_{\V_\varepsilon^d}(\omega)=\varepsilon^{d-1}\tilde{\mathcal{J}}_{\mathcal{G}_\varepsilon^d}(\omega) \leq C.
$$
Then, by the definition of $\hat{\mathcal{N}}_{\omega,\Vd}$ given in \eqref{eqdefnv}, we have that for $\varepsilon>0$ small enough
    \begin{equation*}
            \varepsilon^{d-1}\|\nabla^d_\varepsilon f\|_{\ell^2(\V_\varepsilon^d)}^2,\, \varepsilon^{d}\| f\|_{\ell^2(\V_\varepsilon^d)}^2,\,  \varepsilon^{d}\|f\|_{\ell^p(\V_\varepsilon^d)}^p \leq C.
    \end{equation*}
    
By Lemma \ref{lemh1gv}, the Gagliardo--Nirenberg type inequalities \eqref{eqsobo} and \eqref{eqlinfty}, and the definition of $\hat{\mathcal{N}}_{\omega,\Vd}$  given in \eqref{eqdefnv}, one has
$$
\begin{aligned}
   \|\nabla^d_\varepsilon f\|_{\ell^2(\V_\varepsilon^d)}^2\leq \|\nabla^d_\varepsilon f\|_{\ell^2(\V_\varepsilon^d)}^2+\varepsilon\omega\|f\|_{\ell^2(\V_\varepsilon^d)}^2&=\varepsilon\|f\|_{\ell^p(\V_\varepsilon^d)}^p\\
    &\leq C\varepsilon\|f\|_{\ell^\infty(\V_\varepsilon^d)}^{p-2^*}\|f\|_{\ell^{2^*}(\V_\varepsilon^d)}^{2^*}\\
    &\leq C\varepsilon(\varepsilon^{1/2}\|\nabla^d_\varepsilon f\|_{\ell^2(\V_\varepsilon^d)})^{2^*}(\varepsilon^{1/2}\|\nabla^d_\varepsilon f\|_{\ell^2(\V_\varepsilon^d)})^{p-2^*}\\
    &\leq C\varepsilon^{1+p/2}\|\nabla^d_\varepsilon f\|_{\ell^2(\V_\varepsilon^d)}^p,
\end{aligned}
$$
that is,
$$
\varepsilon^\frac{p+2}{p-2}\|\nabla^d_\varepsilon f\|_{\ell^2(\V_\varepsilon^d)}^2\geq C.
$$
Thus, by repeating the arguments in the proof of \cite[Lemma 2.1]{HJ}, we have
$$
\|f\|_{\ell^2(\V_\varepsilon^d)}^2\geq C \varepsilon\|\nabla^d_\varepsilon f\|_{\ell^2(\V_\varepsilon^d)}^2 \geq C\varepsilon^\frac{4}{2-p}.
$$
Moreover, by the definition of $\hat{\mathcal{N}}_{\omega,\Vd}$  given in \eqref{eqdefnv}, we obtain that 
$$
\|f\|_{\ell^p(\V_\varepsilon^d)}^p\geq \frac{1}{\varepsilon}\|\nabla^d_\varepsilon f\|_{\ell^2(\V_\varepsilon^d)}^2 \geq C\varepsilon^\frac{2p}{2-p}.
$$
\end{proof}
By Lemma \ref{lem43}, we know that $\varepsilon^\frac{4}{p-2}\|f_\varepsilon\|_{\ell^2(\V_\varepsilon^d)}^2\geq \frac{1}{C'}$, where $f \in \hat{\mathcal{N}}_{1, \V_\varepsilon^d}$ is a ground state of $\hat{J}_{1, \V_\varepsilon^d}$.
If,  for $d\geq 3$ and $p =2^*$, we can prove that $\varepsilon^\frac{4}{2^*-2}\|f_\varepsilon\|_{\ell^2(\V_\varepsilon^d)}^2\to +\infty$ as $\varepsilon\to 0$,  then, by repeating the arguments in the proof of Theorem \ref{th4}, we can improve the range of $p$ in Theorem \ref{th4} from $(2+\frac{4}{d},2^*)$ to $(2+\frac{4}{d},2^*]$.

\subsection{Singular limit problems in the Sobolev critical case}
Let $d \geq 3$. Even though it is impossible to investigate the singular limit of action ground states for critical Lane--Emden equations \eqref{eqcle} on lattice graphs  as stated in Remark \ref{reclesl}, we can introduce singular limit problems in the Sobolev critical case as follows. We introduce the action functional involving a critical Sobolev exponent $J_{\R^d}: D^{1,2}(\R^d) \to \R$
$$
J_{\R^d}(u):=\frac{1}{2}\norm{\nabla u }^2_{L^2(\R^d)}-\frac{1}{2^*}\norm{u}^{2^*}_{L^{2^*}(\R^d)}
$$
and the associated Nehari manifold
$$
\mathcal{N}_{\R^d}:=\left\{u \in D^{1,2}(\R^d):\norm{\nabla u }^2_{L^2(\R^d)}= \norm{u}^{2^*}_{L^{2^*}(\R^d)}\right\}.
$$
Denote by 
$$
\mathcal{J}_{\R^d}:=\inf_{u \in \mathcal{N}_{\R^d}}J_{\R^d}(u)
$$
the corresponding ground state problem.

Let $d \geq 3$ and $\alpha>0$ be fixed. For lattice graphs, we introduce the following functional involving a critical Sobolev exponent $\hat{J}_{\Vd}: H^1(\Vd) \to \R$
$$
\hat{J}_{\Vd}(u):=\frac{1}{2}\norm{\nabla^d_\varepsilon u }^2_{L^2(\Vd)}+\frac{\varepsilon^{\alpha+1}}{2}\norm{ u }^2_{L^2(\Vd)}-\frac{\varepsilon}{2^*}\norm{u}^{2^*}_{L^{2^*}(\Vd)}
$$
and the associated Nehari manifold
$$
\hat{\mathcal{N}}_{\Vd}:=\left\{u \in H^1(\Vd):\norm{\nabla^d_\varepsilon u }^2_{L^2(\Vd)}+\varepsilon^{\alpha+1}\norm{ u }^2_{L^2(\Vd)}=\varepsilon \norm{u}^{2^*}_{L^{2^*}(\Vd)}\right\}.
$$
Denote by 
$$
\hat{\mathcal{J}}_{\Vd}:=\inf_{u \in \hat{\mathcal{N}}_{\Vd}}\hat{J}_{\Vd}(u)
$$
the corresponding ground state problem.

By cutting off the ground state of $J_{\R^d}$, it is easy to see that 
$$
\varepsilon^{d-1}\hat{\mathcal{J}}_{\Vd} \leq \mathcal{J}_{\R^d}+o(1)\quad \text{as }\varepsilon \to 0,
$$
and, for the ground state $f_\varepsilon$  of $\hat{J}_{\Vd}$, for $\varepsilon>0$ sufficiently small,
$$
\frac{1}{C'}\leq \varepsilon^{d-1}\norm{\nabla^d_\varepsilon f_\varepsilon }^2_{L^2(\Vd)},\,\,  \varepsilon^d\norm{f_\varepsilon}^{2^*}_{L^{2^*}(\Vd)}\leq C'\quad \text{and }\varepsilon^{d+\alpha}\norm{f_\varepsilon}^{2}_{L^{2}(\Vd)}\leq C'.
$$
However,  since Lemma \ref{lemlL} does not hold for $p=2^*$, for $t_\varepsilon>0$ such that $v_\varepsilon:=t_\varepsilon\hat{\mathcal{A}}f_\varepsilon \in \mathcal{N}_{\R^d}$,  it appears to be quite challenging to show
$$
J_{\R^d}(v_\varepsilon) \leq \varepsilon^{d-1}\tilde{\mathcal{J}}_{\mathcal{G}_\varepsilon^d}(\omega)+o(1)\quad \text{as }\varepsilon \to 0.
$$
A similar problem could be stated as follows:\\
Let $d \geq 3$. For $p >2^*$, we introduce the following functional $\hat{J}_{p,\Vd}: D^{1,2}(\Vd) \to \R$
$$
\hat{J}_{p,\Vd}(u):=\frac{1}{2}\norm{\nabla^d_\varepsilon u }^2_{L^2(\Vd)}-\frac{\varepsilon}{p}\norm{u}^{p}_{L^{p}(\Vd)}
$$
and the associated Nehari manifold
$$
\hat{\mathcal{N}}_{p,\Vd}:=\left\{u \in D^{1,2}(\Vd):\norm{\nabla^d_\varepsilon u }^2_{L^2(\Vd)}=\varepsilon \norm{u}^{p}_{L^{p}(\Vd)}\right\}.
$$
Denote by 
$$
\hat{\mathcal{J}}_{p,\Vd}:=\inf_{u \in \hat{\mathcal{N}}_{p,\Vd}}\hat{J}_{p,\Vd}(u)
$$
the corresponding ground state problem. It is an interesting problem to determine whether
$$
\varepsilon^{d-1}\hat{\mathcal{J}}_{2^*+\varepsilon,\Vd} \to \mathcal{J}_{\R^d}, \quad \text{as }\varepsilon \to 0^{+}.
$$
\appendix
\section{Critical Lane--Emden equations on periodic metric grids}\label{app}
For $d \geq 3$, let $D^{1,2}(\Gd)$ denote the completion of $H^1(\Gd)$ with respect to the norm 
$\norm{u'}_{L^2(\Gd)}$. By the Gagliardo--Nirenberg type inequality \eqref{eqlinfty}, we know that $D^{1,2}(\Gd)\hookrightarrow C(\Gd)$. Therefore, by  \cite[Lemma 2.4]{DHJ} or the Gagliardo--Nirenberg type inequality for $p=+\infty$ (see \cite[Section 2]{AST2}), we have
\begin{equation}\label{eqappd12vanish}
    \lim_{\abs{x}\to +\infty}\abs{u(x)}\to 0, \quad \forall u \in D^{1,2}(\Gd),
\end{equation}
where, for $x \in \Gd$, $\abs{x}:=\operatorname{dist}(x,0)$.

For $d \geq 3$ and $q\geq2^*$, define
$$
S_{q,\G^d_1}:=\inf_{u \in D^{1,2}(\Gd)\setminus \{0\}}\frac{\norm{u'}_{L^2(\G^d_1)}^2}{\norm{u}_{L^q(\G^d_1)}^2}.
$$
Then, by repeating the arguments in the proof of \cite[Theorem 1]{hua2021existence} and using the Gagliardo--Nirenberg inequality on $\G_1^d$ (see, e.g., \cite[Lemma 3.1]{Do2}), we know that, for every $d \geq 3$ and $q>2^*$, there exists a positive minimizer for $S_{q,\G^d_1}$.

Set 
\[
Q_{2^*,\G_1^d}(u):=\frac{\norm{u}^{2^*}_{L^{2^*}(\G^d_1)}}{\norm{u'}^{2^*}_{L^{2}(\G^d_1)}},\quad u \in D^{1,2}(\G^d_1)\setminus\{0\}.
\]
We now extend Theorems \ref{th6} and \ref{th7} to the metric grid $\G^d_1$. 
\begin{theorem}\label{thmg1}
If \(d\geq 3\), then 
\begin{equation*}
K_{2^*,\G_1^d}:=S_{2^*,\G^d_1}^{- {2^*}/{2}}>d  K_{2^*,\mathbb R^d}:=S_{\R^d}^{-{2^*}/{2}},
\end{equation*}
and the best constant \(K_{2^*,\G_1^d}\) is attained by a positive function \(u\in D^{1,2}(\G_1^d)\), which satisfies 
\[
Q_{2^*,\G_1^d}(u)=K_{2^*,\G_1^d}
\]
and solves
\[
\begin{cases}
-u''=u^{2^*-1}
& \text{on every edge of }\mathcal{G}^d_1,\\
\displaystyle\sum_{\e\succ \mathrm{v}}\frac{du}{dx_\e}(\mathrm{v})=0
& \text{for every vertex }\mathrm{v}\text{ of }\mathcal{G}^d_1,
\end{cases}
\]
Here, $\e \succ \vv$ means that the edge $\e$ is incident at the vertex $\vv$, whereas
$\frac{du}{dx_{\e}}(\vv)$ denotes the outward derivative of $u$ at $\vv$ along $\e$.
\end{theorem}

\begin{theorem}\label{thmg2}
If \(d\geq 3\), then there exists $\delta=\delta(d)>0$ such that, for every $q\in (2^*-\delta,2^*)$,
\[
K_{q,\G_1^d}>d^{\frac{(d-2)(q-2)}{4}}K_{q,\mathbb R^d},
\]
where $Q_{q,\G_1^d}$ and $K_{q,\G_1^d}$ are defined as in \cite[Section~2]{Do2}, and \(K_{q,\G_1^d}\) is attained by a positive function \(u\in H^{1}(\G_1^d)\), which satisfies 
\[
Q_{q,\G_1^d}(u)=K_{q,\G_1^d}
\]
and, for some $\lambda>0$, solves
\[
\begin{cases}
-u''+\lambda u =u^{q-1}
& \text{on every edge of }\mathcal{G}^d_1,\\
\displaystyle\sum_{\e\succ \mathrm{v}}\frac{du}{dx_\e}(\mathrm{v})=0
& \text{for every vertex }\mathrm{v}\text{ of }\mathcal{G}^d_1,
\end{cases}
\]
\end{theorem}

Our estimate
\[
K_{2^*,\mathcal G_1^d}>dK_{2^*,\mathbb R^d},
\quad d\geq 3,
\]
shows that the equality established in \cite[Theorem~2.5]{Do2} for
$q\in[2+\frac4d,2^*)$ does not extend to the Sobolev critical case
$q=2^*$. This failure is related to a gap in the proof of
\cite[Lemma~7.2]{Do2}: in deriving \cite[Eq.~(72)]{Do2}, one obtains
\[
\|v'\|_{L^2(\mathcal G_1^d)}^{2-\frac d2(q-2)}
\geq
K_{q,\mathcal G_1^d}
\|v\|_{L^2(\mathcal G_1^d)}^{d+\frac{2-d}{2}q}.
\]
Since the exponent $2-\frac d2(q-2)$ is negative for
$q>2+\frac4d$, taking its reciprocal power must reverse the direction
of the inequality. Hence, \cite[Eq.~(72)]{Do2} should have the opposite
inequality sign, and the claimed contradiction does not follow.

\begin{proof}[Proofs of Theorems \ref{thmg1} and \ref{thmg2}]
    For simplicity, we write $S_{\G^d_1}:=S_{2^*,\G^d_1}$. We first assume that
\begin{equation}\label{eq:strict-metric-sobolev-comparison}
S_{\mathcal G_1^d}
<
d^{-\frac{d-2}{d}}S_{\mathbb R^d}
\quad\text{for every }d\geq 3.
\end{equation}
Under this assumption, we prove Theorems \ref{thmg1} and \ref{thmg2}.

Let $\{u_n\} \subset H^1(\G^d_1)$ be a  minimizing sequence for $S_{\G^d_1}$. We may assume that  $u_n$ is nonnegative and $\|u_n\|_{L^{2^*}(\G^d_1)}=1$ for all $n \in \mathbb{N}^{+}$.  Moreover, $u_n$ satisfies
     \begin{equation}\label{eqapppf2*1}
         \norm{u_n'}_{L^2(\G^d_1)}^2\to S_{\G^d_1}>0 \quad \text{as }n \to +\infty.
     \end{equation}
    We now prove that $\norm{u_n}_{L^{\infty}(\G^d_1)}\not \to 0$ as $n \to +\infty$. We argue by contradiction and suppose that $\norm{u_n}_{L^{\infty}(\G^d_1)}  \to 0$ as $n \to +\infty$.
       
       By  \cite[Lemma 4.4]{Do2}, we obtain
\begin{equation}\label{eqappgradient}
\|\nabla \mathcal{A}u_n\|_{L^2(\R^d)}^2\leq \norm{u_n'}_{L^2(\G^d_1)}^2,
\end{equation}
where $\mathcal{A}$ is the extension operator from $H^1 (\Gd)$ to $H^1(\R^d)$ defined as in \cite[Section 2]{Do2}. Then, by Lemma \ref{lemlp}, we have
\[
\left|\|u_n\|_{L^{2^*}(\G^d_1)}^{2^*}-d\sum_{\vv \in \V^d_1}\abs{\hat{u}_n}^{2^*}\right|\leq C\|u_n\|_{L^{2\cdot2^*-2}(\G^d_1)}^{2^*-1} \|u_n'\|_{L^2(\G^d_1)},
\]
\[
\left|\|\tilde{\hat{u}}_n\|_{L^{2^*}(\G^d_1)}^{2^*}-d\sum_{\vv \in \V^d_1}\abs{\hat{u}_n}^{2^*}\right|\leq C\|\tilde{\hat{u}}_n\|_{L^{2\cdot2^*-2}(\G^d_1)}^{2^*-1} \|\tilde{u}_n'\|_{L^2(\G^d_1)}.
\]    
Then, by   \cite[Eq. (26)]{Do2} and the Sobolev inequality (see \cite[Eq. (16)]{Do2}),
\begin{equation}\label{eqappthun}
\begin{aligned}
    \left|\|\tilde{\hat{u}}_n\|_{L^{2^*}(\G^d_1)}^{2^*}-\|u_n\|_{L^{2^*}(\G^d_1)}^{2^*}\right|
    &\leq C\|u_n\|_{L^{2\cdot2^*-2}(\G^d_1)}^{2^*-1} \|u_n'\|_{L^2(\G^d_1)}+ C\|\tilde{\hat{u}}_n\|_{L^{2\cdot2^*-2}(\G^d_1)}^{2^*-1} \|\tilde{\hat{u}}_n'\|_{L^2(\G^d_1)}\\
    &\leq C\norm{u_n}^{\frac{2^*}{2}-1}_{L^{\infty}(\G^d_1)}\left( \norm{u_n}^{\frac{2^*}{2}}_{L^{2^*}(\G^d_1)}\|u_n'\|_{L^{2}(\G^d_1)}+\norm{\tilde{\hat{u}}_n}^{\frac{2^*}{2}}_{L^{2^*}(\G^d_1)}\|\tilde{\hat{u}}_n'\|_{L^2(\G^d_1)}\right)\\
    &\leq C\norm{u_n}^{\frac{2^*}{2}-1}_{L^{\infty}(\G^d_1)}\left( \|u_n'\|_{L^{2}(\G^d_1)}^{\frac{2^*}{2}+1}+\|\tilde{\hat{u}}_n'\|_{L^2(\G^d_1)}^{\frac{2^*}{2}+1}\right)\\
    &\leq C\norm{u_n}^{\frac{2^*}{2}-1}_{L^{\infty}(\G^d_1)}\|u_n'\|_{L^2(\G^d_1)}^{\frac{2^*}{2}+1}.
\end{aligned}
\end{equation}
By repeating the arguments in the proof of \cite[Eq. (59)]{Do2} and using  \eqref{eqappgradient}, \cite[Eq. (26)]{Do2} and \eqref{eqsobo}, we have
    \[
    \begin{aligned}
    \left|{\|\tilde{\hat{u}}_n\|_{L^{2^*}(\G^d_1)}^{2^*}}-d\norm{\mathcal{A}u_n}^{2^*}_{L^{2^*}(\R^d)}\right|
    &\leq C\norm{\mathcal{A}u_n}^{2^*-1}_{L^{2\cdot2^*-2}(\R^d)}\norm{\nabla \mathcal{A}u_n}_{L^2(\R^d)}\\
    &\leq C \| \hat{u}_n\|_{\ell^{2 \cdot 2^*-2}(\V^d_1)}^{2^*-1}\norm{\tilde{\hat{u}}_n'}_{L^2(\G^d_1)}\\
       &\leq C \| \hat{u}_n\|_{\ell^{2 \cdot 2^*-2}(\V^d_1)}^{2^*-1}\norm{u_n'}_{L^2(\G^d_1)}\\
       &\leq C
       \norm{u_n}^{\frac{2^*}{2}-1}_{L^{\infty}(\G^d_1)}  \| u_n\|_{\ell^{2^*}(\V^d_1)}^{\frac{2^*}{2}}\|u_n'\|_{L^{2}(\G^d_1)}\\
       &\leq C\norm{u_n}^{\frac{2^*}{2}-1}_{L^{\infty}(\G^d_1)} \|u_n'\|_{L^{2}(\G^d_1)}^{\frac{2^*}{2}+1}.
    \end{aligned}
    \]
Therefore, by \eqref{eqapppf2*1}, \eqref{eqappgradient}, \eqref{eqappthun} and the assumption that $\norm{u_n}_{L^{\infty}(\G^d_1)}  \to 0$ as $n \to +\infty$, we obtain
\begin{equation}\label{eqappnormaun2star}
   \norm{\mathcal{A}u_n}^{2^*}_{L^{2^*}(\R^d)}\to\frac{1}{d}\quad \text{as }n\to +\infty. 
\end{equation}
Hence, by \eqref{eqapppf2*1}, \eqref{eqappgradient}, and the Sobolev inequality on $\R^d$, we obtain
\[
S_{\G^d_1}
=\norm{u_n'}^{2}_{L^2(\G^d_1)}+o(1)\geq \|\nabla \mathcal{A}u_n\|_{L^2(\R^d)}^{2}+o(1)\geq S_{\R^d}\|\mathcal{A}u_n\|_{L^{2^*}(\R^d)}^{2}+o(1).
\]
Thus, it follows from \eqref{eqappnormaun2star} that
\[
S_{\G^d_1}\geq d^{-2/2^*}S_{\R^d}
=d^{-\frac{d-2}{d}}S_{\R^d},
\]
which contradicts \eqref{eq:strict-metric-sobolev-comparison}. We have proved 
\[
\limsup_{n \to +\infty}\norm{u_n}_{L^\infty(\G^d_1)}>0.
\]
By  \eqref{eqappd12vanish}, up to a subsequence,  we can choose $\vv_n\in\V^d_1$ and $\eta>0$ such that
    $\max_{\operatorname{dist}(x,\vv_n)\leq 1}u_n(x)\geq\eta$, and define $w_n(x):=u_n(x+\vv_n)$. We have that 
    \[
  \|w_n'\|_{L^2(\G^d_1)}=\|u'_n\|_{L^2(\G^d_1)},\quad  \|w_n\|_{L^{2^*}(\G^d_1)}=\|u_n\|_{L^{2^*}(\G^d_1)}=1, \quad  \max_{\operatorname{dist}(x,0)\leq1}w_n(x)\geq\eta..
    \]
   For every $R>0$, H\"older's inequality and
$\|w_n\|_{L^{2^*}(\G_1^d)}=1$ imply that $\{w_n\}$ is bounded
in $H^1(B_R(0))$. Hence, by the compact embedding
\[
H^1(B_R(0))\hookrightarrow C(\overline{B_R(0)}),
\]
and a diagonal argument, after passing to a further subsequence,
there exists $w\in D^{1,2}(\G_1^d)$ such that
\[
w_n\rightharpoonup w
\quad\text{in }D^{1,2}(\G_1^d)
\]
and
\[
w_n\to w
\quad\text{locally uniformly on }\G_1^d.
\]
In particular, $w\not\equiv0$. By repeating the argument in the proof of Lemma \ref{lem:strict_discrete_sobolev_constant}, we complete the proof of Theorem \ref{thmg1}. 
    
    Noting that the argument in the proof of \cite[Proposition 2.6]{Do2} holds for every $d \geq 3$ and every $q \in (2,2^*)$, by \eqref{eq:strict-metric-sobolev-comparison} and repeating the argument in the proof of Theorem \ref{th7}, we complete the proof of Theorem \ref{thmg2}.

We then prove \eqref{eq:strict-metric-sobolev-comparison}. We first consider $d=3$. In this case, $2^*=6$. Define $f$ on
$\V_1^3$ by
\begin{equation}\label{eq:test-function-d3}
f(\vv)
:=
\frac{1}{\sqrt{1+6|\vv|^2}},
\quad
\vv\in\V_1^3,
\end{equation}
where $|\cdot|$ denotes the Euclidean norm, and let
$u:=\widetilde f$ be the affine extension of $f$ to
$\mathcal G_1^3$.

Since
\[
f(\vv)=O(|\vv|^{-1})
\]
and
\[
|f(\vv+e_i)-f(\vv)|
=
O(|\vv|^{-2}),
\quad i=1,2,3,
\]
we have
\[
u\in L^6(\mathcal G_1^3),
\qquad
u'\in L^2(\mathcal G_1^3).
\]

Although $u\notin L^2(\mathcal G_1^3)$, it can be approximated in
$L^6(\mathcal G_1^3)$ and in the Dirichlet norm by compactly
supported functions. Indeed, let $\eta_R$ be a Lipschitz cutoff
satisfying
\[
0\leq\eta_R\leq1,
\qquad
\eta_R=1\ \text{on }B_R,
\qquad
\eta_R=0\ \text{on }\mathcal G_1^3\setminus B_{2R},
\qquad
|\eta_R'|\leq\frac{C}{R},
\]
where $B_R$ denotes the ball centered at the origin, and set
\[
u_R:=\eta_Ru.
\]
Then $u_R\in H^1(\mathcal G_1^3)$ and
\[
\|u_R-u\|_{L^6(\mathcal G_1^3)}
+
\|u_R'-u'\|_{L^2(\mathcal G_1^3)}
\longrightarrow0.
\]
Indeed,
\[
\|u_R'-u'\|_2^2
\leq
2\int_{\mathcal G_1^3\setminus B_R}|u'|^2
+
\frac{C}{R^2}
\int_{B_{2R}\setminus B_R}|u|^2.
\]
The first term tends to zero since $u'\in L^2(\mathcal G_1^3)$.
Moreover, the number of edges in a shell of radius $k$ is of order
$k^2$, whereas $u^2=O(k^{-2})$, and hence
\[
\int_{B_{2R}\setminus B_R}|u|^2=O(R).
\]
It is therefore enough to estimate the homogeneous Sobolev quotient
of $u$.

By symmetry,
\begin{equation}\label{eqappdiri}
\|u'\|_{L^2(\mathcal G_1^3)}^2
=
3\sum_{(m,n)\in\mathbb Z^2}
\sum_{k\in\mathbb Z}
\left[
\frac{1}
{\sqrt{1+6((k+1)^2+m^2+n^2)}}
-
\frac{1}
{\sqrt{1+6(k^2+m^2+n^2)}}
\right]^2.
\end{equation}
For $(m,n)\in\mathbb Z^2$, set
\[
f_{m,n}(t)
:=
\frac{1}{\sqrt{1+6(t^2+m^2+n^2)}}.
\]
By the Cauchy--Schwarz inequality,
\[
|f_{m,n}(k+1)-f_{m,n}(k)|^2
\leq
\int_k^{k+1}|f_{m,n}'(t)|^2\,dt.
\]
Consequently,
\[
\sum_{k\in\mathbb Z}
|f_{m,n}(k+1)-f_{m,n}(k)|^2
\leq
\int_{\mathbb R}|f_{m,n}'(t)|^2\,dt
=
\frac{\pi\sqrt6}
{8\bigl(1+6(m^2+n^2)\bigr)^{3/2}}.
\]

For the central line $(m,n)=(0,0)$, we use a sharper estimate. Set
\begin{equation}\label{eqappe0def}  
a_k:=\frac{1}{\sqrt{1+6k^2}},
\qquad
e_0:=2\sum_{k=0}^{\infty}(a_k-a_{k+1})^2.
\end{equation}
For $k\geq3$,
\[
(a_k-a_{k+1})^2
=\left(\int_k^{k+1}f_{0,0}'(t)\,dt\right)^2\leq \int_k^{k+1}\abs{f_{0,0}'(t)}^2\,dt=
\int_k^{k+1}
\frac{36t^2}{(1+6t^2)^3}\,dt
<
\int_k^{k+1}\frac{dt}{6t^4},
\]
and hence
\[
2\sum_{k=3}^{\infty}(a_k-a_{k+1})^2
<
2\int_3^\infty\frac{dt}{6t^4}
=
\frac1{243}.
\]
Using
\[
\frac{17}{45}<\frac1{\sqrt7}<\frac8{21},
\qquad
\frac2{15}<\frac1{\sqrt{55}},
\]
together with
\[
a_0=1,
\qquad
a_1=\frac1{\sqrt7},
\qquad
a_2=\frac15,
\qquad
a_3=\frac1{\sqrt{55}},
\]
we obtain
\begin{equation}\label{eqappe0}
\begin{aligned}
e_0
<
2\left[
\left(\frac{28}{45}\right)^2
+
\left(\frac{19}{105}\right)^2
+
\left(\frac1{15}\right)^2
\right]
+\frac1{243}=
\frac{253861}{297675}
<
\frac{171}{200}.
\end{aligned}
\end{equation}

We next estimate
\begin{equation}\label{eqappsigmadef}
\Sigma
:=
\sum_{(m,n)\in\mathbb Z^2}
\frac1{\bigl(1+6(m^2+n^2)\bigr)^{3/2}}.
\end{equation}
For $k\geq1$, set
\[
H_k
:=
\sum_{\max\{|m|,|n|\}=k}
\frac1{\bigl(1+6(m^2+n^2)\bigr)^{3/2}}.
\]
Then
\begin{equation}\label{eqappsigma}
\Sigma-1=\sum_{k=1}^{\infty}H_k
\end{equation}
and
\[
H_k
=
\frac4{(1+6k^2)^{3/2}}
+
8\sum_{j=1}^{k-1}
\frac1{(1+6(k^2+j^2))^{3/2}}
+
\frac4{(1+12k^2)^{3/2}}.
\]

For the first four shells, direct rational estimates give
\[
\begin{aligned}
H_1
&=
\frac4{7^{3/2}}+\frac4{13^{3/2}}<
\frac{4\cdot53995+4\cdot21335}{10^6}
<
\frac{3014}{10^4},
\\[1mm]
H_2
&=
\frac4{25^{3/2}}
+\frac8{31^{3/2}}
+\frac4{49^{3/2}}\\
&<
\frac{4\cdot8000+8\cdot5794+4\cdot2916}{10^6}
<
\frac{901}{10^4},
\\[1mm]
H_3
&=
\frac4{55^{3/2}}
+\frac8{61^{3/2}}
+\frac8{79^{3/2}}
+\frac4{109^{3/2}}\\
&<
\frac{
4\cdot2452+8\cdot2099+8\cdot1425+4\cdot879
}{10^6}
<
\frac{416}{10^4},
\\[1mm]
H_4
&=
\frac4{97^{3/2}}
+\frac8{103^{3/2}}
+\frac8{121^{3/2}}
+\frac8{151^{3/2}}
+\frac4{193^{3/2}}\\
&<
\frac{
4\cdot1047+8\cdot957+8\cdot752
+8\cdot539+4\cdot373
}{10^6}
<
\frac{237}{10^4}.
\end{aligned}
\]
In particular,
\[
\sum_{k=1}^4H_k<\frac{571}{1250}.
\]

For $k\geq5$, the map
\[
t\longmapsto
\frac1{\bigl(6(k^2+t^2)\bigr)^{3/2}}
\]
is decreasing on $[0,+\infty)$. Therefore,
\[
\begin{aligned}
H_k
&<
\frac4{(6k^2)^{3/2}}
+
8\int_0^k
\frac{dt}{\bigl(6(k^2+t^2)\bigr)^{3/2}}
+
\frac4{(12k^2)^{3/2}}\\
&=
\frac{2}{3\sqrt3}\frac1{k^2}
+
\left(
\frac{2}{3\sqrt6}
+
\frac1{6\sqrt3}
\right)\frac1{k^3}\\
&<
\frac25\frac1{k^2}
+
\frac{17}{45}\frac1{k^3},
\end{aligned}
\]
where we used
\[
\sqrt3>\frac53,
\qquad
\sqrt6>\frac{12}{5}.
\]
Moreover,
\[
\sum_{k=5}^{\infty}\frac1{k^2}
\leq
\frac1{25}+\int_5^\infty\frac{dt}{t^2}
=
\frac6{25}
\]
and
\[
\sum_{k=5}^{\infty}\frac1{k^3}
\leq
\frac1{125}+\int_5^\infty\frac{dt}{t^3}
=
\frac7{250}.
\]
It follows that
\[
\sum_{k=5}^{\infty}H_k
<
\frac25\frac6{25}
+
\frac{17}{45}\frac7{250}
=
\frac{1199}{11250}.
\]
Consequently, by \eqref{eqappsigma},
\[
\Sigma-1
<
\frac{571}{1250}
+
\frac{1199}{11250}
=
\frac{3169}{5625}
<
\frac{141}{250}.
\]

Since
\[
\frac{\pi\sqrt6}{8}
<
\frac18\frac{22}{7}\frac{49}{20}
=
\frac{539}{560},
\]
by \eqref{eqappdiri}, \eqref{eqappe0def}, \eqref{eqappe0} and \eqref{eqappsigmadef}, we obtain
\begin{equation}\label{eq:energy-upper-d3}
\|u'\|_{L^2(\mathcal G_1^3)}^2
\leq
3\left(
e_0+\frac{\pi\sqrt6}{8}(\Sigma-1)
\right)<
3\left(
\frac{171}{200}
+
\frac{539}{560}\frac{141}{250}
\right)
=
\frac{83871}{20000}.
\end{equation}

We now estimate the $L^6$ norm from below. For $a,b>0$, set
\[
I(\alpha,\beta)
:=
\int_0^1\bigl((1-t)\alpha+t\beta\bigr)^6\,dt
=
{\frac{\alpha^7-\beta^7}{7(\alpha-\beta)}}.
\]
We retain only the six edges joining the origin to vertices
satisfying $|\vv|^2=1$, the twenty-four edges joining vertices
satisfying $|\vv|^2=1$ to vertices satisfying $|\vv|^2=2$, and
the six edges joining vertices satisfying $|\vv|^2=1$ to vertices
satisfying $|\vv|^2=4$ along the coordinate axes.

Since the corresponding endpoint values are
\[
1,
\qquad
\frac1{\sqrt7},
\qquad
\frac1{\sqrt{13}},
\qquad
\frac15,
\]
the monotonicity of $I$ in each variable gives
\begin{equation}\label{eq:mass-lower-d3}
\begin{aligned}
\|u\|_{L^6(\mathcal G_1^3)}^6
&\geq
6I\left(1,\frac1{\sqrt7}\right)
+
24I\left(\frac1{\sqrt7},\frac1{\sqrt{13}}\right)
+
6I\left(\frac1{\sqrt7},\frac15\right)\\
&>
6I\left(1,\frac38\right)
+
24I\left(\frac38,\frac{11}{40}\right)
+
6I\left(\frac38,\frac15\right)\\
&=
\frac{10083370077}{7168000000}
>
\frac75.
\end{aligned}
\end{equation}

It follows from \eqref{eq:energy-upper-d3} and
\eqref{eq:mass-lower-d3} that
\[
\frac{\|u'\|_{L^2(\mathcal G_1^3)}^2}
{\|u\|_{L^6(\mathcal G_1^3)}^2}
<
\frac{83871}{20000}
\left(\frac57\right)^{1/3}
<
\frac{15}{4}.
\]
Indeed, the last inequality follows by cubing, since
\[
\left(\frac{15}{4}\right)^3
-
\left(\frac{83871}{20000}\right)^3\frac57
=
\frac{647480614689}{11200000000000}
>0.
\]

Finally, by \eqref{eq:sharp-euclidean-sobolev-constant},
\[
3^{-1/3}S_{\mathbb R^3}
=
3^{2/3}\left(\frac{\pi}{2}\right)^{4/3}.
\]
Using $\pi>157/50$, we have
\[
\left(3^{-1/3}S_{\mathbb R^3}\right)^3
=
\frac{9\pi^4}{16}
>
\frac9{16}\left(\frac{157}{50}\right)^4
>
\left(\frac{15}{4}\right)^3,
\]
where the last inequality follows from
\[
\frac9{16}\left(\frac{157}{50}\right)^4
-
\left(\frac{15}{4}\right)^3
=
\frac{194721309}{100000000}
>0.
\]
Therefore,
\[
\frac{\|u'\|_{L^2(\mathcal G_1^3)}^2}
{\|u\|_{L^6(\mathcal G_1^3)}^2}
<
\frac{15}{4}
<
3^{-1/3}S_{\mathbb R^3}.
\]
The cutoff approximation above then gives
\[
S_{\mathcal G_1^3}
<
3^{-1/3}S_{\mathbb R^3}.
\]

We next consider $d \geq 4$.
For $a\in[0,1)$, define $u_a$ at the vertices of
$\mathcal G_1^d$ by
\[
f_a(\vv)
:=
\begin{cases}
1, & \vv=0,\\
a, & |\vv|_1=1,\\
0, & |\vv|_1\geq2,
\end{cases}
\]
and define $u_a:=\tilde{f}_a$, which extends $f_a$ linearly on every edge of $\mathcal G_1^d$.
Clearly,
\[
u_a\in H^1(\mathcal G_1^d).
\]

There are $2d$ edges joining the origin to the vertices satisfying
$|\vv|_1=1$. Moreover, every such vertex is incident to $2d-1$
edges whose other endpoint belongs to the set
$\{\vv\in \V^d_1:|\vv|_1=2\}$. Therefore,
\begin{equation}\label{eq:two-layer-test-energy}
\begin{aligned}
\|u_a'\|_{L^2(\mathcal G_1^d)}^2=
2d(1-a)^2+2d(2d-1)a^2=2d\bigl(1-2a+2da^2\bigr).
\end{aligned}
\end{equation}
 On an edge whose endpoint values are $\alpha,\beta\geq0$,
the linear interpolation satisfies
\[
\int_0^1\bigl((1-t)\alpha+t\beta\bigr)^{2^*}\,dt
=
\frac{\alpha^{2^*+1}-\beta^{2^*+1}}
{(2^*+1)(\alpha-\beta)}.
\]
Consequently,
\begin{equation}\label{eq:two-layer-test-mass}
\|u_a\|_{L^{2^*}(\mathcal G_1^d)}^{2^*}
=
\frac{2d}{{2^*}+1}
\frac{1-a^{{2^*}+1}}{1-a}
+
\frac{2d(2d-1)}{{2^*}+1}a^{2^*}
=
\frac{2d}{{2^*}+1}
\left(
\frac{1-a^{{2^*}+1}}{1-a}
+(2d-1)a^{2^*}
\right).
\end{equation}

We now consider $d=4$. In this case $2^*=4$. Taking
\[
a=\frac16,
\]
we obtain from \eqref{eq:two-layer-test-energy}
\[
\begin{aligned}
\|u_{1/6}'\|_{L^2(\mathcal G_1^4)}^2=
8\left[
\left(\frac56\right)^2
+7\left(\frac16\right)^2
\right]
=
\frac{64}{9}.
\end{aligned}
\]
Moreover, by \eqref{eq:two-layer-test-mass},
\[
\|u_{1/6}\|_{L^4(\mathcal G_1^4)}^4
=
\frac85
\left(
1+\frac16+\frac1{6^2}
+\frac1{6^3}
+\frac8{6^4}
\right)
=
\frac{781}{405}.
\]
Therefore,
\begin{equation}\label{eq:two-layer-test-d4}
\left(
\frac{
\|u_{1/6}'\|_{L^2(\mathcal G_1^4)}^2
}{
\|u_{1/6}\|_{L^4(\mathcal G_1^4)}^2
}
\right)^2
=
\frac{20480}{781}.
\end{equation}
On the other hand, \eqref{eq:sharp-euclidean-sobolev-constant}
gives
\[
S_{\mathbb R^4}
=
\frac{8\pi}{\sqrt6},
\]
and thus
\[
\left(
4^{-1/2}S_{\mathbb R^4}
\right)^2
=
\frac{8\pi^2}{3}.
\]
Since $\pi>157/50$, we have
\[
\frac{8\pi^2}{3}
>
\frac{8}{3}\left(\frac{157}{50}\right)^2
>
\frac{20480}{781}.
\]
Combining this inequality with \eqref{eq:two-layer-test-d4},
we obtain
\[
S_{\mathcal G_1^4}
<
4^{-1/2}S_{\mathbb R^4}.
\]

We next consider $d=5$. In this case,
\(2^*=\frac{10}{3}\).
Taking
\[
a=\frac18,
\]
we obtain
\[
\|u_{1/8}'\|_{L^2(\mathcal G_1^5)}^2
=
10\left[
\left(\frac78\right)^2
+9\left(\frac18\right)^2
\right]
=
\frac{145}{16}.
\]
Then,  \eqref{eq:two-layer-test-mass} yields
\[
\|u_{1/8}\|_{L^{10/3}(\mathcal G_1^5)}^{10/3}
=
\frac{30}{13}
\left(
\frac{1-2^{-13}}{1-\frac18}
+9\cdot2^{-10}
\right)
=
\frac{61905}{23296}.
\]
Therefore, 
\begin{equation}\label{eq:two-layer-test-d5}
S_{\mathcal G_1^5}\leq \frac{\|u_{1/8}'\|_{L^2(\mathcal G_1^5)}^2}{\|u_{1/8}\|_{L^{10/3}(\mathcal G_1^5)}^2}
\leq
\frac{145}{16}
\left(
\frac{23296}{61905}
\right)^{3/5}
<
\frac{145}{16}
\left(\frac25\right)^{3/5}
<
\frac{145}{16}\frac7{12}
=
\frac{1015}{192}
<
\frac{16}{3}.
\end{equation}

By \eqref{eq:sharp-euclidean-sobolev-constant},
\[
S_{\mathbb R^5}
=
\frac{15}{4}\pi^{6/5}.
\]
Since $\pi>3$, we have
\[
5^{-3/5}S_{\mathbb R^5}
>
\frac{15}{4}3^{6/5}5^{-3/5}>\frac{16}{3}.
\]
Together with \eqref{eq:two-layer-test-d5}, this gives
\[
S_{\mathcal G_1^5}
<
5^{-3/5}S_{\mathbb R^5}.
\]

Finally, let $d\geq6$ and choose
\[
a=\frac1{2d}.
\]
Then
\[
2d\bigl(1-2a+2da^2\bigr)=2d-1.
\]
Moreover, since $0<a<1$, we have
\[
\frac{1-a^{2^*+1}}{1-a}>1.
\]
It follows from
\eqref{eq:two-layer-test-energy} and
\eqref{eq:two-layer-test-mass} that
\begin{equation*}
S_{\mathcal G_1^d}
\leq
\frac{
2d\bigl(1-2a+2da^2\bigr)
}{
\left\{
\dfrac{2d}{2^*+1}
\left(
\dfrac{1-a^{2^*+1}}{1-a}
+(2d-1)a^{2^*}
\right)
\right\}^{2/2^*}
}.
\end{equation*}
Hence, 
\begin{equation}\label{eq:two-layer-upper-dge6}
S_{\mathcal G_1^d}
<
(2d-1)
\left(\frac{2^*+1}{2d}\right)^{2/2^*}
=(2d-1)
\left(
\frac{3d-2}{2d(d-2)}
\right)^{\frac{d-2}{d}}.
\end{equation}
Set
\[
A_d:=
(2d-1)
\left(
\frac{3d-2}{2d(d-2)}
\right)^{\frac{d-2}{d}}.
\]
We claim that
\[
A_d
<
d^{-\frac{d-2}{d}}S_{\mathbb R^d}
\quad\text{for every }d\geq6.
\]

Recall that
\[
S_{\mathbb R^d}
=
d(d-2)\pi
\left(
\frac{\Gamma(d/2)}{\Gamma(d)}
\right)^{2/d}.
\]
Hence
\[
\frac{A_d}
{d^{-\frac{d-2}{d}}S_{\mathbb R^d}}
=
\frac{2d-1}{\pi d^{2/d}(d-2)}
\left(
\frac{3d-2}{2d(d-2)}
\right)^{1-\frac2d}
\left(
\frac{\Gamma(d)}{\Gamma(d/2)}
\right)^{2/d}.
\]
By the duplication formula,
\[
\frac{\Gamma(d)}{\Gamma(d/2)}
=
\frac{2^{d-1}}{\sqrt{\pi}}
\Gamma\left(\frac{d+1}{2}\right).
\]
Recall that $y!=\Gamma(y+1)=y\Gamma(y)$ for $y>0$, see, e.g.,
\cite[Eq. (6.1.5)]{AbramowitzStegun}. Hence, by Stirling estimate \cite[Eq. (6.1.38)]{AbramowitzStegun},
\[
\Gamma(y)
<
\sqrt{2\pi}\,y^{y-\frac12}
\exp\left(-y+\frac1{12y}\right),
\quad y>0.
\]
applied with $y=(d+1)/2$, yields
\[
\Gamma\left(\frac{d+1}{2}\right)^{2/d}
<
(2\pi)^{1/d}\frac{d+1}{2}
\exp\left(
-\frac{d+1}{d}
+\frac1{3d(d+1)}
\right).
\]
Consequently,
\[
\frac{A_d}
{d^{-\frac{d-2}{d}}S_{\mathbb R^d}}
<
R(d),
\]
where
\[
R(x):=
\frac{
2^{1/x}(2x-1)(x+1)(3x-2)^{1-2/x}
}{
\pi x(x-2)^{2-2/x}
}
\exp\left(
-\frac{x+1}{x}
+\frac1{3x(x+1)}
\right),
\quad x\geq6.
\]

We now show that $R$ is strictly decreasing on $[6,+\infty)$.
A direct differentiation gives
\[
\begin{aligned}
\frac{d}{dx}\log R(x)
&=
\frac{
2\log\left(\dfrac{3x-2}{x-2}\right)-\log2
}{x^2}\\
&\quad-
\frac{
51x^5+18x^4-74x^3-8x^2-4x+8
}{
3x^2(x-2)(x+1)^2(2x-1)(3x-2)
}.
\end{aligned}
\]
For $x\geq6$,
\[
\frac{3x-2}{x-2}
=
3+\frac4{x-2}
\leq4,
\]
and hence
\[
2\log\left(\frac{3x-2}{x-2}\right)-\log2
\leq\log8<\frac52.
\]
On the other hand, writing
\[
P(x):=
51x^5+18x^4-74x^3-8x^2-4x+8,
\]
we have
\[
\begin{aligned}
&6P(x)
-
51(x-2)(x+1)^2(2x-1)(3x-2)\\
&\qquad=
3\left(
155x^4+124x^3-169x^2-144x+84
\right)>0
\end{aligned}
\]
for every $x\geq6$. Thus,
\[
\frac{P(x)}
{3(x-2)(x+1)^2(2x-1)(3x-2)}
>
\frac{17}{6}.
\]
It follows that
\[
x^2\frac{d}{dx}\log R(x)
<
\frac52-\frac{17}{6}
=
-\frac13,
\]
so that $R$ is strictly decreasing on $[6,+\infty)$.

At $x=6$, we have
\[
R(6)
=
\frac{77}{6\pi\sqrt2}
\exp\left(-\frac{73}{63}\right).
\]
Since $\pi>3$ and $18\sqrt2>25$,
\[
\frac{77}{6\pi\sqrt2}<\frac{77}{25}.
\]
Moreover,
\[
\exp\left(\frac{73}{63}\right)
>
1+\frac{73}{63}
+\frac12\left(\frac{73}{63}\right)^2
+\frac16\left(\frac{73}{63}\right)^3
>
\frac{77}{25}.
\]
Therefore,
\[
R(6)<1.
\]
Since $R$ is decreasing,
\[
R(d)\leq R(6)<1
\quad\text{for every }d\geq6.
\]
We conclude that
\[
A_d
<
d^{-\frac{d-2}{d}}S_{\mathbb R^d}.
\]
Combining this with \eqref{eq:two-layer-upper-dge6}, we obtain
\[
S_{\mathcal G_1^d}
<
d^{-\frac{d-2}{d}}S_{\mathbb R^d}
\quad\text{for every }d\geq6.
\]
\end{proof}
\subsection*{Conflict of interest}

The authors declare no conflict of interest.

\subsection*{Ethics approval}
 Not applicable.

\subsection*{Data Availability Statements}
Data sharing not applicable to this article as no datasets were generated or analysed during the current study.

\subsection*{Acknowledgements}
 C. Ji is partially supported by National Natural Science Foundation of China (No. 12571117).

  \end{document}